\documentclass[preprint,12pt]{elsarticle}

\usepackage{amssymb}
\usepackage{amsmath}
\usepackage{amsthm}
\usepackage{hyperref}
\usepackage{lipsum}
\usepackage{amsfonts}
\usepackage{amssymb}
\usepackage{graphicx}
\usepackage{epstopdf}
\usepackage{algpseudocode}
\usepackage{tensor}
\usepackage{verbatim}
\usepackage{mathtools}
\usepackage{cleveref}
\usepackage{framed}
\usepackage{tikz}
\usetikzlibrary{calc, shapes, angles, quotes, patterns, 
decorations.pathreplacing, cd}
\usepackage{subcaption}
\usepackage{bm}
\usepackage{esint}
\usepackage{dsfont}
\usepackage{tcolorbox}
\usepackage{afterpage}

\newlength{\leftstackrelawd}
\newlength{\leftstackrelbwd}
\def\leftstackrel#1#2{\settowidth{\leftstackrelawd}%
	{${{}^{#1}}$}\settowidth{\leftstackrelbwd}{$#2$}%
	\addtolength{\leftstackrelawd}{-\leftstackrelbwd}%
	\leavevmode\ifthenelse{\lengthtest{\leftstackrelawd>0pt}}%
	{\kern-.5\leftstackrelawd}{}\mathrel{\mathop{#2}\limits^{#1}}}

\newcommand{\bdd}[1]{ \boldsymbol{#1} }
\newcommand{\unitvec}[1]{\bdd{#1}}
\newcommand{\vertiii}[1]{{\left\vert\kern-0.25ex\left\vert\kern-0.25ex\left\vert 
#1 
		\right\vert\kern-0.25ex\right\vert\kern-0.25ex\right\vert}}

\newcommand{\hrot}{\bdd{H}(\mathrm{rot};\Omega)}
\newcommand{\hrotgamma}{\bdd{H}_{\Gamma}(\mathrm{rot};\Omega)}

\newcommand{\dual}[1]{ {#1}' }

\newcommand{\symgrad}[1]{ \varepsilon(#1) }
\newcommand{\harmonic}[1]{\bdd{\mathfrak{#1}}}

\usepackage{lineno}

\biboptions{sort&compress}

\newtheorem{theorem}{Theorem}[section]
\newtheorem{lemma}{Lemma}[section]
\newtheorem{corollary}{Corollary}[section]
\newdefinition{remark}{Remark}

\crefname{equation}{}{}
\Crefname{equation}{Equation}{Equations}
\crefname{lemma}{Lemma}{Lemmas}
\crefname{theorem}{Theorem}{Theorems}
\crefname{corollary}{Corollary}{Corollaries}
\crefname{figure}{Figure}{Figures}
\Crefname{figure}{Figure}{Figures}
\crefname{table}{Table}{Tables}
\Crefname{table}{Table}{Tables}

\numberwithin{equation}{section}

\makeatletter
\let\orgdescriptionlabel\descriptionlabel
\renewcommand*{\descriptionlabel}[1]{%
	\let\orglabel\label
	\let\label\@gobble
	\phantomsection
	\edef\@currentlabel{#1}%
	\let\label\orglabel
	\orgdescriptionlabel{#1}%
}
\makeatother

\usepackage{amsopn}

\DeclareMathOperator{\grad}{\mathbf{grad}}
\DeclareMathOperator{\dive}{div}

\DeclareMathOperator{\rot}{rot}

\DeclareMathOperator{\image}{Im}
\DeclareMathOperator{\circop}{circ}

\DeclareMathOperator*{\argmin}{arg\,min}

\newcommand{\mathbbb}[1]{\pmb{\mathbb{#1}}}

\newcommand{\discrete}[1]{\mathbb{#1}}
\newcommand{\discretev}[1]{\mathbbb{#1}}

\renewcommand{\d}[1]{\,\mathrm{d}{#1}}

\journal{Computer Methods in Applied Mechanics and Engineering}

\begin{document}

\begin{frontmatter}



\title{Do locking-free finite element schemes lock for holey Reissner-Mindlin 
	plates with mixed boundary conditions? \tnoteref{t1}}
\tnotetext[t1]{The second author acknowledge that this material is based upon 
work supported by the National Science Foundation under Award No. DMS-2201487.}

\author[1]{Mark Ainsworth\corref{cor1}}
\ead{mark\_ainsworth@brown.edu}

\author[2]{Charles Parker}
\ead{charles.parker@maths.ox.ac.uk}

\cortext[cor1]{Corresponding author}

\affiliation[1]{organization={Division of Applied Mathematics, Brown University},
	addressline={Box F, 182 George Street}, 
	city={Providence},
	postcode={02912}, 
	state={RI},
	country={USA}}

\affiliation[2]{organization={Mathematical Institute, University of Oxford},
	addressline={Andrew Wiles Building, Woodstock Road}, 
	city={Oxford},
	postcode={OX2 6GG}, 
	country={UK}}  
	
\begin{abstract}
	We revisit finite element discretizations of the Reissner-Mindlin plate
in the case of non-simply connected (holey) domains with mixed boundary 
conditions.	
Guided by the de Rham complex, we develop conditions under which schemes 
deliver locking-free, optimal rates of convergence. We naturally recover the 
typical assumptions arising for clamped, simply supported plates. More 
importantly, we also see \emph{new} conditions arise naturally from the 
presence of holes in the domain or in the case of mixed boundary conditions. 
We show that, fortunately, many of the existing popularly used schemes 
\emph{do}, in fact, satisfy all of the conditions, and thus are locking-free.
\end{abstract}


\end{frontmatter}



\section{Introduction}
\label{sec:intro}

Although the connection between differential forms, exterior calculus and de
Rham exact sequences to finite element analysis goes back to Bossavit 
\cite{Bossavit98} in the seting of approximation of Maxwell's equations, the 
recognition that finite element exterior calculus provides powerful tool for 
understanding discretisation schemes is, thanks to the efforts of 
\cite{Arnold2010,Arn18,ArnFalkWin06,Hiptmair02} now widely understood. This 
connection seems particularly powerful in the case where high order finite 
elements are considered and has led to a number of results. As far as our own 
work is concerned, these ideas have been exploited to establish that various 
families of mixed finite elements are uniformly inf-sup stable 
\cite{AinCP19StokesI,AinCP21LE} which, in turn, paves the way for high order, 
uniformly stable schemes for incompressible flow that are pointwise mass 
conserving, and to explain locking of finite element in linear elasticity 
\cite{AinCP21LE}.  On a different tack, ideas from exact sequences were used 
\cite{AinCP23KirchI} to construct $C^1$ (or $H^2$-conforming) finite element 
approximations without having to implement families of elements with built-in 
$C^1$ continuity. The same ideas were shown in \cite{AinCP24KirchII} to lead to 
computational procedures for constructing $C^1$ finite element approximations 
using elements for which it is known that there is no local basis (meaning that 
they cannot be implemented using classical finite element methodology). 

In the current work, we revisit the classical Reissner-Mindlin (RM) plate
model. The literature on finite element approximation for RM is enormous and
one may reasonably question what, if anything, new can be said about finite
element schemes for RM? While RM is of interest in its own right as a
computational model for thin elastic plates, part of the interest in the
problem arises from the fact that it constitutes a coupled system of elliptic
partial differential equations in which there are two distinct scales
corresponding to the $O(1)$ diameter of the plate and its thickness $t\ll 1$.
As such, the model may be viewed as a prototype for more complex multiscale
problems. The presence of two distinct scales leads to difficulties in the
numerical discretization in the form of locking whereby a finite element scheme
converges, if it converges at all, at a sub-optimal rate until the mesh size
$h$ is reduced to $O(t)$. Nowadays, various schemes are available that have
been proved to not lock in this way 
\cite{BoffiBrezziFortin13,Braess07,Brezzi1991MITC}. The analysis and development 
of
many of these schemes makes use of ideas from discrete de Rham sequences but
often predates the widespread adoption of ideas from finite element exterior
calculus. 

Virtually all of the existing literature on finite element methods (conforming
or otherwise) for RM considers, for simplicity, the case of a clamped, simply
connected plate
\cite{Ain02,Arn93,Arn07,Arn05,Arnold89,Arn97RM,%
	BatheBrezzi85,BatheBrezzi87,Behrens11,Bosing10,%
	Bramble98,Brezzi89,Brezzi86,Brezzi1991MITC,%
	Calo14,Carstensen11,Chen25,Chinosi06,%
	Veiga13,DiPietro22,Duran92,Falk2000,Falk2008,%
	Hansbo11,Huang24,Huoyuan18,Lovadina05,%
	Pechstein17,Peisker92,Sky23,StenbergSuri97,%
	Ye20,Zhang23}. 
While this certainly simplifies the analysis,
practical applications where one might wish to use the RM model are more likely
to involve all types of boundary condition including stress-free, simply
supported, and clamped on complicated domains with holes. The mention of domains
with holes raises warning flags for any analysis that is based on exact
sequences, discrete or otherwise. In particular, the structure of the continuous
exact sequence on domains with holes is complicated by the presence of
mathematical structures known as harmonic forms which have an entirely
different structure from those that appear in the simply connected case. These
are discussed in detail later in the manuscript but, for present purposes,
their presence means that arguments valid in the simply connected case are no
longer necessarily valid in the more general case of domains with holes.
Equally well, when mixed boundary conditions are prescribed, similar
considerations apply and again, arguments that are used in the case of pure
clamping and no longer necessarily valid in the case of mixed boundary
conditions. This raises the danger whereby schemes developed and analyzed
in the case of simply connected domains with pure clamping may not perform
as expected on practical problems. 

In the current work, we revisit finite element discretizations of the RM plate
in the case of non-simply connected domains with mixed boundary conditions.
Our approach starts from exact sequences and the de Rham diagram. Firstly, we
discuss the problem of developing schemes for RM at a rather intuitive level
guided by the de Rham sequence. One outcome of this is that one sees the
typical assumptions, which form the starting point for many existing works,
materialize naturally from intuitive arguments based on exact sequences. More
importantly, one also sees \emph{new} conditions arise naturally from the
presence of discrete harmonic forms in the case of domains with holes or in the
case of mixed boundary conditions. Respecting these new conditions is necessary
if one wants to have locking-free schemes or if one wants to be able to use the
existing schemes with confidence for such problems. These conditions are absent
from existing analyses with the result that, as mentioned above, one cannot be
sure that the schemes will be valid for practical problems. 

Having derived the new conditions following heuristic arguments guided by exact
sequences, we then provide rigorous mathematical analyses showing that the
conditions are sufficient to deliver locking-free, optimal rates of convergence
provides that the conditions are satisfied. This naturally leads to the
question of whether existing schemes satisfy the full set of conditions. The
good news, though hardly obvious at a glance, is that many of the existing
popularly used schemes \emph{do}, in fact, satisfy the full set of conditions.
While the analysis needed to arrive at this conclusion is rather involved, the
outcome should be what one expects given that the schemes have certainly been
used by practitioners to approximate problems on domains with holes and mixed
boundary conditions without reports of failure. In a sense, the fact that many
existing schemes do in fact satisfy the full set of conditions may be
attributed to good fortune. However, it is more the case that existing schemes
were constructed to respect the exact sequence property in the special case of
simply connected domains with pure clamping, which focused the choice of method 
to the extent that the resulting schemes do indeed satisfy the full set of 
conditions identified here.

The remainder of the paper is organized as follows. We introduce the RM model and 
discuss locking in \cref{sec:problem-setting}. We then use de Rham complex to 
derive conditions for locking-free schemes, state the quasi-optimal a priori 
error estimate, and illustrate the theory with numerical examples in 
\cref{sec:mitc}. In \cref{sec:demist-harm}, we further simplify the new 
conditions, which we show are satisfied by existing schemes in 
\cref{sec:mitc-examples} and by high-order methods in \cref{sec:high-order}. We 
carry out the error analysis in \cref{sec:error-analysis} and conclude in 
\cref{sec:conclusion}.

\section{Problem setting}
\label{sec:problem-setting}

Let $\Omega \subset \mathbb{R}^2$ denote the midsurface of an isotropic, 
homogeneous plate of thickness $t \in (0, 1]$. We assume that $\Omega$ is a 
polygonal domain with $H$ polygonal holes; i.e. $\Omega_1, \ldots, \Omega_{H} 
\subset \Omega_0 \subset \mathbb{R}^2$ each of which is a disjoint, simply 
connected, polygonal domain such that 
\begin{align*}
	\Omega = \Omega_0 \setminus \bigcup_{i=1}^{H} \Omega_i, \quad \bar{\Omega}_i 
	\subsetneq \Omega_0, \quad \text{and} \quad \bar{\Omega}_i \cap 
	\bar{\Omega}_j = \emptyset, \qquad 1 \leq i \neq j \leq H.
\end{align*}
The domain boundary $\Gamma := \partial \Omega$ is partitioned into disjoint sets 
$\Gamma_c$, $\Gamma_s$, and $\Gamma_f$ corresponding to the hard clamped, hard 
simply-supported, and free portion of the boundary where $\Gamma_c \cup \Gamma_s$ 
is non-empty (see e.g. \cref{fig:domain-example}). The admissible displacements 
$w$ and 
rotations $\bdd{\theta}$ of the plate belong to the following Sobolev spaces:
\begin{align}
	H^1_{\Gamma}(\Omega) &:= \{ v \in H^1(\Omega) : v|_{\Gamma_{c} \cup 
	\Gamma_s}  = 0\}, \\
	\bdd{\Theta}_{\Gamma}(\Omega) &:= \{ \bdd{\psi} \in H^1(\Omega)^2 : 
	\bdd{\psi}|_{\Gamma_c}  = \bdd{0} \text{ and } \unitvec{t} \cdot 
	\bdd{\psi}|_{\Gamma_s} = 0 \},
\end{align}
where $\unitvec{t}$ denotes the unit tangent to $\Gamma$.

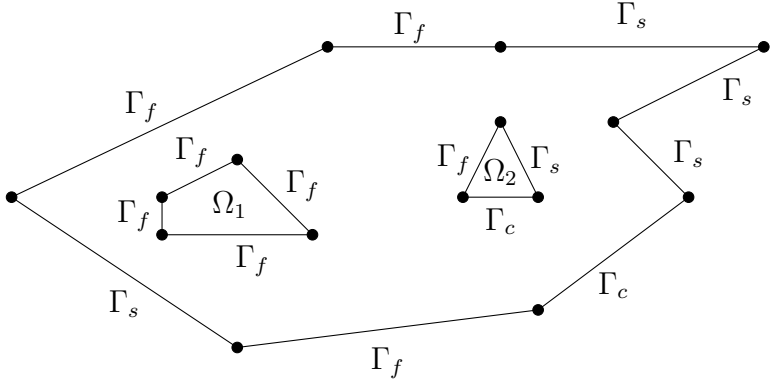
\begin{figure}[htb]
	\centering
	\begin{tikzpicture}[scale=1]
		\coordinate (A1) at (-3, 0);
		\coordinate (A2) at (1, 0.5);
		\coordinate (A3) at (3, 2);
		\coordinate (A4) at (2, 3);
		\coordinate (A5) at (4, 4);
		\coordinate (A6) at (0.5, 4);
		\coordinate (A7) at (-1.8, 4);
		\coordinate (A8) at (-6, 2);
		\coordinate (B1) at (-4, 1.5);
		\coordinate (B2) at (-2, 1.5);
		\coordinate (B3) at (-3, 2.5);
		\coordinate (B4) at (-4, 2);
		\coordinate (C1) at (0, 2);
		\coordinate (C2) at (1, 2);
		\coordinate (C3) at (0.5, 3);
		\coordinate (Omega1) at (-3.1, 1.9);
		\coordinate (Omega2) at (0.5, 7/3);

		\draw (A1) -- (A2) -- (A3) -- (A4) -- (A5) -- (A6) -- (A7) -- (A8) -- 
		(A1);
		\draw (B1) -- (B2) -- (B3) -- (B4) -- (B1);
		\draw (C1) -- (C2) -- (C3) -- (C1);
		
		\filldraw (A1) circle (2pt);
		\filldraw (A2) circle (2pt);    
		\filldraw (A3) circle (2pt);
		\filldraw (A4) circle (2pt);
		\filldraw (A5) circle (2pt);
		\filldraw (A6) circle (2pt);
		\filldraw (A7) circle (2pt);
		\filldraw (A8) circle (2pt);    
		
		\filldraw (B1) circle (2pt);
		\filldraw (B2) circle (2pt);    
		\filldraw (B3) circle (2pt);
		\filldraw (B4) circle (2pt);
		
		\filldraw (C1) circle (2pt);
		\filldraw (C2) circle (2pt);    
		\filldraw (C3) circle (2pt);
		
		\draw ($($(A1)!0.5!(A2)$) + (0, -0.1) $) 
		node[align=center,below]{$\Gamma_f$};
		
		\draw ($($(A2)!0.5!(A3)$) + (0, -0.1) $) 
		node[align=center,below]{$\Gamma_c$};
		
		\draw ($($(A3)!0.5!(A4)$) + (0.15, 0.08) $) 
		node[align=center,right]{$\Gamma_s$};
		
		\draw ($($(A4)!0.5!(A5)$) + (0.3, -0.1) $) 
		node[align=center,right]{$\Gamma_s$};
		
		\draw ($($(A5)!0.5!(A6)$) + (0.0, 0.1) $) 
		node[align=center,above]{$\Gamma_s$};
		
		\draw ($($(A6)!0.5!(A7)$) + (-0.05, 0.25) $) 
		node[align=center]{$\Gamma_f$};
		
		\draw ($($(A7)!0.5!(A8)$) + (0., 0.2) $) 
		node[align=center,left]{$\Gamma_f$};
		
		\draw ($($(A8)!0.5!(A1)$) + (0., -0.1) $) 
		node[align=center,below]{$\Gamma_s$};
		
		\draw ($($(B1)!0.5!(B2)$) + (0.2, 0.) $) 
		node[align=center,below]{$\Gamma_f$};
		
		\draw ($($(B2)!0.5!(B3)$) + (0., 0.15) $) 
		node[align=center,right]{$\Gamma_{f}$};
		
		\draw ($($(B3)!0.5!(B4)$) + (-0.1, 0.) $) 
		node[align=center,above]{$\Gamma_{f}$};
		
		\draw ($($(B4)!0.5!(B1)$) + (0., 0.) $) 
		node[align=center,left]{$\Gamma_{f}$};
		
		\draw ($($(C1)!0.5!(C2)$) + (0., 0.) $) 
		node[align=center,below]{$\Gamma_c$};
		
		\draw ($($(C2)!0.5!(C3)$) + (0., 0.) $) 
		node[align=center,right]{$\Gamma_{s}$};
		
		\draw ($($(C1)!0.5!(C3)$) + (0., 0.) $) 
		node[align=center,left]{$\Gamma_{f}$};
		
		\draw (Omega1) node[align=center]{$\Omega_1$};
		\draw (Omega2) node[align=center]{$\Omega_2$};
	\end{tikzpicture}
	\caption{Example domain.}\label{fig:domain-example}
\end{figure}

In the Reissner-Mindlin model, the deflection $w \in H^1_{\Gamma}(\Omega)$ and 
rotation $\bdd{\theta} \in \bdd{\Theta}_{\Gamma}(\Omega)$ of the plate under the 
applied loads minimize the scaled energy functional defined by
\begin{align}
	\label{eq:rm-energy}
	J_t(w, \bdd{\theta}) \leq J_t(v, \bdd{\psi}) := \frac{1}{2} a(\bdd{\psi}, 
	\bdd{\psi}) + \frac{\lambda}{2t^2} \| \bdd{\Xi}(v, \bdd{\psi})  \|^2 - F(v) - 
	G( \bdd{\psi})
\end{align}
for all $(v, \bdd{\psi}) \in H^1_{\Gamma}(\Omega) \times 
\bdd{\Theta}_{\Gamma}(\Omega)$, 
where $F$, respectively $G$, are linear functionals depending on the particular 
model that take into account the scaled transverse load, respectively the 
in-plane load and applied surface tractions on the top and bottom of the plate 
\cite{Falk2008}. The coefficient $\lambda=Ek/(2(1+\nu))$ where $E > 0$ is the 
Young's modulus, $\nu \in [0, 1/2)$ is the Poisson ratio, and 
$k > 0$ is the shear correction factor. The bilinear form  $a(\cdot, \cdot)$ is 
defined by
\begin{align*}
	a(\bdd{\theta}, \bdd{\psi}) = \frac{E}{12(1-\nu^2)} \int_{\Omega} \left[ 
	(1-\nu) \symgrad{\bdd{\theta}} : \symgrad{ \bdd{\psi} } + \nu \dive 
	\bdd{\theta} \dive \bdd{\psi} \right] \d{\bdd{x}},
\end{align*}
where $\symgrad{\bdd{\psi}} = (\grad \bdd{\psi} + (\grad \bdd{\psi})^T)/2$. The 
operator $\bdd{\Xi}$, which will play a key role throughout this manuscript, is 
defined by the rule
\begin{align}
	\label{eq:b-operator-def}
	\bdd{\Xi}(v, \bdd{\psi}) = \grad v - \bdd{\psi} \qquad \forall (v, 
	\bdd{\psi}) \in H^1_{\Gamma}(\Omega) \times \bdd{\Theta}_{\Gamma}(\Omega).
\end{align}

The minimizers of \cref{eq:rm-energy} satisfy the following variational problem: 
Find $(w, \bdd{\theta}) \in H^1_{\Gamma}(\Omega) \times 
\bdd{\Theta}_{\Gamma}(\Omega)$ such that
\begin{align}
	\label{eq:mr-variational-form}
	B_{t}( w, \bdd{\theta}; v, \bdd{\psi} ) = F(v) + G(\bdd{\psi}) \qquad \forall 
	(v, \bdd{\psi}) \in H^1_{\Gamma}(\Omega) \times \bdd{\Theta}_{\Gamma}(\Omega),
\end{align}
where
\begin{align}
	B_{t}( w, \bdd{\theta}; v, \bdd{\psi} ) := a(\bdd{\theta}, \bdd{\psi}) + 
	\lambda t^{-2}(\bdd{\Xi}(w, \bdd{\theta}), \bdd{\Xi}(v, \bdd{\psi}) ).
\end{align}
Sufficient conditions for the well-posedness of \cref{eq:mr-variational-form}, 
and consequently the existence and uniqueness of minimizers to 
\cref{eq:rm-energy}, are that the following Korn and Poincar\'{e} inequalities 
hold: There exists $C_P > 0$ and $C_K > 0$ such that
\begin{subequations}
	\label{eq:poincare-and-korn}
	\begin{alignat}{2}
		\label{eq:poincare-inequality}
		\|v\|_1 &\leq C_P \| \grad v \| \qquad & &\forall v \in 
		H^1_{\Gamma}(\Omega), \\
		\label{eq:korn-inequality}
		\| \bdd{\psi} \|_1 &\leq C_K \| \symgrad{\bdd{\psi}} \| \qquad & &\forall 
		\bdd{\psi} \in \bdd{\Theta}_{\Gamma}(\Omega).
	\end{alignat}
\end{subequations}
In particular, under these conditions, there holds
\begin{align}
	\label{eq:a-coercive}
	\| \bdd{\psi} \|_1^2 \leq C_K \|\symgrad{\bdd{\psi}}\|_2 \leq C a(\bdd{\psi}, 
	\bdd{\psi}) \qquad \forall \bdd{\psi} \in \bdd{\Theta}_{\Gamma}(\Omega).
\end{align}
Here and in what follows, $C$ is used to denote a generic positive constant that 
may depend on the domain, boundary conditions, $C_P$, $C_K$, and material 
parameters $E$ and $\nu$ and shear correction factor $k$, but unless stated 
otherwise, is independent of other quantities and, in particular, the thickness 
$t$. Thanks to Poincar\'{e}'s inequality \cref{eq:poincare-inequality}, we also 
obtain
\begin{align*}
	\|v\|_1^2 \leq C_P \|\grad v\|^2 \leq C \left(  \| \bdd{\Xi}(v, \bdd{\psi}) 
	\|^2 + \|\bdd{\psi}\|^2 \right) \leq C B_{t}(v, \bdd{\psi}; v, \bdd{\psi}) 
\end{align*}
for all $(v, \bdd{\psi}) \in H^1_{\Gamma}(\Omega) \times 
\bdd{\Theta}_{\Gamma}(\Omega)$. Thus, $B_t(\cdot,\cdot)$ is coercive:
\begin{align}
	\label{eq:bt-coercive}
	B_t(v, \bdd{\psi}; v, \bdd{\psi}) \geq \alpha \left( \| v \|_1 +  
	\|\bdd{\psi}\|_1 \right)^2 \qquad \forall (v, \bdd{\psi}) \in 
	H^1_{\Gamma}(\Omega) \times \bdd{\Theta}_{\Gamma}(\Omega),
\end{align}
where $\alpha > 0$ is independent of $t$. The bilinear form $B_t(\cdot,\cdot)$ is 
also bounded:
\begin{align}
	|B_{t}( w, \bdd{\theta}; v, \bdd{\psi} )| &\leq C \| \symgrad{\bdd{\theta}}\| 
	\| \symgrad{\bdd{\psi}}\| + t^{-2} \left( \|\grad w\| + \|\bdd{\theta}\| 
	\right) \left( \|\grad v\| + \|\bdd{\psi}\| \right)  \notag \\
	\label{eq:bt-bounded}
	&\leq C t^{-2} \left( \|w\|_1 + \|\bdd{\psi}\|_1 \right) \left( \|v\|_1 + 
	\|\bdd{\psi}\|_1 \right) 
\end{align}
for all $w, v \in H^1_{\Gamma}(\Omega)$ and $\bdd{\theta}, \bdd{\psi} \in 
\bdd{\Theta}_{\Gamma}(\Omega)$. Thus, for $F \in \dual{H^1_{\Gamma}(\Omega)}$ and 
$G \in \dual{\bdd{\Theta}_{\Gamma}(\Omega)}$,  problem 
\cref{eq:mr-variational-form} is uniquely solvable thanks to the Lax-Milgram 
theorem.

\subsection{Discretization}

Let $\{ \mathcal{T}_h \}$ denote a sequence of meshes indexed by the mesh size $h 
> 0$ with vertices $\{\mathcal{V}_h\}$ and edges $\{\mathcal{E}_h\}$. Let 
$\discrete{W}^h \subset H^1(\Omega)$ and $\discretev{V}^h \subset 
\bdd{H}^1(\Omega)$ denote finite element spaces defined on the mesh 
$\mathcal{T}_h$ and define the corresponding admissible deflections as follows:
\begin{align}
	\discrete{W}^h_{\Gamma} &:=  \discrete{W}^h \cap H^1_{\Gamma}(\Omega) \quad 
	\text{and} \quad \discretev{V}^h_{\Gamma} :=  \discretev{V}^h \cap 
	\bdd{\Theta}_{\Gamma}(\Omega).
\end{align}
Then, the following finite element approximation of \cref{eq:mr-variational-form} 
consists of seeking $(w_h, \bdd{\theta}_h) \in \discrete{W}_{\Gamma}^h \times 
\discretev{V}_{\Gamma}^h$ such that
\begin{align}
	\label{eq:mr-variational-form-fem-nored}
	B_{t}(w_h, \bdd{\theta}_h; v, \bdd{\psi})
	= F(v) + G(\bdd{\psi}) \qquad \forall (v, \bdd{\psi}) \in 
	\discrete{W}_{\Gamma}^h \times \discretev{V}_{\Gamma}^h.
\end{align}
As usual, the conforming finite element method is automatically well-posed, and 
C\'{e}a's Lemma \cite[Theorem
2.4.1]{Cia02} implies the resulting error satisfies
\begin{align}
	\label{eq:error-est-quasi-bad}
	\| w - w_h\|_1 + \| \bdd{\theta} - \bdd{\theta}_h\|_1 &\leq \frac{C}{t} 
	\left( \inf_{v \in \discrete{W}^h_{\Gamma}} \|w - v\|_1 +  \inf_{\bdd{\psi} 
	\in \discretev{V}^h_{\Gamma}} \|\bdd{\theta} - \bdd{\psi} \|_1 \right).
\end{align}
The presence of the factor $1/t \gg 1$ raises the concern that the error of the 
scheme \cref{eq:mr-variational-form-fem-nored} may be significantly larger than 
the best approximation error appearing in parentheses on the right hand side of 
\cref{eq:error-est-quasi-bad}. This fear proves to be well-founded, and it is 
well-known \cite{Bathe96,BoffiBrezziFortin13,Braess07} that (unless the spaces 
$\discrete{W}^h$ and 
$\discretev{V}^h$ satisfy additional assumptions) the method 
\cref{eq:mr-variational-form-fem-nored} exhibits shear locking where, if $t \ll 
h$, then the error does not reduce as the mesh is refined or if it does, then it 
converges to zero at a suboptimal rate.

\subsection{Heuristic explanation of locking}
\label{sec:locking-explain}

Shear locking arises from the interplay between the variational formulation of 
the problem and the nature of the finite dimensional spaces 
$\discrete{W}^h_{\Gamma}$ and $\discretev{V}^h_{\Gamma}$. More specifically, the 
fact that the solution $(w_h, \bdd{\theta}_h)$ of 
\cref{eq:mr-variational-form-fem-nored} minimizes the energy \cref{eq:rm-energy}
\begin{align*}
	(w_h, \bdd{\theta}_h) = \argmin_{(v, \bdd{\psi}) \in \discrete{W}^h_{\Gamma} 
	\times \discretev{V}^h_{\Gamma}} \frac{1}{2} a(\bdd{\psi}, \bdd{\psi}) + 
	\frac{\lambda}{2t^2} \| \bdd{\Xi}(v, \bdd{\psi})  \|^2 - F(v) - G( \bdd{\psi})
\end{align*}
means that if $t \ll 1$, then the first two terms balance only if  
$\|\bdd{\Xi}(w_h, \bdd{\theta}_h)\| = \mathcal{O}(t)$. In turn, this means that, 
for all practical purposes, when $t \ll 1$ the minimizer must satisfy 
$\bdd{\Xi}(w_h, \bdd{\theta}_h) = \bdd{0}$ or $\bdd{\theta}_h = \grad w_h$. The 
fact that $\grad w_h$ may be discontinuous while $\bdd{\theta}_h$ is continuous 
results in the minimizers becoming more occupied with maintaining compatibility 
than with approximating the true deflections. Consequently, the resulting errors 
are larger than what would be attained from the pure approximation error 
appearing in \cref{eq:error-est-quasi-bad}.

This effect can be exhibited analytically by decomposing the solution into two 
components -- one independent of $t$ and a $t$-dependent remainder as follows. 
Let $(z_h, \bdd{\zeta}_h) \in \discrete{W}^h_{\Gamma} \times 
\discretev{V}^h_{\Gamma}$ denote the solution to the ``$t=0$ limit" problem
\begin{align}
	\label{eq:discrete-in-kernel-limit}
	(z_h, \bdd{\zeta}_h) = \argmin_{(v, \bdd{\psi}) \in \ker \bdd{\Xi}^h} 
	\frac{1}{2} a(\bdd{\psi}, \bdd{\psi}) - F(v) - G( \bdd{\psi}),
\end{align}
where $\ker \bdd{\Xi}^h$ is the restriction of the kernel of $\bdd{\Xi}$ to 
$\discrete{W}^h_{\Gamma} \times \discretev{V}^h_{\Gamma}$; i.e.
\begin{align*}
	\ker \bdd{\Xi} &:= \{ (v, \bdd{\psi}) \in H^1_{\Gamma}(\Omega) \times 
	\bdd{\Theta}_{\Gamma}(\Omega) : \bdd{\Xi}(v,  \bdd{\psi}) \equiv \bdd{0} \},
\end{align*}
and
\begin{align*}
	\ker \bdd{\Xi}^h &:= \ker \bdd{\Xi} \cap \discrete{W}^h_{\Gamma} \times 
	\discretev{V}^h_{\Gamma}.
\end{align*}
Then, $(w_h, \bdd{\theta}_h) = (u_h, \bdd{\phi}_h) + (z_h, \bdd{\zeta}_h)$, where
\begin{align*}
	(u_h, \bdd{\phi}_h) = \argmin_{(v, \bdd{\psi}) \in\discrete{W}^h_{\Gamma} 
	\times \discretev{V}^h_{\Gamma}} \frac{1}{2} a(\bdd{\psi}, \bdd{\psi}) + 
	\frac{\lambda}{2t^2} \| \bdd{\Xi}(v, \bdd{\psi})  \|^2 + a(\bdd{\zeta}_h, 
	\bdd{\psi}) - F(v) - G( \bdd{\psi}), 
\end{align*}
and so $\|\bdd{\Xi}(u_h, \bdd{\phi}_h)\| = \|\bdd{\Xi}(w_h, \bdd{\theta}_h)\| = 
\mathcal{O}(t)$.

The same arguments applied to the true solution $(w, \bdd{\theta})$ of 
\cref{eq:mr-variational-form} show that
\begin{align*}
	(w, \bdd{\theta}) = (u, \bdd{\phi}) + (z, \bdd{\zeta}),
\end{align*}
where $(u, \bdd{\phi}) \in H^1_{\Gamma}(\Omega) \times 
\bdd{\Theta}_{\Gamma}(\Omega)$ satisfies $\|\bdd{\Xi}(u, \bdd{\phi})\| = 
\mathcal{O}(t)$ and $(z, \bdd{\zeta}) \in \ker \bdd{\Xi}$ is independent of $t$. 
Thus, if the finite element solution is to provide uniform in $t$ 
approximability, then the space $\ker \bdd{\Xi}^h$ should provide a quasi-optimal 
approximation to $(z, \bdd{\zeta}) \in \ker \bdd{\Xi}$; i.e. 
\begin{multline}
	\label{eq:error-est-ker}
	\inf_{ (z_h, \bdd{\zeta}_h) \in \ker \bdd{\Xi}^h }  \left( \| z - z_h\|_1 + 
	\| \bdd{\zeta} - \bdd{\zeta}_h\|_1 \right) \\
	\leq C \inf_{ (v_h, \bdd{\psi}_h) \in \discrete{W}_{\Gamma}^{h} \times 
	\discretev{V}_{\Gamma}^h } \left(  \| z - v_h\|_1 + \| \bdd{\zeta} - 
	\bdd{\psi}_h\|_1 \right).
\end{multline}
This property means that the subspace $\ker \bdd{\Xi}^h$ is able to approximate 
functions belonging to $\ker \bdd{\Xi}$ to the same order as the full space 
$\discrete{W}_{\Gamma}^{h} \times \discretev{V}_{\Gamma}^h$.

Unfortunately, estimate \cref{eq:error-est-ker} fails to hold for many common 
choices of finite element spaces. For instance \cite{Falk2008}, suppose that 
$\discrete{W}^h$ and $\discretev{V}^h$ consist of continuous, piecewise linear 
polynomials. Then, the condition $\bdd{\Xi}(z_h, \bdd{\zeta}_h) \equiv \bdd{0}$ 
means that $\bdd{\zeta}_h$ is piecewise constant, and hence constant, since 
$\bdd{\zeta}_h$ is continuous. However, Korn's inequality 
\cref{eq:korn-inequality} and Poincar\'{e}'s inequality 
\cref{eq:poincare-inequality} then show that $\bdd{\zeta}_h \equiv \bdd{0}$ and 
$z_h \equiv 0$. Thus, $\bdd{\Xi}(z_h, \bdd{\zeta}_h) \equiv \bdd{0}$ if and only 
if both $z_h$ and $\bdd{\zeta}_h$ vanish meaning that \cref{eq:error-est-ker} 
fails. In the limit $t \to 0$, these results mean that the resulting finite 
element approximation $(w_h, \bdd{\theta}_h)$ to 
\cref{eq:mr-variational-form-fem-nored} 
also satisfies $w_h \equiv 0$ and $\bdd{\theta}_h \equiv \bdd{0}$ which 
corresponds to the plate remaining undeflected or \textit{locked}. The same 
issues occur with higher order elements to a lesser extent but still manifest 
themselves as suboptimal rates of convergence.

\subsection{Alleviating locking with reduction operators}
\label{sec:intro-reduct}

The standard approach to alleviate shear locking is to relax the constraint 
$\bdd{\Xi}(z_h, \bdd{\zeta}_h) \equiv \bdd{0}$.  One way in which to do this is 
to introduce a linear \textit{reduction} operator $\bdd{R}$ and define a new 
operator $\bdd{\Xi}_{\bdd{R}}$ formally by the rule
\begin{align}
	\label{eq:intro-reduction-operator}
	\bdd{\Xi}_{\bdd{R}}(v, \bdd{\psi}) = \bdd{R}(\grad v -  \bdd{\psi}) \qquad 
	\forall (v, \bdd{\psi}) \in H^1_{\Gamma}(\Omega) \times 
	\bdd{\Theta}_{\Gamma}(\Omega).
\end{align}
Replacing $\bdd{\Xi}$ in the variational problem \cref{eq:rm-energy} with 
$\bdd{\Xi}_{\bdd{R}}$ results in a modified energy 
\begin{align}
	\label{eq:rm-energy-reduced}
	J_{t, \bdd{R}}(w, \bdd{\theta}) := \frac{1}{2} a(\bdd{\theta}, \bdd{\theta}) 
	+ \frac{\lambda}{2t^2} \| \bdd{\Xi}_{\bdd{R}}(w, \bdd{\theta} )\|^2 - F(w) - 
	G(\bdd{\theta})
\end{align}
in which the Kirchhoff-Love term $\| \bdd{\Xi}_{\bdd{R}}(w, \bdd{\theta} )\|^2$ 
is weakened. The goal is to balance the choice of $\bdd{R}$ so that the set 
\begin{align}
	\ker \bdd{\Xi}_{\bdd{R}}^h := \{ (v, \bdd{\psi}) \in \discrete{W}^h_{\Gamma} 
	\times \discretev{V}^h_{\Gamma} : \bdd{\Xi}_{\bdd{R}}(v,  \bdd{\psi}) \equiv 
	\bdd{0} \}
\end{align}
has optimal approximation properties while not compromising the integrity of the 
underlying physical model through replacing $J_t$ with the modified energy $J_{t, 
\bdd{R}}$. The corresponding variational problem reads: Find $(w_h, 
\bdd{\theta}_h) \in \discrete{W}_{\Gamma}^h \times \discretev{V}_{\Gamma}^h$ such 
that
\begin{align}
	\label{eq:mr-variational-form-fem}
	B_{t, \bdd{R}}(w_h, \bdd{\theta}_h; v, \bdd{\psi})
	= F(v) + G(\bdd{\psi}) \qquad \forall (v, \bdd{\psi}) \in 
	\discrete{W}_{\Gamma}^h \times \discretev{V}_{\Gamma}^h,
\end{align}
where the reduced bilinear form is given by
\begin{align}
	B_{t, \bdd{R}}(w, \bdd{\theta}; v, \bdd{\psi}) &:= 	a(\bdd{\theta}, 
	\bdd{\psi}) + \lambda t^{-2} ( \bdd{\Xi}_{\bdd{R}}( w, \bdd{\theta}),  
	\bdd{\Xi}_{\bdd{R}}(v, \bdd{\psi})).
\end{align}

\section{Application of the de Rham complex to Reissner-Mindlin discretization}
\label{sec:mitc}

In this section, we show that the de Rham complex provides a systematic approach 
to the construction of reduction operators $\bdd{R}$. We start by collecting some 
key properties of the de Rham complex that will be used in the sequel.

\subsection{Brief r\'{e}sum\'{e} on the de Rham complex}

The rotation operator $\rot: \hrot \to L^2(\Omega)$ is defined by the rule $\rot 
\bdd{\theta} := \partial_x \theta_2 - \partial_y \theta_1$, where the space
\begin{align*}
	\hrot := \{ \bdd{\gamma} \in \bdd{L}^2(\Omega) : \rot \bdd{\gamma} \in 
	L^2(\Omega) \}
\end{align*}
is equipped with inner product and norm defined by
\begin{align*}
	(\bdd{\gamma}, \bdd{\eta})_{\rot} := (\bdd{\gamma}, \bdd{\eta}) + (\rot 
	\bdd{\gamma}, \rot \bdd{\eta}) \quad \text{and} \quad 
	\|\bdd{\gamma}\|_{\rot}^2 := (\bdd{\gamma}, \bdd{\gamma})_{\rot} \qquad 
	\forall \bdd{\gamma}, \bdd{\eta} \in \hrot.
\end{align*}
Vector fields in $\hrot$ have a well-defined tangential component on the boundary 
in the sense that the tangential trace operator $\hrot \ni \bdd{\gamma} \to 
\unitvec{t} \cdot \bdd{\gamma}|_{\Gamma} \in H^{-1/2}(\Gamma)$ defines a 
continuous mapping \cite[p. 27 Theorem 2.5]{GiraultRaviart86}. We may therefore 
define the subspace
\begin{align*}
	\hrotgamma := \left\{ \bdd{\gamma} \in \hrot : \int_{\Gamma} (\unitvec{t} 
	\cdot 
	\bdd{\gamma}) v \d{s}  = 0 \quad \forall v \in H^1(\Omega) : v|_{\Gamma_f} 
	\equiv 0 \right\}
\end{align*}
consisting of vector fields with vanishing tangential trace on $\Gamma \setminus 
\Gamma_f$, where we use $\int_{\Gamma}$ to also denote the $H^{-1/2}(\Gamma) 
\times H^{1/2}(\Gamma)$ duality pairing. Finally, we define
\begin{align}
	\label{eq:l2gamma-def}
	L^2_{\Gamma}(\Omega) := \{ q \in L^2(\Omega) : (q, 1) = 0 \text{ if } 
	|\Gamma_f| = 0 \}.
\end{align}

These spaces, together with $H^1_{\Gamma}(\Omega)$, form a \textit{complex}
\begin{equation}
	\label{eq:de-rham-bcs}
	\begin{tikzcd}
		0 \arrow{r} & H^1_{\Gamma}(\Omega)  \arrow{r}{\grad} & \hrotgamma 
		\arrow{r}{\rot} & L^2_{\Gamma}(\Omega) \arrow{r}{} & 0
	\end{tikzcd}
\end{equation}	
in which the image of each operator is contained in the kernel of the next 
operator in the sequence; e.g. $\grad H^1_{\Gamma}(\Omega) \subset \hrotgamma$, 
$\rot \grad v \equiv 0$ for all $v \in H^1_{\Gamma}(\Omega)$, and the operator 
$\rot : \hrotgamma \to L^2_{\Gamma}(\Omega)$ is surjective \cite[Propositions 
51.3 \& 51.6]{ErnGuerII21}. Moreover, in the special case of a simply connected 
plate with $\Gamma_f$ connected, the image of each operator is precisely the 
kernel of the following operator (see e.g. \cite[p. 1111-1112 Example 9]{Licht17} 
or \cref{lem:harmonic-complex-cont} below), and the complex \cref{eq:de-rham-bcs} 
is then said to be \textit{exact}. However, if the domain is not simply connected 
or if $\Gamma_f $ is not connected, then $\grad H^1_{\Gamma}(\Omega)$ is a proper 
subspace of $\ker \rot \hrotgamma$ and, as a result, the space 
\begin{align}
	\label{eq:harmonic-forms-def}
	\harmonic{H}_{\Gamma}(\Omega) := \{ \harmonic{h} \in \hrotgamma : \rot 
	\harmonic{h} \equiv 0  \text{ and }  (\harmonic{h}, \grad w) = 0 \ \forall w 
	\in H^1_{\Gamma}(\Omega) \}
\end{align}
is nontrivial. This means that there exist non-zero vector fields, known as 
\textit{harmonic forms}, that have vanishing rotation and which are orthogonal to 
$\grad H^1_{\Gamma}(\Omega)$. 

In order to see the relevance of the de Rham sequence to the discretization of 
the Reissner-Mindlin model, we consider spaces $\discrete{W}_{\Gamma}^h \subset 
H_{\Gamma}^1(\Omega)$, $\discretev{U}_{\Gamma}^h \subset \hrotgamma$, and 
$\discrete{Q}^h_{\Gamma} \subset L^2_{\Gamma}(\Omega)$ for which the following 
discrete analogue of \cref{eq:de-rham-bcs} holds: 
\begin{equation}
	\label{mot:eq:wh-gammah-qh-sequence}
	\begin{tikzcd}
		0 \arrow{r} & \discrete{W}_{\Gamma}^h \arrow{r}{\grad} & 
		\discretev{U}^h_{\Gamma} \arrow{r}{\rot} & \discrete{Q}^h_{\Gamma} 
		\arrow{r} & 0.
	\end{tikzcd}
\end{equation}
The discrete analogue of the harmonic forms is defined to be
\begin{align}
	\label{eq:harmonic-forms-discrete-def}
	\harmonic{H}^h_{\Gamma} := \{ \harmonic{h} \in \discretev{U}^h_{\Gamma} : 
	\rot \harmonic{h} \equiv 0  \text{ and }  (\harmonic{h}, \grad w) = 0 \ 
	\forall w \in \discrete{W}^h_{\Gamma} \},
\end{align}
and the discrete sequence \cref{mot:eq:wh-gammah-qh-sequence} is exact if and 
only if $\harmonic{H}^h_{\Gamma} = \{ \bdd{0} \}$ and $\rot 
\discretev{U}^h_{\Gamma} = \discrete{Q}^h_{\Gamma}$. 

The sequences \cref{eq:de-rham-bcs} and \cref{mot:eq:wh-gammah-qh-sequence} can 
be often related to one another via bounded, linear projection operators
\begin{align*}
	\Pi_{\discrete{W}} : H^1_{\Gamma}(\Omega) \to \discrete{W}^h_{\Gamma}, \quad 
	\bdd{\Pi}_{\discretev{U}} : \hrotgamma \to \discretev{U}^h_{\Gamma}, \quad 
	\text{and} \quad \Pi_{\discrete{Q}} : L^2_{\Gamma}(\Omega) \to 
	\discrete{Q}^h_{\Gamma}
\end{align*}
such that the following diagram commutes:
\begin{equation}
	\label{eq:de-rham-commuting}
	\begin{tikzcd}
		H^1_{\Gamma}(\Omega) \arrow{r}{\grad} \arrow{d}{\Pi_{\discrete{W}}} &
		\hrotgamma \arrow{r}{\rot} \arrow{d}{\bdd{\Pi}_{\discretev{U}}} & 
		L^2_{\Gamma}(\Omega) \arrow{d}{\Pi_{\discrete{Q}}} \\
		\discrete{W}^h_{\Gamma} \arrow{r}{\grad} & \discretev{U}_{\Gamma}^h 
		\arrow{r}{\rot} & \discrete{Q}_{\Gamma}^h,
	\end{tikzcd}
\end{equation}
i.e. $\grad \Pi_{\discrete{W}} = \bdd{\Pi}_{\discretev{U}} \grad$ and $\rot 
\bdd{\Pi}_{\discretev{U}} = \Pi_{\discrete{Q}} \rot$.  Many examples of spaces 
and operators that satisfy \cref{eq:de-rham-commuting} are known in the 
literature (see e.g. \cite{Arn18,BoffiBrezziFortin13,Dem07}).

\subsection{Application to Reissner-Mindlin discretization}
\label{sec:appl-de-rham-plate}

The Galerkin discretization of the Reissner-Mindlin model requires the selection 
of subspaces $\discrete{W}^h_{\Gamma} \subset H^1_{\Gamma}(\Omega)$ and 
$\discretev{V}^h_{\Gamma} \subset \bdd{\Theta}_{\Gamma}(\Omega) \subset 
\hrotgamma$ with which to approximate the displacement $w$ and rotation 
$\bdd{\theta}$. We have already seen that the approximation properties of the 
subspace
\begin{align*}
	\{ (v_h, \bdd{\psi}_h) \in \discrete{W}^h_{\Gamma} \times 
	\discretev{V}^h_{\Gamma} : \bdd{\Xi}(v_h, \bdd{\psi}_h) \equiv \bdd{0} \}	
\end{align*}
is responsible for locking. Re-examining this constraint in light of the discrete 
sequence \cref{mot:eq:wh-gammah-qh-sequence} shows that the equality 
$\bdd{\Xi}(v_h, \bdd{\psi}_h) \equiv \bdd{0} \iff \bdd{\psi}_h = \grad v_h$ 
should be interpreted by viewing $\discretev{V}_{\Gamma}^h$ as a subspace of 
$\hrotgamma$ as distinct from the $H^1$-conforming space 
$\bdd{\Theta}_{\Gamma}(\Omega)$. In other words, the constraint $\bdd{\Xi}(v, 
\bdd{\psi}) \equiv \bdd{0}$ at the continuous level should be viewed at the 
discrete level through the lens of the projection operator 
$\bdd{\Pi}_{\discretev{U}}$ that appears in the commuting diagram 
\cref{eq:de-rham-commuting}, and the constraint $\bdd{\Xi}(v, \bdd{\psi}) = \grad 
v - \bdd{\psi} \equiv \bdd{0}$ should itself be discretized as 
$\bdd{\Pi}_{\discretev{U}}(\grad v_h - \bdd{\psi}_h) \equiv \bdd{0}$. The fact 
that $\discretev{V}^h_{\Gamma}$ is not necessarily a subspace of 
$\discretev{U}^h_{\Gamma}$ means that the condition 
\begin{align*}
	\bdd{0} \equiv \bdd{\Pi}_{\discretev{U}}(\grad v_h - \bdd{\psi}_h) = \grad 
	v_h - \bdd{\Pi}_{\discretev{U}} \bdd{\psi}_h
\end{align*}
is more relaxed than the pointwise condition $\grad v_h - \bdd{\psi}_h \equiv 
\bdd{0}$. This suggests defining an operator $\bdd{\Xi}_{\discretev{U}}$ by the 
rule $\bdd{\Xi}_{\discretev{U}}(v_h, \bdd{\psi}_h) = 
\bdd{\Pi}_{\discretev{U}}(\grad v_h - \bdd{\psi}_h)$ where the commuting 
projection $\bdd{\Pi}_{\discretev{U}}$ naturally takes the role of the reduction 
operator $\bdd{R}$ which we seek.

The only snag with this approach stems from the fact that in order to implement 
\cref{eq:mr-variational-form-fem} using standard finite element sub-assembly 
techniques, one needs to explicitly compute the action of $\bdd{R}$ element by 
element. Unfortunately, the projection operator $\bdd{\Pi}_{\discretev{U}}$, in 
most cases, cannot be locally defined owing to the fact that the traces of 
functions in $\hrotgamma$ are not well-defined on individual edges which 
precludes computing the action of $\bdd{R}$ separately on each element. 

However, a careful examination \cref{eq:mr-variational-form-fem} reveals that the 
action of the reduction operator is in fact only required for functions belonging 
to the smoother space $\discretev{V}_{\Gamma}^h + \grad \discrete{W}_{\Gamma}^h$ 
rather than $\hrotgamma$. Such functions have more regularity than just 
$\hrotgamma$ which means one can define reduction operators in terms of the 
traces on individual edges. As such, the action of 
the operator can, at least in principle, be computed element by element. 
Fortunately, for most choices of spaces forming a complex 
\cref{mot:eq:wh-gammah-qh-sequence}, there exist appropriate locally-defined 
interpolation operators (defined on the smoother spaces). In particular, there 
typically exist bounded linear projection operators
\begin{align}
	\label{mot:interpolation-operator}
	T : H^2_{\Gamma}(\Omega) + \discrete{W}^h_{\Gamma} \to 
	\discrete{W}^h_{\Gamma} \quad \text{and} \quad \bdd{R} : 
	\bdd{\Theta}_{\Gamma}(\Omega) + \discretev{U}^h_{\Gamma} \to 
	\discretev{U}^h_{\Gamma}
\end{align}
such that the following diagram commutes:
\begin{equation}
	\label{eq:mot:mitc:red-commute}
	\begin{tikzcd}
		H^2_{\Gamma}(\Omega) \arrow{r}{\grad} \arrow{d}{T} &
		\bdd{\Theta}_{\Gamma}(\Omega) \arrow{r}{\rot} \arrow{d}{\bdd{R}} & 
		L^2_{\Gamma}(\Omega) \arrow{d}{P} \\
		\discrete{W}^h_{\Gamma} \arrow{r}{\grad} & \discretev{U}_{\Gamma}^h 
		\arrow{r}{\rot} & \discrete{Q}_{\Gamma}^h,
	\end{tikzcd}
\end{equation}
where $P : L^2(\Omega) \to \discrete{Q}_{\Gamma}^h$ is the $L^2$-projection 
operator
\begin{align}
	\label{mot:eq:l2-projection-q}
	( P r, q) = (r, q) \qquad \forall q \in \discrete{Q}^h_{\Gamma}, \ \forall r 
	\in L^2(\Omega)
\end{align}
and
\begin{align*}
	H^2_{\Gamma}(\Omega) := \{ w \in H^2(\Omega) : w|_{\Gamma_{c} \cup \Gamma_s} 
	= 0 \text{ and } \partial_n w|_{\Gamma_c} = 0 \}.
\end{align*}
Strictly speaking, the operator $\bdd{\Xi}_{\bdd{R}}$ defined in 
\cref{eq:intro-reduction-operator} is not valid on all of $H^1_{\Gamma}(\Omega) 
\times \bdd{\Theta}_{\Gamma}(\Omega)$. However, thanks to the discrete de Rham 
complex property together with the projection property in 
\cref{mot:mitc:red-commute}, there holds
\begin{align*}
	\bdd{R} \grad w  = \grad w \qquad \forall w \in \discrete{W}^h_{\Gamma}.
\end{align*}
Consequently, without affecting the discrete scheme 
\cref{eq:mr-variational-form-fem} or the space $\ker \bdd{\Xi}_{\bdd{R}}^h$, we 
may redefine the operator $\bdd{\Xi}_{\bdd{R}}$ as follows:
\begin{align}
	\label{eq:reduction-operator}
	\bdd{\Xi}_{\bdd{R}}(v, \bdd{\psi}) := \grad v - \bdd{R} \bdd{\psi} \qquad 
	\forall (v, \bdd{\psi}) \in H^1_{\Gamma}(\Omega) \times 
	\bdd{\Theta}_{\Gamma}(\Omega).
\end{align}
The operator $\bdd{R}$ is also typically bounded on 
$\bdd{\Theta}_{\Gamma}(\Omega)$: There exits $M_{\bdd{R}} > 0$ independent of $h$ 
such that
\begin{align}
	\label{eq:reduction-bounded}
	\| \bdd{R} \bdd{\psi}\|_{\rot} \leq M_{\bdd{R}}  \| \bdd{\psi}\|_{1}  \qquad 
	\forall \bdd{\psi} \in \bdd{\Theta}_{\Gamma}(\Omega).
\end{align}

In summary, we have the following conditions: 
\begin{align}
	\label{mot:mitc:red-commute}
	&\text{\parbox{0.8\linewidth}{The spaces $\discrete{W}_{\Gamma}^h \subset 
	H_{\Gamma}^1(\Omega)$, $\discretev{U}_{\Gamma}^h \subset \hrotgamma$ and 
	$\discrete{Q}^h_{\Gamma} \subset L^2_{\Gamma}(\Omega)$ and linear projection 
	operators $\bdd{R} : \bdd{\Theta}_{\Gamma}(\Omega) + \discretev{U}^h_{\Gamma} 
	\to \discretev{U}^h_{\Gamma}$ and $T : H^2_{\Gamma}(\Omega) + 
	\discrete{W}_{\Gamma}^{h} \to \discrete{W}^h_{\Gamma}$ are chosen so that 
	\cref{eq:mot:mitc:red-commute} commutes and the bottom row forms a discrete 
	complex.} } \\
	\label{mot:mitc:red-bounded}
	&\text{\parbox{0.8\linewidth}{There exists $M_{\bdd{R}} > 0$ independent of 
	$h$ such that \cref{eq:reduction-bounded} holds.}} 
\end{align}

\begin{remark}
	\label{remark:consequences-sequence-r}
	The identity $\rot \discretev{U}^h_{\Gamma} =  \discrete{Q}^h_{\Gamma}$ is a 
	direct consequence of the commuting diagram property 
	\cref{mot:mitc:red-commute}: For $q \in \discrete{Q}^h_{\Gamma}$, there 
	exists $\bdd{\theta} \in \bdd{\Theta}_{\Gamma}(\Omega)$ such that $\rot 
	\bdd{\theta} = q$ \cite[Lemma A.3]{AinCP23KirchI}. As a result, a vector 
	field 
	of the form $\bdd{\gamma} := \bdd{R} \bdd{\theta} \in 
	\discretev{U}^h_{\Gamma}$ satisfies 
	\begin{align*}
		\rot \bdd{R} \bdd{\theta} = P \rot \bdd{\theta} = P q = q.
	\end{align*}
\end{remark}

\subsection{Choosing the rotation space $\discretev{V}^h_{\Gamma}$}

Earlier we argued that, in order to avoid locking, the space $\ker 
\bdd{\Xi}_{\bdd{R}}^h$ should satisfy the following quasi-optimal approximation 
property: 
\begin{multline}
	\label{eq:error-est-ker-red}
	\inf_{ (z_h, \bdd{\zeta}_h) \in \ker \bdd{\Xi}_{\bdd{R}}^h  } \left(   \| z - 
	z_h\|_1 + \| \bdd{\zeta} - \bdd{\zeta}_h\|_1 \right) \\
	\leq C \inf_{ (v_h, \bdd{\psi}_h) \in \discrete{W}_{\Gamma}^{h} \times 
	\discretev{V}_{\Gamma}^h } \left( \| z - v_h\|_1 + \| \bdd{\zeta} - 
	\bdd{\psi}_h\|_1 \right),
\end{multline}
for all $(z, \bdd{\zeta}) \in \ker \bdd{\Xi}$. It is not always easy to see 
directly whether \cref{eq:error-est-ker-red} is satisfied, so we develop an 
alternative criterion that can be used to ensure this is the case. Firstly, 
notice that if the inf-sup condition
\begin{align}
	\label{eq:invert-xi-inf-sup-u}
	\inf_{ \substack{\bdd{\eta} \in \discretev{U}^h_{\Gamma} \\ \bdd{\eta} \neq 
			\bdd{0} } } \sup_{ \substack{ (v, \bdd{\psi}) \in 
			\discrete{W}^h_{\Gamma} 
			\times \discretev{V}^h_{\Gamma} \\ (v, \bdd{\psi}) \neq (0, \bdd{0}) 
			} } 
	\frac{ ( \bdd{\eta}, \bdd{\Xi}_{\bdd{R}}(v, \bdd{\psi}) )_{\rot} }{ (\|v\|_1 
		+ \|\bdd{\psi}\|_1) \|\bdd{\eta}\|_{\rot} } \geq \beta_{\bdd{R}}
\end{align}
is satisfied with $\beta_{\bdd{R}} > 0$ independent of $h$, then the following 
analogue of \cref{eq:error-est-ker-red} holds for all $(z, \bdd{\zeta}) \in \ker 
\bdd{\Xi}_{\bdd{R}}$:
\begin{multline}
	\label{eq:ker-opt-approx}
	\inf_{ (z_h, \bdd{\zeta}_h) \in \ker \bdd{\Xi}_{\bdd{R}}^h} \left( \| z - z_h 
	\|_1 + \|\bdd{\zeta} - \bdd{\zeta}_h\|_1 \right) \\ 
	\leq \frac{C M_{\bdd{R}} }{\beta_{\bdd{R}}} \inf_{ (v_h, \bdd{\psi}_h) \in 
	\discrete{W}^h_{\Gamma} \times \discretev{V}^h_{\Gamma} } \left( \|z - 
	v_h\|_1 + \|\bdd{\zeta} - \bdd{\psi}_h \|_1 \right). 
\end{multline}
Here, we used \cite[Proposition 5.1.3]{BoffiBrezziFortin13} along with the 
estimate
\begin{align*}
	|(\bdd{\eta}, \bdd{\Xi}_{\bdd{R}}(v, \bdd{\psi}))_{\rot}| \leq 
	\|\bdd{\eta}\|_{\rot} \left( \| \grad v \| + \| \bdd{R} \bdd{\psi} \|_{\rot} 
	\right) \leq C M_{\bdd{R}} \|\bdd{\eta} \|_{\rot} \left( \|v\|_1 + 
	\|\bdd{\psi}\| \right)	
\end{align*}
which holds for all $\bdd{\eta} \in \hrotgamma$, $v \in H^1_{\Gamma}(\Omega)$, 
and $\bdd{\psi} \in \bdd{\Theta}_{\Gamma}(\Omega)$. As it stands, 
\cref{eq:ker-opt-approx} applies to the approximation of $(z, \bdd{\zeta}) \in 
\ker \bdd{\Xi}_{\bdd{R}}$, whereas \cref{eq:error-est-ker-red} relates to $(z, 
\bdd{\zeta}) \in \ker \bdd{\Xi}$. Nevertheless \cref{eq:ker-opt-approx} implies 
that a variant of \cref{eq:error-est-ker-red} holds provided the right hand side 
is augmented with an additional term accounting for the error incurred through 
using the reduction operator in place of the identity:
\begin{lemma}
	\label{lem:ker-opt-approx-true-kerxi}
	Suppose that \cref{mot:mitc:red-commute}, \cref{mot:mitc:red-bounded}, and 
	the 
	inf-sup condition \cref{eq:invert-xi-inf-sup-u} hold. Then, for all $(z, 
	\bdd{\zeta}) \in \ker \bdd{\Xi}$, there holds
	\begin{multline}
		\label{eq:ker-opt-approx-true-kerxi}
		\inf_{ (z_h, \bdd{\zeta}_h) \in \ker \bdd{\Xi}_{\bdd{R}}^h } \left( \| z 
		- z_h \|_1 + \|\bdd{\zeta} - \bdd{\zeta}_h\|_1 \right) \\
		\leq  \frac{C M_{\bdd{R}} }{ \beta_{\bdd{R}} } \inf_{ (v_h, \bdd{\psi}_h) 
		\in \discrete{W}^h_{\Gamma} \times \discretev{V}^h_{\Gamma} } \left( \|z 
		- v_h\|_1 + \|\bdd{\zeta} - \bdd{\psi}_h \|_1 \right) 
		+ 2C_P \|\bdd{\zeta} - \bdd{R} \bdd{\zeta}\|.  
	\end{multline} 
\end{lemma}
\begin{proof}
	Let $(z, \bdd{\zeta}) \in \ker \bdd{\Xi}$ be given. Then, $\grad z = 
	\bdd{\zeta}$ and so $z \in H^2_{\Gamma}(\Omega)$. Thanks to 
	\cref{mot:mitc:red-commute}, $\grad T z = \bdd{R} \grad z$, and so $(Tz, 
	\bdd{\zeta}) \in \ker \bdd{\Xi}_{\bdd{R}}$. Thus, the triangle inequality and 
	\cref{eq:ker-opt-approx} give
	\begin{align*}
		&\inf_{ (z_h, \bdd{\zeta}_h) \in \ker \bdd{\Xi}_{\bdd{R}} } \left( \| z - 
		z_h \|_1 + \|\bdd{\zeta} - \bdd{\zeta}_h\|_1 \right) \\
		&\qquad  \leq \| z - T z\|_1 + \inf_{ (z_h, \bdd{\zeta}_h) \in \ker 
		\bdd{\Xi}_{\bdd{R}} } \left( \|Tz - z_h\|_1 + \| \bdd{\zeta} - 
		\bdd{\zeta}_h\|_1 \right) \\
		&\qquad \leq \| z - T z\|_1 +  \frac{C  M_{\bdd{R}} }{ \beta_{\bdd{R}} } 
		\inf_{ (v_h, \bdd{\psi}_h) \in \discrete{W}^h_{\Gamma} \times 
		\discretev{V}^h_{\Gamma} } \left( \|Tz - v_h\|_1 + \|\bdd{\zeta} - 
		\bdd{\psi}_h \|_1 \right) \\
		&\qquad \leq 2 \| z - T z\|_1 +  \frac{C M_{\bdd{R}} }{ \beta_{\bdd{R}} } 
		\inf_{ (v_h, \bdd{\psi}_h) \in \discrete{W}^h_{\Gamma} \times 
		\discretev{V}^h_{\Gamma} } \left( \|z - v_h\|_1 + \|\bdd{\zeta} - 
		\bdd{\psi}_h \|_1 \right). 
	\end{align*}
	Poincar\'{e}'s inequality \cref{eq:poincare-inequality} and 
	\cref{mot:mitc:red-commute} gives
	\begin{align*}
		\| z - T z\|_1 \leq C_P \| \grad z - T \grad z\| = C_P \| \bdd{\zeta} - 
		\bdd{R} \bdd{\zeta}\|.
	\end{align*} 
\end{proof}

The above results show that, in order for the scheme to satisfy the 
quasi-optimality approximation property \cref{eq:ker-opt-approx-true-kerxi}, it 
is sufficient that the inf-sup condition \cref{eq:invert-xi-inf-sup-u} holds. It 
may not be immediately clear that \cref{eq:invert-xi-inf-sup-u} constitutes an 
improvement over \cref{eq:ker-opt-approx-true-kerxi}. However, examining the 
inf-sup condition more closely, we see that it requires that the spaces 
$\discrete{W}^h_{\Gamma}$, $\discretev{V}^h_{\Gamma}$, and 
$\discretev{U}^h_{\Gamma}$ and the operator $\bdd{R}$ be chosen so that
\begin{align}
	\label{eq:image-xi-u}
	\image \bdd{\Xi}_{\bdd{R}}^h := \{ \bdd{\Xi}_{\bdd{R}}(v, \bdd{\psi}) : (v, 
	\bdd{\psi}) \in \discrete{W}^h_{\Gamma} \times \discretev{V}^h_{\Gamma} \} = 
	\discretev{U}^h_{\Gamma},
\end{align}
which, in turn, can be reduced to more familiar conditions as follows. One 
necessary condition for \cref{eq:image-xi-u} is that $\rot \image 
\bdd{\Xi}_{\bdd{R}}^h = \rot \discretev{U}^h_{\Gamma} = \discrete{Q}^h_{\Gamma}$, 
where the last equality follows from \cref{remark:consequences-sequence-r}. In 
turn, using the commuting diagram \cref{eq:mot:mitc:red-commute}, we require that
\begin{align*}
	\discrete{Q}^h_{\Gamma} = \rot \image \bdd{\Xi}_{\bdd{R}}^h = \rot \grad 
	\discrete{W}^h_{\Gamma} + \rot \bdd{R} \discretev{V}^h_{\Gamma} = P \rot 
	\discretev{V}^h_{\Gamma},
\end{align*}
or equivalently, the following inf-sup condition must hold: 
\begin{align}
	\label{mot:mitc:stokes-inf-sup-discrete}
	\begin{aligned}
		&\text{\parbox{0.8\linewidth}{There exists $\beta_{\rot} > 0$ independent 
				of $h$ such that}} \\
		&\qquad \qquad \inf_{ 0 \neq q \in \discrete{Q}^h_{\Gamma} } 
		\sup_{\bdd{0} \neq 
			\bdd{\theta} \in \discretev{V}^h_{\Gamma} } \frac{ (\rot 
			\bdd{\theta}, q) 
		}{ \|\bdd{\theta}\|_1 \|q\| } > \beta_{\rot}.
	\end{aligned}	
\end{align}

Condition \cref{mot:mitc:stokes-inf-sup-discrete} is nothing more (up to a 
rotation on $\bdd{\theta}$) than the familiar inf-sup, or Babu\v{s}ka-Brezzi, 
condition commonly seen in the mixed discretization of the Stokes equations. 
Here, as we have just shown, the condition is a natural consequence of the 
quasi-optimal approximation condition \cref{eq:ker-opt-approx-true-kerxi}. A 
second necessary condition for \cref{eq:image-xi-u} to hold is that the rot-free 
functions in $\image \bdd{\Xi}_{\bdd{R}}^h$ and $\discretev{U}^h_{\Gamma}$ should 
coincide. Thanks to the complex property \cref{mot:mitc:red-commute}, there holds
\begin{align*}
	\{ \bdd{\gamma} \in  \image \bdd{\Xi}_{\bdd{R}}^h : \rot \bdd{\gamma} \equiv 
	0 \} &= \grad \discrete{W}^{h}_{\Gamma} + \{ \bdd{R} \bdd{\theta} : 
	\bdd{\theta} \in \discretev{V}^h_{\Gamma}, \ \rot \bdd{R} \bdd{\theta} \equiv 
	0 \}, \\
	\{ \bdd{\gamma} \in  \discretev{U}^h_{\Gamma} : \rot \bdd{\gamma} \equiv 0 \} 
	&= \grad \discrete{W}^{h}_{\Gamma} \oplus^{\perp} \harmonic{H}^h_{\Gamma},
\end{align*}
where $\harmonic{H}^h_{\Gamma}$ is the space of discrete harmonic forms 
\cref{eq:harmonic-forms-discrete-def} and $\oplus^{\perp}$ indicates that the 
direct sum is orthogonal in $\bdd{L}^2(\Omega)$. If we define $\bdd{H} : 
\bdd{L}^2(\Omega) \to \harmonic{H}^h_{\Gamma}$ to be the $L^2$-projection operator
\begin{align}
	\label{eq:projection-harmonics-def}
	( \bdd{H} \bdd{\gamma}, \harmonic{g} ) = (\bdd{\gamma}, \harmonic{g}) \qquad 
	\forall \harmonic{g} \in \harmonic{H}^h_{\Gamma}, \ \forall \bdd{\gamma} \in 
	\bdd{L}^2(\Omega),
\end{align}
then the rot-free functions in $\image \bdd{\Xi}_{\bdd{R}}$ and 
$\discretev{U}^h_{\Gamma}$ coincide if and only if
\begin{align}
	\bdd{H} \{ \bdd{R} \bdd{\theta} : \bdd{\theta} \in \discretev{V}^h_{\Gamma}, 
	\ \rot \bdd{R} \bdd{\theta} \equiv 0 \} = \harmonic{H}^h_{\Gamma},
\end{align}
or, equivalently, a second inf-sup condition must hold: 
\begin{align}
	\label{mot:mitc:harmonic-inf-sup-discrete}
	\begin{aligned}
		& \text{\parbox{0.8\linewidth}{There exists $\beta_{\harmonic{H}} > 0$ 
				independent of $h$ such that}} \\
		&\qquad \qquad \inf_{ 0 \neq \harmonic{h} \in \harmonic{H}^h_{\Gamma} } 
		\sup_{\substack{ \bdd{0} \neq \bdd{\theta} \in \discretev{V}^h_{\Gamma} 
				\\ \rot \bdd{R} \bdd{\theta} \equiv 0 } } \frac{ (\bdd{R} 
				\bdd{\theta}, 
			\harmonic{h}) }{ \|\bdd{\theta}\|_1 \|\harmonic{h}\| } > 	
		\beta_{\harmonic{H}}.
	\end{aligned}	
\end{align}

\Cref{lem:inf-sup-equivalences} below confirms that conditions 
\cref{mot:mitc:stokes-inf-sup-discrete} and 
\cref{mot:mitc:harmonic-inf-sup-discrete} are equivalent to 
\cref{eq:invert-xi-inf-sup-u}, but as we shall see shortly, are easier to handle. 
The inf-sup condition \cref{mot:mitc:stokes-inf-sup-discrete} has appeared 
previously in the literature associated with locking-free methods for 
Reissner-Mindlin plates \cite{Brezzi89,Brezzi1991MITC,Peisker92} and has been the 
subject of 
investigation in the setting of incompressible flow. By way of contrast, despite 
the fact that \cref{mot:mitc:harmonic-inf-sup-discrete} also arises naturally 
from the quasi-optimal approximation property 
\cref{eq:ker-opt-approx-true-kerxi}, it seems to have not been recognized in the 
literature hitherto. One reason for this discrepancy may be that the majority of 
articles dealing with Reissner-Mindlin plates only concern themselves with simply 
connected domains and pure clamping conditions for which there are no nontrivial 
harmonic forms (meaning that \cref{mot:mitc:harmonic-inf-sup-discrete} is 
vacuous). 

The next result shows that the above conditions are equivalent to the inf-sup 
condition \cref{eq:invert-xi-inf-sup-u}:
\begin{lemma}
	\label{lem:inf-sup-equivalences}
	Suppose that \cref{mot:mitc:red-commute,mot:mitc:red-bounded} hold. Then, the 
	inf-sup condition \cref{eq:invert-xi-inf-sup-u} holds if and only if 
	conditions 
	\cref{mot:mitc:stokes-inf-sup-discrete,mot:mitc:harmonic-inf-sup-discrete} 
	hold. Moreover, the inf-sup constants satisfy
	\begin{align}
		\label{eq:inf-sup-equivalences}
		\beta_{\bdd{R}} \geq \frac{ C \beta_{\rot} \beta_{\harmonic{H}} }{ 
		M_{\bdd{R}}^2 }, \quad \beta_{\rot} \geq \frac{C \beta_{\bdd{R}}}{ 
		M_{\bdd{R}}^2 }, \quad \text{and} \quad \beta_{\harmonic{H}} \geq 
		\beta_{\bdd{R}}.
	\end{align}
\end{lemma}
The proof of \cref{lem:inf-sup-equivalences} appears in  
\cref{sec:proof-inf-sup-equivalences}.

\subsection{A priori error estimates}

Our quest for locking-free schemes has led to four conditions 
\cref{mot:mitc:red-commute,mot:mitc:red-bounded,mot:mitc:stokes-inf-sup-discrete,mot:mitc:harmonic-inf-sup-discrete}.
In fact, in the ensuing analysis, it will turn out that we shall not require the 
operator $T$ in the diagram \cref{eq:mot:mitc:red-commute}. Likewise, we shall 
only need for $\bdd{R}$ to be $\bdd{L}^2$-bounded from $\discretev{V}^h_{\Gamma}$ 
to $\discretev{U}^h_{\Gamma}$. Consequently, the first two conditions can be 
relaxed resulting in the conditions 
\ref{mitc:red-commute}--\ref{mitc:harmonic-inf-sup-discrete} displayed in 
\cref{fig:mitc-conditions}. 

\begin{remark}
	\label{rem:reduction-bounded}
	It is worth noting that, thanks to the commuting projection property 
	\ref{mitc:red-commute}, $\bdd{L}^2$-boundedness \ref{mitc:red-bounded}, and 
	the triangle inequality, there holds:
	\begin{align*}
		\| \bdd{R} \bdd{\psi} \|_{\rot} &= \left( \|\bdd{R} \bdd{\psi}\|^2 + 
		\|\rot \bdd{R} \bdd{\psi} \|^2 \right)^{\frac{1}{2}} 
		\leq \left( C_{\bdd{R}}^2 \| \bdd{\psi}\|^2 +  \| \Pi_{\discrete{Q}} \rot 
		\bdd{\psi} \|^2 \right)^{\frac{1}{2}} 
		\leq C C_{\bdd{R}} \|\bdd{\psi} \|_{1},
	\end{align*}
	so that \ref{mitc:red-commute}--\ref{mitc:red-bounded} automatically imply 
	that $\bdd{R} : \discretev{V}^h_{\Gamma} \to \discretev{U}^h_{\Gamma}$ is 
	bounded in $\hrot$.
\end{remark}

\noindent With these preparations in place, we can finally state the main result:
\begin{theorem}
	\label{lem:apriori-take-0}
	Suppose that \ref{mitc:red-commute}--\ref{mitc:harmonic-inf-sup-discrete} 
	hold. Let $(w, \bdd{\theta}) \in H^1_{\Gamma}(\Omega) \times 
	\bdd{\Theta}_{\Gamma}(\Omega)$ denote the solution to 
	\cref{eq:mr-variational-form} with $\bdd{\gamma} = \lambda t^{-2} 
	\bdd{\Xi}(w, \bdd{\theta})$ and let $(w_h, \bdd{\theta}_h) \in 
	\discrete{W}^h_{\Gamma} \times \discretev{V}^h_{\Gamma}$ denote the finite 
	element solution \cref{eq:mr-variational-form-fem} with $\bdd{\gamma}_h = 
	\lambda t^{-2} \bdd{\Xi}_{\bdd{R}}(w_h, \bdd{\theta}_h)$. Then, there holds
	\begin{align}
		\begin{aligned}
			\label{eq:apriori-take-0}
			&\| w - w_h \|_1 + \| \bdd{\theta} - \bdd{\theta}_{h} \|_1  + t \| 
			\bdd{\gamma} - \bdd{\gamma}_h\| \\
			&\quad \leq  \frac{ C C_{\bdd{R}} }{ \beta_{\bdd{R}} + t } \bigg( 
			\inf_{v \in \discrete{W}^h_{\Gamma}} \| w - v \|_1 + t \inf_{ 
			\bdd{\eta} \in \discretev{U}^h_{\Gamma}}  \|\bdd{\gamma} - 
			\bdd{\eta}\|  \\ 
			&\qquad \qquad \qquad \qquad  + \inf_{ \bdd{\psi} \in 
			\discretev{V}^h_{\Gamma}} \left( \|\bdd{\theta} - \bdd{\psi}\|_1 + 
			\|\bdd{\psi} - \bdd{R} \bdd{\psi}\|_{\rot} \right) 
			+\sup_{ \bdd{\psi} \in \discretev{V}^h_{\Gamma} } \frac{( \bdd{P} 
				\bdd{\gamma}, \bdd{\psi} - \bdd{R} \bdd{\psi} ) }{ \| \bdd{\psi} 
				\|_1 
			} \bigg).
		\end{aligned}
	\end{align}
	where $\beta_{\bdd{R}}$, defined in \cref{eq:invert-xi-inf-sup-u}, satisfies 
	\begin{align}
		\label{eq:inf-sup-equivalences-reduced}
		\beta_{\bdd{R}} \geq \frac{ C \beta_{\rot} \beta_{\harmonic{H}} }{ 
		C_{\bdd{R}}^2 }
	\end{align}
	and $\bdd{P} : \bdd{L}^2(\Omega) \to \discretev{U}^h_{\Gamma}$ is the 
	following $L^2$-projection operator:
	\begin{align}
		\label{eq:pi-x-p-def}
		( \bdd{P} \bdd{\eta},  \bdd{\gamma} ) = ( \bdd{\eta}, \bdd{\gamma} ) 
		\qquad \forall \bdd{\gamma} \in \discretev{U}^h_{\Gamma}, \ \forall 
		\bdd{\eta} \in \bdd{L}^2(\Omega).
	\end{align}
\end{theorem}
\noindent The proof of \cref{lem:apriori-take-0} appears in 
\cref{sec:proof-apriori-0}.

\begin{figure}[]
	\begin{tcolorbox}
		\begin{description}
			\item[(C1)\label{mitc:red-commute}] The spaces 
			$\discrete{W}^h_{\Gamma} \subset H^1_{\Gamma}(\Omega)$, 
			$\discretev{U}_{\Gamma}^h \subset \hrotgamma$ and 
			$\discrete{Q}^h_{\Gamma} \subset L^2_{\Gamma}(\Omega)$ and linear 
			projection operator $\bdd{R} : \bdd{\Theta}_{\Gamma}(\Omega) + 
			\discretev{U}^h_{\Gamma} \to \discretev{U}^h_{\Gamma}$ are chosen so 
			that the following diagram commutes and the bottom row forms a 
			discrete complex:
			\begin{equation}
				\label{eq:R-L2-project-commute}
				\begin{tikzcd}
					& \bdd{\Theta}_{\Gamma}(\Omega) \arrow{r}{\rot} 
					\arrow{d}{\bdd{R}} & L^2_{\Gamma}(\Omega) \arrow{d}{P} \\
					\discrete{W}^h_{\Gamma} \arrow{r}{\grad} & 
					\discretev{U}_{\Gamma}^h \arrow{r}{\rot} & 
					\discrete{Q}_{\Gamma}^h,
				\end{tikzcd}
			\end{equation}
			where $P : L^2(\Omega) \to \discrete{Q}_{\Gamma}^h$ is the 
			$L^2$-projection operator \cref{mot:eq:l2-projection-q}.
			
			\vspace{1em}
			
			\item[(C2)\label{mitc:red-bounded}] $\discretev{V}^h_{\Gamma} \subset 
			\bdd{\Theta}_{\Gamma}(\Omega)$ and there exists a constant 
			$C_{\bdd{R}} \geq 1$ independent of $h$ such that
			\begin{align}
				\label{eq:reduction-l2-approx}
				\|\bdd{R} \bdd{\psi} \| \leq C_{\bdd{R}} \|\bdd{\psi}\|
				\qquad \forall \bdd{\psi} \in \discretev{V}^h_{\Gamma}.
			\end{align}

			\vspace{1em}

			\item[(C3)\label{mitc:stokes-inf-sup-discrete}] There exists 
			$\beta_{\rot} > 0$ independent of $h$ such that 
			\begin{align}
				\label{eq:stokes-inf-sup-discrete}
				\inf_{ 0 \neq q \in \discrete{Q}^h_{\Gamma} } \sup_{\bdd{0} \neq 
				\bdd{\theta} \in \discretev{V}^h_{\Gamma} } \frac{ (\rot 
				\bdd{\theta}, q) }{ \|\bdd{\theta}\|_1 \|q\| } > \beta_{\rot}.
			\end{align}

			\item[(C4)\label{mitc:harmonic-inf-sup-discrete}] There exists 
			$\beta_{\harmonic{H}} > 0$ independent of $h$ such that 
			\begin{align}
				\label{eq:harmonic-inf-sup-discrete}
				\inf_{ 0 \neq \harmonic{h} \in \harmonic{H}^h_{\Gamma} } 
				\sup_{\substack{ \bdd{0} \neq \bdd{\theta} \in 
				\discretev{V}^h_{\Gamma} \\ \rot \bdd{R} \bdd{\theta} \equiv 0 } 
				} \frac{ (\bdd{R} \bdd{\theta}, \harmonic{h}) }{ 
				\|\bdd{\theta}\|_1 \|\harmonic{h}\| } > 	\beta_{\harmonic{H}}.
			\end{align}
		\end{description}
	\end{tcolorbox}
	\caption{Summary of our final set of sufficient conditions for a finite 
		element discretization of the Reissner-Mindlin model to be locking-free}
	\label{fig:mitc-conditions}
\end{figure}

Note that the first two terms on the RHS of \cref{eq:apriori-take-0} relate to 
approximation properties of the spaces that typically arise 
from an application of C\'{e}a's Lemma to conforming Galerkin schemes. The 
remaining two terms relate to the introduction of the reduction operator 
$\bdd{R}$ in place of the identity and are typical of terms seen in the analysis 
of non-conforming schemes.
The first of these additional terms can be bounded using approximation properties 
of $\discretev{U}^h_{\Gamma}$ and $\discrete{Q}^h_{\Gamma}$ as follows:
For any $\bdd{\psi} \in \discretev{V}^h_{\Gamma}$,  
\ref{mitc:red-commute}--\ref{mitc:red-bounded} show that $\bdd{R}$ is 
$\bdd{L}^2$-bounded on $\discretev{V}^h_{\Gamma} + \discretev{U}^h_{\Gamma}$, and 
so 
\begin{align*}
	\| \bdd{\psi} - \bdd{R} \bdd{\psi}\| = \inf_{\bdd{\eta} \in 
	\discretev{U}^h_{\Gamma}} \| \bdd{\psi} - \bdd{\eta} - \bdd{R}(\bdd{\psi} - 
	\bdd{\eta})\| 
	&\leq  C C_{\bdd{R}} \inf_{ \bdd{\eta} \in \discretev{U}^h_{\Gamma}} 
	\|\bdd{\psi} - \bdd{\eta}\| \\
	&\leq C C_{\bdd{R}} \left(  \|\bdd{\theta} - \bdd{\psi}\|_1 + \inf_{ 
	\bdd{\eta} \in \discretev{U}^h_{\Gamma}} \| \bdd{\theta} - \bdd{\eta}\| 
	\right),
\end{align*}
and
\begin{align*}
	\| \rot \bdd{\psi} - \rot \bdd{R} \bdd{\psi}\| &=  \|\rot  \bdd{\psi} - P 
	\rot \bdd{\psi}\|  \\
	&\leq \| \rot \bdd{\theta} - \rot \bdd{\psi}\| + \|\rot \bdd{\theta}  - P 
	\rot \bdd{\theta}\|  + \|P(\rot \bdd{\theta} - \rot \bdd{\psi})\| \\
	&\leq C \|\bdd{\theta} - \bdd{\psi}\|_1 + \inf_{q \in 
	\discrete{Q}^h_{\Gamma}} \|\rot \bdd{\theta}  - q\|.
\end{align*}
Consequently, there holds
\begin{multline}
	\label{eq:reduction-rot-error}
	\inf_{ \bdd{\psi} \in \discretev{V}^h_{\Gamma}} \left( \|\bdd{\theta} - 
	\bdd{\psi}\|_1 + \|\bdd{\psi} - \bdd{R} \bdd{\psi}\|_{\rot} \right) \\
	\leq C C_{\bdd{R}} \left( \inf_{ \bdd{\psi} \in \discretev{V}^h_{\Gamma}}  
	\|\bdd{\theta} - \bdd{\psi}\|_1 + \inf_{q \in \discrete{Q}^h_{\Gamma}} \|\rot 
	\bdd{\theta}  - q\| + \inf_{ \bdd{\eta} \in \discretev{U}^h_{\Gamma}} \| 
	\bdd{\theta} - \bdd{\eta}\|  \right).	
	\end{multline}
	The final term in \cref{eq:apriori-take-0} is not amenable to a similar 
	general treatment and instead must be treated on a case-by-case basis 
	depending on the particular discretization.

	Of course, all of this hinges on whether or not there are any families of 
	elements satisfying 
	conditions  \ref{mitc:red-commute}--\ref{mitc:harmonic-inf-sup-discrete}? 
	While 
	families that satisfy 
	\ref{mitc:red-commute}--\ref{mitc:stokes-inf-sup-discrete} 
	are plentiful \cite{Brezzi89,Brezzi1991MITC,Falk2008,Peisker92}, 
	whether or not they satisfy the new 
	condition \ref{mitc:harmonic-inf-sup-discrete} is not advertised in the 
	literature. Nevertheless, the absence of reports in the engineering 
	literature of 
	locking arising in the case of more general topology and boundary conditions 
	for 
	which \ref{mitc:harmonic-inf-sup-discrete} is not vacuous suggests that such 
	element families may indeed satisfy \ref{mitc:harmonic-inf-sup-discrete}. 
	This 
	question is pursued \cref{sec:demist-harm}.
	
	\subsection{Illustrative Examples}
	\label{sec:numerics}

	The norm appearing on the LHS of the a priori estimate 
	\cref{eq:apriori-take-0} is weighted by the thickness $t$, which means that 
	the 
	$\bdd{L}^2$ shear stress error, in 
	isolation, need not remain uniformly bounded as $t \to 0$, whereas 
	\cref{eq:apriori-take-0} shows that both the displacement and rotation error 
	are bounded as $t \to 0$ provided that the RHS is also bounded as $t \to 
	0$. We illustrate these observation in a simple numerical example.
	
	Let $\Omega = \Omega_0 \setminus (\bar{\Omega}_1 \cup \bar{\Omega}_2)$ be 
	the domain pictured in \cref{fig:box-mesh}, where $\Omega_0 = (0, 2) \times 
	(0, 
	1)$, $\Omega_1 = (1/4, 3/4)^2$, and $\Omega_2 = (5/4, 7/4) \times (1/4, 
	3/4)$. We 
	clamp the left edge 
	$\{x = 0\}$ and simply-support the
	right edge $\{x = 2\}$, while the remaining edges of $\Gamma$ are free. We 
	choose 
	the  linear functionals $F$ and $G$ in \cref{eq:mr-variational-form} so that 
	the 
	exact displacement, rotation, and shear stress are given by
	\begin{align*}
w &= \sin^3(4\pi x) \sin^3(4\pi y), \quad \bdd{\theta} = \left(1 - 100 t^2 
\lambda^{-1}\right) \grad w, \quad \text{and} \quad \bdd{\gamma} = 100 \grad 
w.
\end{align*}
It is easily checked that each of the holes is stress-free.

\begin{minipage}{0.95\textwidth}
\vspace{1em}
\begin{minipage}[b]{0.6\textwidth}
	\centering
	\includegraphics[height=6em]%
	{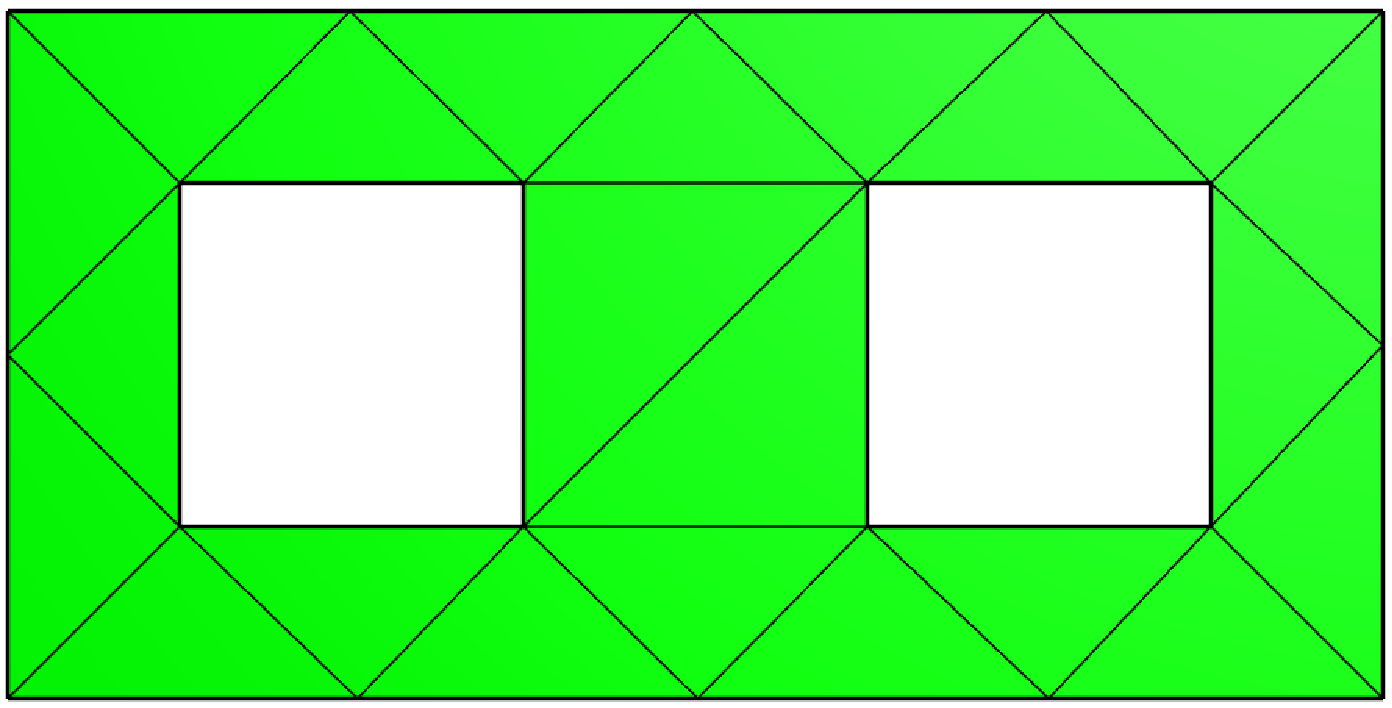}
	\captionof{figure}{Domain and initial mesh for examples}
	\label{fig:box-mesh}
\end{minipage}
\hfill
\begin{minipage}[b]{0.3\textwidth}
	\centering
	\begin{tabular}{c|c}
		$t$ & Marker \\
		\hline
		$1$ & \raisebox{-1mm}{\scalebox{2}{$\circ$}} \\
		$10^{-1}$ & $|$ \\
		$10^{-2}$ & $\square$ \\
		$10^{-3}$ & $\times$
	\end{tabular}
	\captionof{table}{Plot legends}
	\label{tab:plot-legend}
\end{minipage}
\\
\end{minipage}

We discretize \cref{eq:mr-variational-form-fem} with the lowest-order ($p=2$) 
Raviart-Thomas (RT) family described below in \cref{sec:rt-family} and the 
lowest-order 
($p=2$) Brezzi-Douglas-Marini (BDM) family described below in 
\cref{sec:bdm-family}. We consider a sequence of meshes obtained by progressively 
subdividing every triangle into four congruent subtriangles. For each method, we 
display the total relative error
\begin{align}
\label{eq:combined-error}
\frac{ \|w - w_h\|_1 + \|\bdd{\theta} - \bdd{\theta}_h\|_1 + t \|\bdd{\gamma} 
	- \bdd{\gamma}_h\| }{ \|w \|_1 + \|\bdd{\theta}\|_1 + t \|\bdd{\gamma}\| }
	\end{align}
	in \cref{fig:rt2-total-error} and \cref{fig:bdm2-total-error}, where
	the discrete shear stress $\bdd{\gamma}_h = \lambda t^{-2} 
	\bdd{\Xi}_{\bdd{R}}(w_h, \bdd{\theta}_h)$.  We will show in 
	\cref{sec:mitc-examples} that both of these 
	families satisfy conditions  
	\ref{mitc:red-commute}--\ref{mitc:harmonic-inf-sup-discrete}, and as 
	predicted by 
	\cref{thm:rt-family-apriori,thm:bdm-family-satisfies-conds} below, the error 
	decays as 
	$\mathcal{O}(h^2)$ for $t \in \{1,10^{-1},10^{-2},10^{-3}\}$ and
	remains bounded as $t \to 0$. 
	
	In the remaining subplots in \cref{fig:rt-errors,fig:bdm-errors}, we display 
	the 
	relative $H^1$ errors for $w_h$ and $\bdd{\theta}_h$ and relative $L^2$ error 
	for 
	$\bdd{\gamma}_h$. As suggested above, the displacement 
	and rotation errors are bounded as $t \to 0$, while the 
	convergence of the discrete shear stress $\bdd{\gamma}_h$ does indeed 
	deteriorate 
	as $t \to 0$. We also observe that the 
	displacement and rotation each converge at the best approximation rates
	\begin{align*}
\inf_{v \in \discrete{W}^h_{\Gamma}} \|w - v\|_1 
\quad \text{and} \quad 
\inf_{\bdd{\psi} \in \discretev{V}^h_{\Gamma}} \|\bdd{\theta} - \bdd{\psi}\|_1
\end{align*}
for all values of $t$. That is to say, the displacement converges as 
$\mathcal{O}(h^2)$ for 
the Raviart-Thomas family and $\mathcal{O}(h^3)$ for the BDM family, while the 
rotation converges as $\mathcal{O}(h^2)$ for both families.

\begin{figure}[htb]
\centering
\begin{subfigure}[b]{0.45\linewidth}
	\centering
	\includegraphics[width=\linewidth]%
	{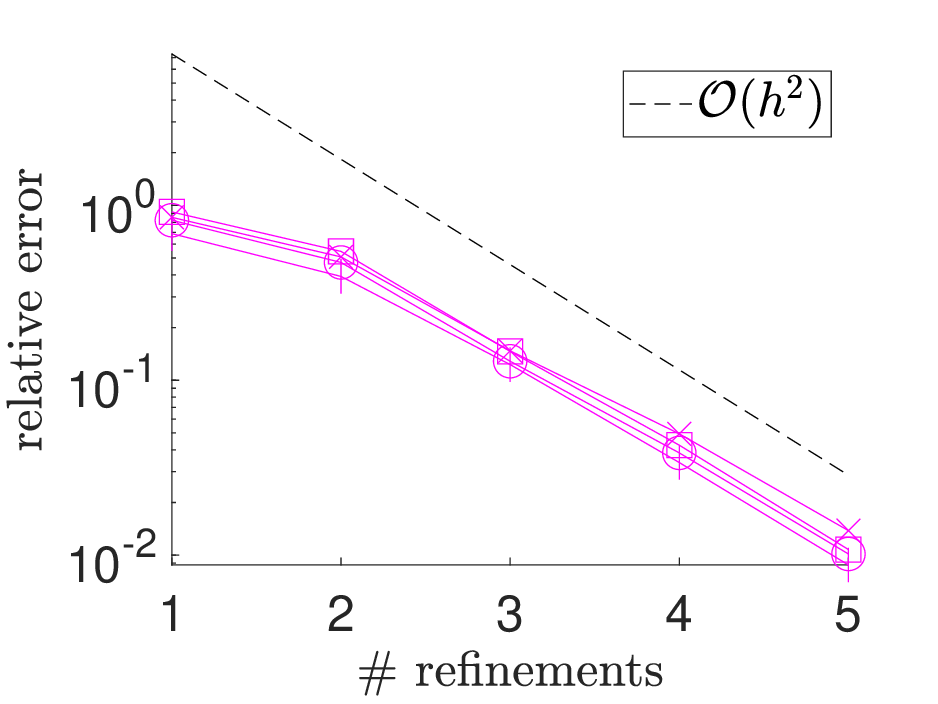}
	\caption{}
	\label{fig:rt2-total-error}
\end{subfigure}
\hfill
\begin{subfigure}[b]{0.45\linewidth}
	\centering
	\includegraphics[width=\linewidth]%
	{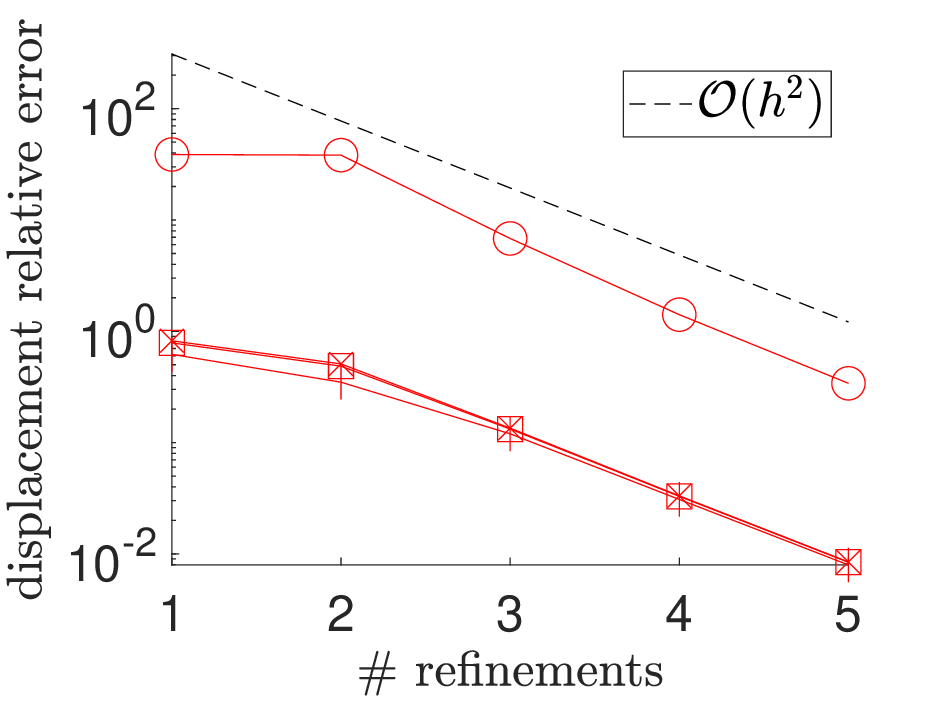}
	\caption{}
	\label{fig:rt2-w-error}
\end{subfigure}
\\
\begin{subfigure}[b]{0.45\linewidth}
	\centering
	\includegraphics[width=\linewidth]%
	{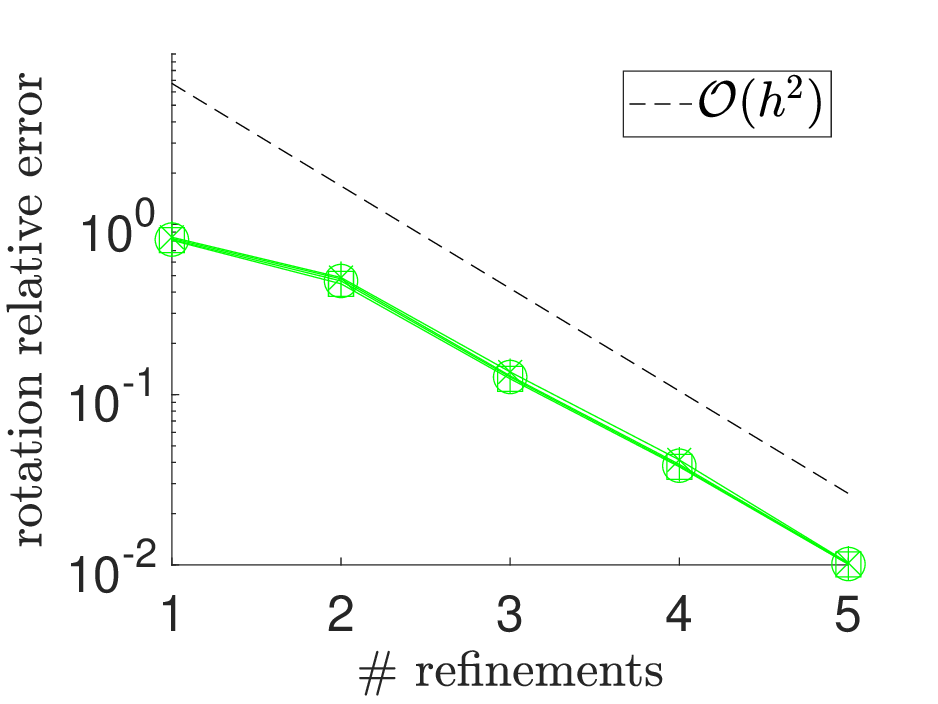}
	\caption{}
	\label{fig:rt2-theta-error}
\end{subfigure}
\hfill
\begin{subfigure}[b]{0.45\linewidth}
	\centering
	\includegraphics[width=\linewidth]%
	{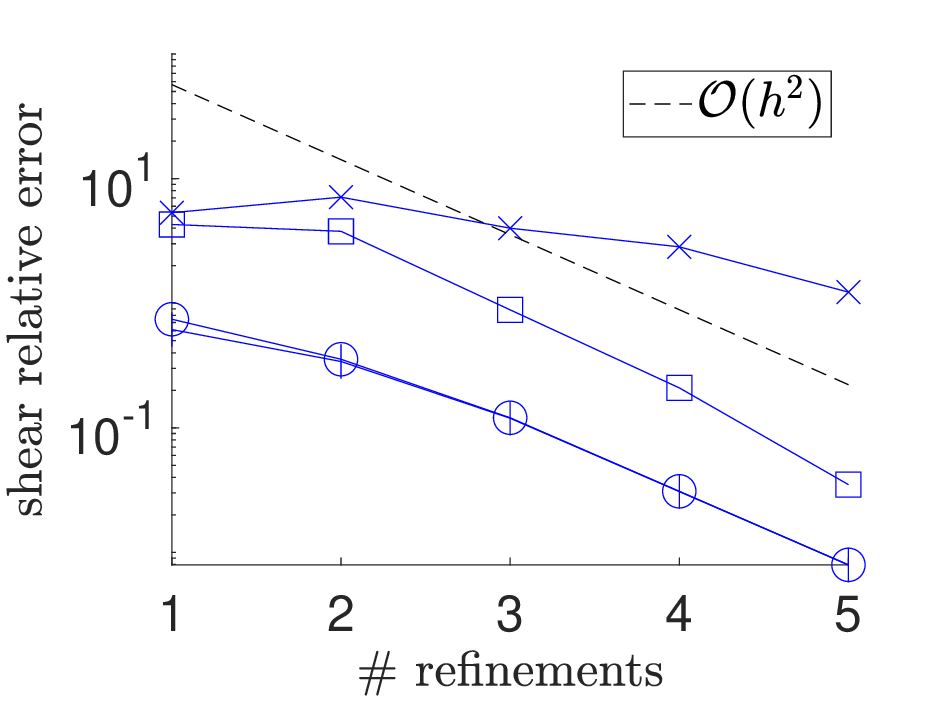}
	\caption{}
	\label{fig:rt2-shear-error}
\end{subfigure}
\caption{Relative errors for the lowest order ($p=2$) RT family 
	in \cref{sec:rt-family} for $t \in \{ 1, 10^{-1}, 10^{-2}, 10^{-3} \}$. 
	See \cref{tab:plot-legend} for the legend.}
\label{fig:rt-errors}
\end{figure}

\begin{figure}[htb]
\centering
\begin{subfigure}[b]{0.45\linewidth}
	\centering
	\includegraphics[width=\linewidth]%
	{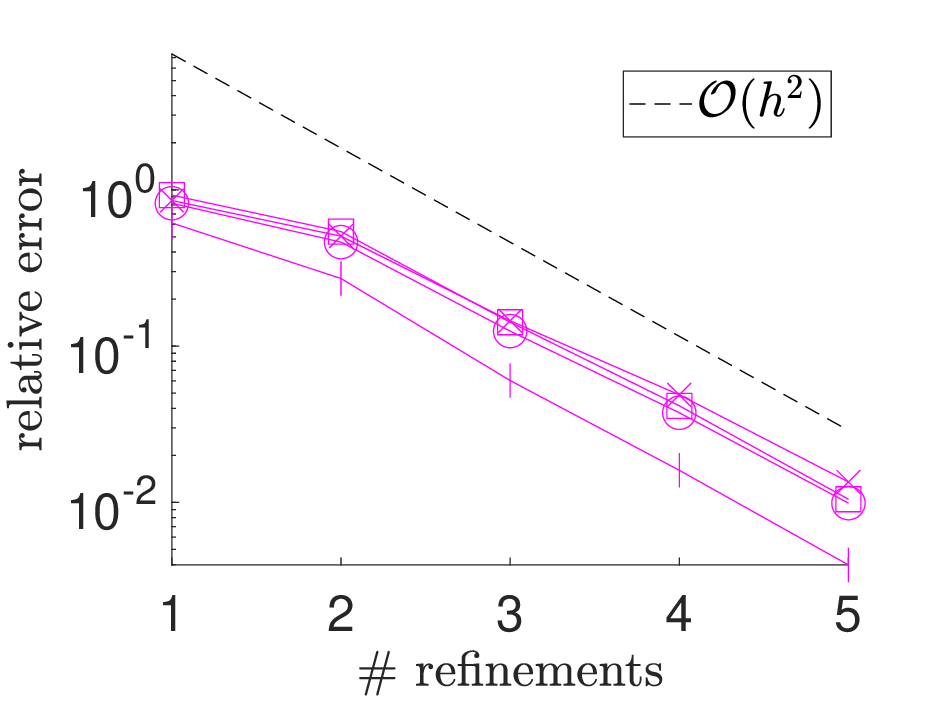}
	\caption{}
	\label{fig:bdm2-total-error}
\end{subfigure}
\hfill
\begin{subfigure}[b]{0.45\linewidth}
	\centering
	\includegraphics[width=\linewidth]%
	{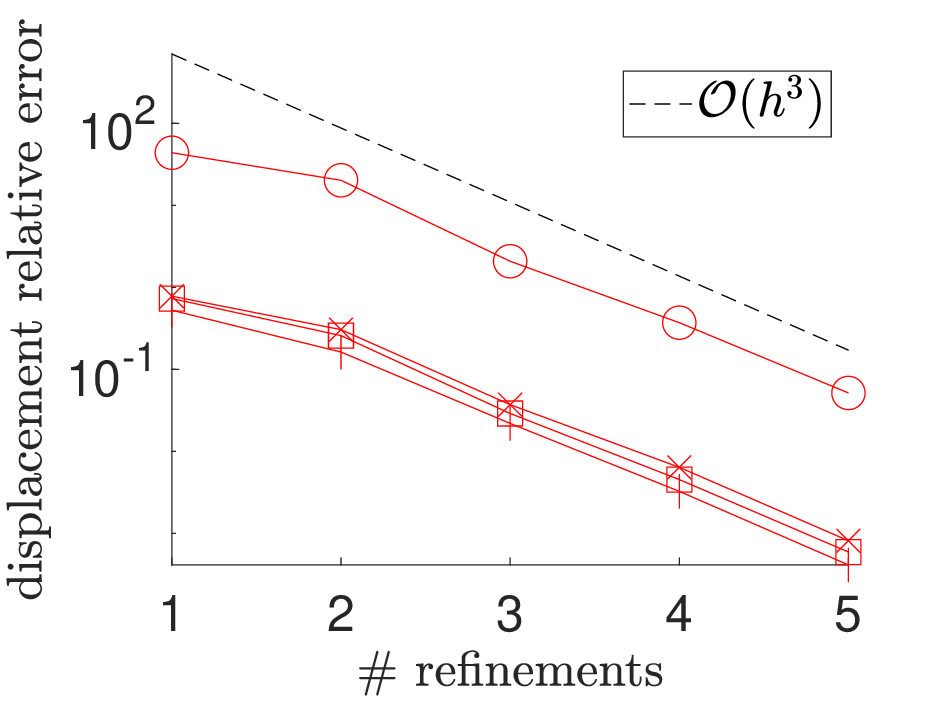}
	\caption{}
	\label{fig:bdm2-w-error}
\end{subfigure}
\\
\begin{subfigure}[b]{0.45\linewidth}
	\centering
	\includegraphics[width=\linewidth]%
	{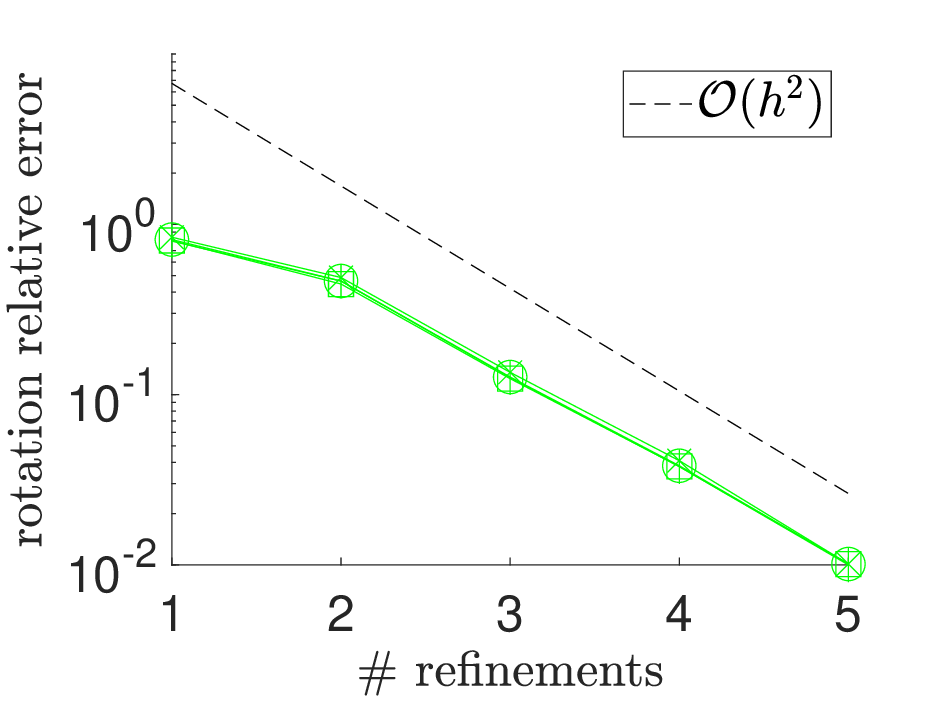}
	\caption{}
	\label{fig:bdm2-theta-error}
\end{subfigure}
\hfill
\begin{subfigure}[b]{0.45\linewidth}
	\centering
	\includegraphics[width=\linewidth]%
	{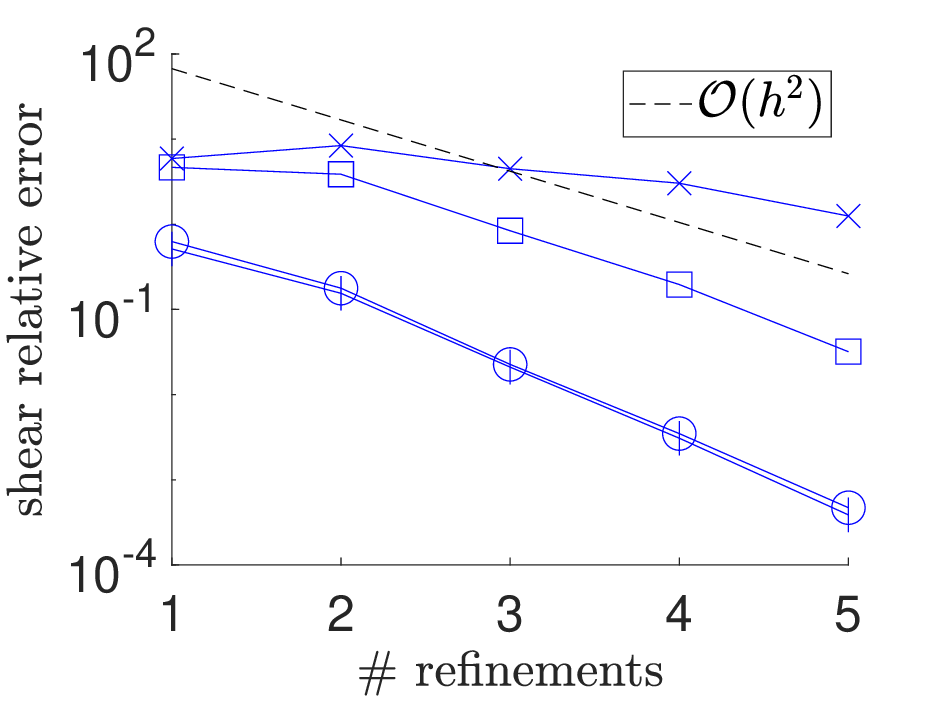}
	\caption{}
	\label{fig:bdm2-shear-error}
\end{subfigure}
\caption{Relative errors for the lowest order ($p=2$) BDM family in 
	\cref{sec:bdm-family} for $t \in \{ 1, 10^{-1}, 10^{-2}, 10^{-3} \}$. See 
	\cref{tab:plot-legend} for the legend.}
\label{fig:bdm-errors}
\end{figure}

We also discretize \cref{eq:mr-variational-form-fem-nored} with the lowest-order 
($p=4$) standard continuous finite element spaces given in 
\cref{eq:standard-spaces} and 
compute the discrete shear stress by $\bdd{\gamma}_h = \lambda t^{-2} 
\bdd{\Xi}(w_h, \bdd{\theta}_h)$.  We observe in \cref{fig:cg4-total-error} that 
the total error \cref{eq:combined-error} is bounded uniformly in $t$ and decays 
as $\mathcal{O}(h^4)$, which is rigorously justified in 
\cref{rem:high-order-convergence} below.  As with the 
RT and BDM families, the displacement (\cref{fig:cg4-w-error}) and 
rotation (\cref{fig:cg4-theta-error}) appear to converge at the best 
approximation rates:  $\mathcal{O}(h^5)$ and $\mathcal{O}(h^4)$, respectively. 
Moreover, the shear stress errors in \cref{fig:cg4-shear-error} do not appear to 
be bounded with respect to $t$, although the best approximation rate 
$\mathcal{O}(h^4)$ is observed after 3 refinements for the considered values of 
$t$.

\begin{figure}[htb]
\centering
\begin{subfigure}[b]{0.45\linewidth}
	\centering
	\includegraphics[width=\linewidth]%
	{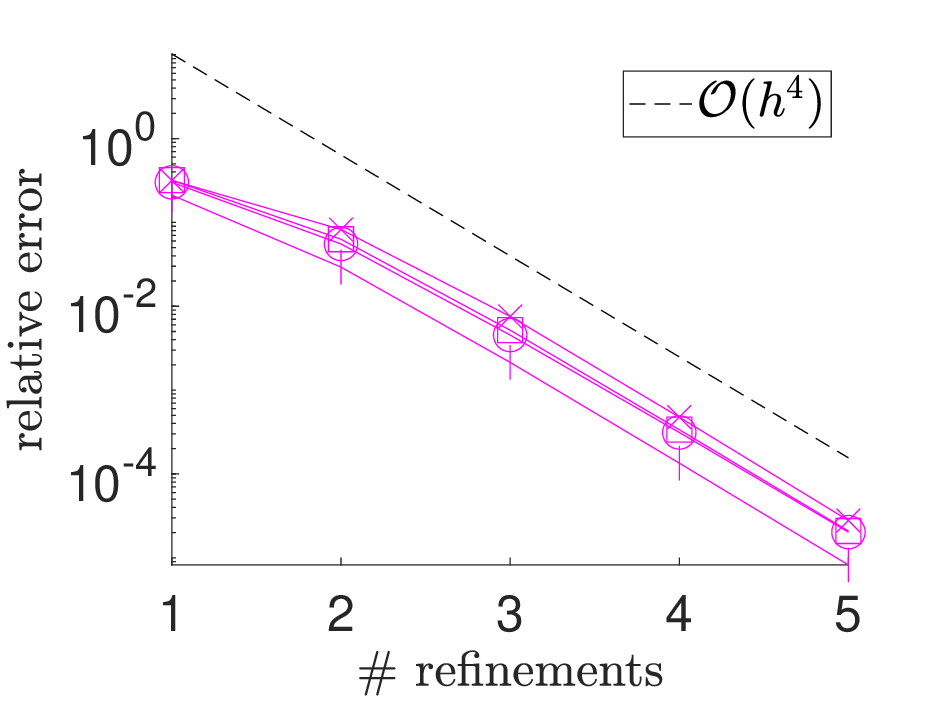}
	\caption{}
	\label{fig:cg4-total-error}
\end{subfigure}
\hfill
\begin{subfigure}[b]{0.45\linewidth}
	\centering
	\includegraphics[width=\linewidth]%
	{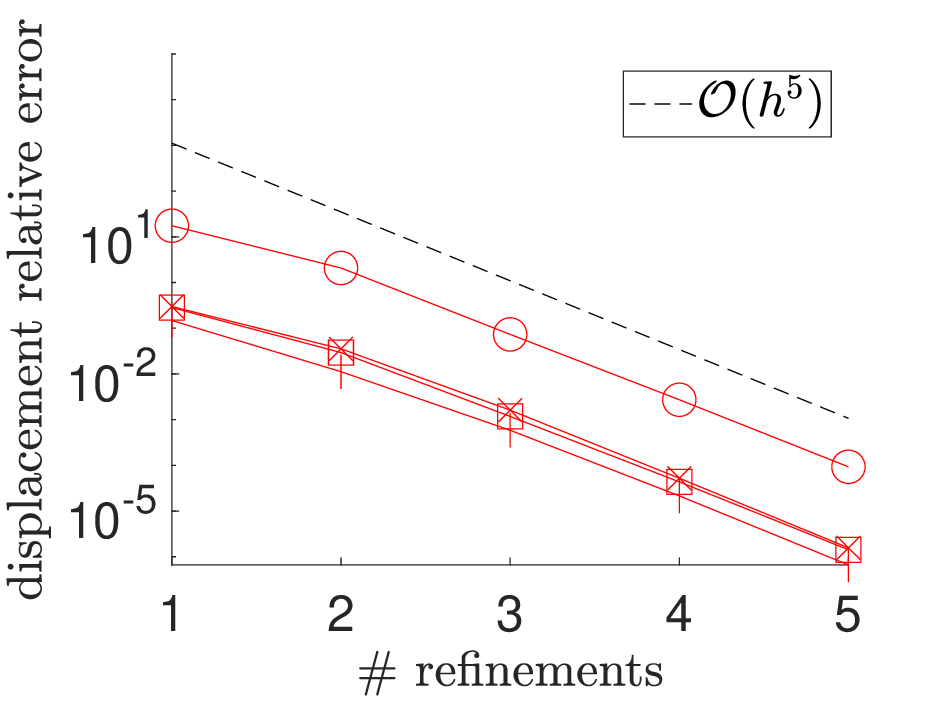}
	\caption{}
	\label{fig:cg4-w-error}
\end{subfigure}
\\
\begin{subfigure}[b]{0.45\linewidth}
	\centering
	\includegraphics[width=\linewidth]%
	{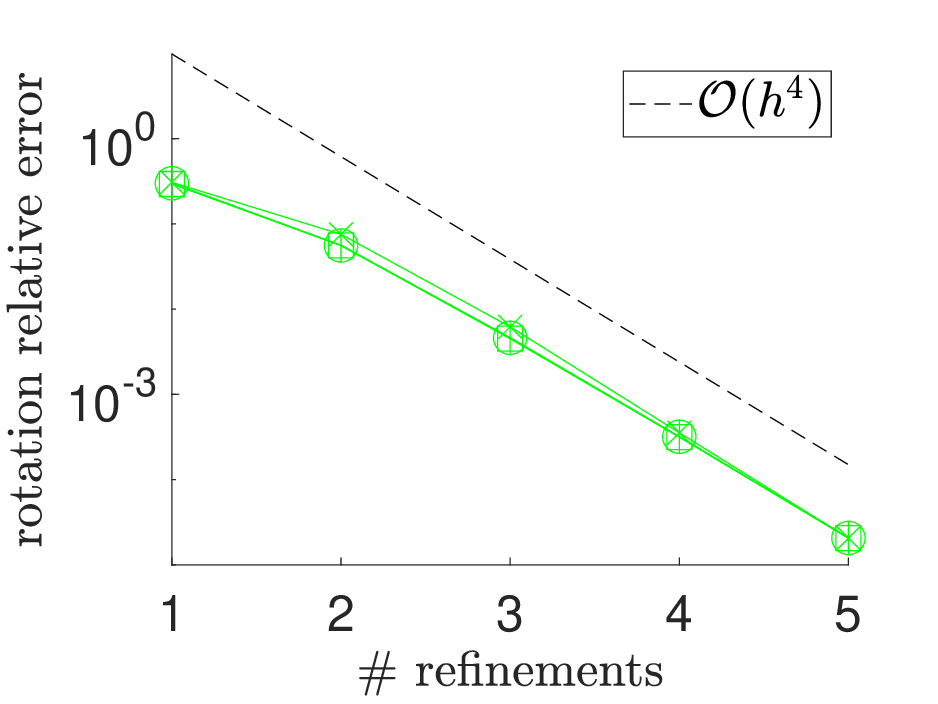}
	\caption{}
	\label{fig:cg4-theta-error}
\end{subfigure}
\hfill
\begin{subfigure}[b]{0.45\linewidth}
	\centering
	\includegraphics[width=\linewidth]%
	{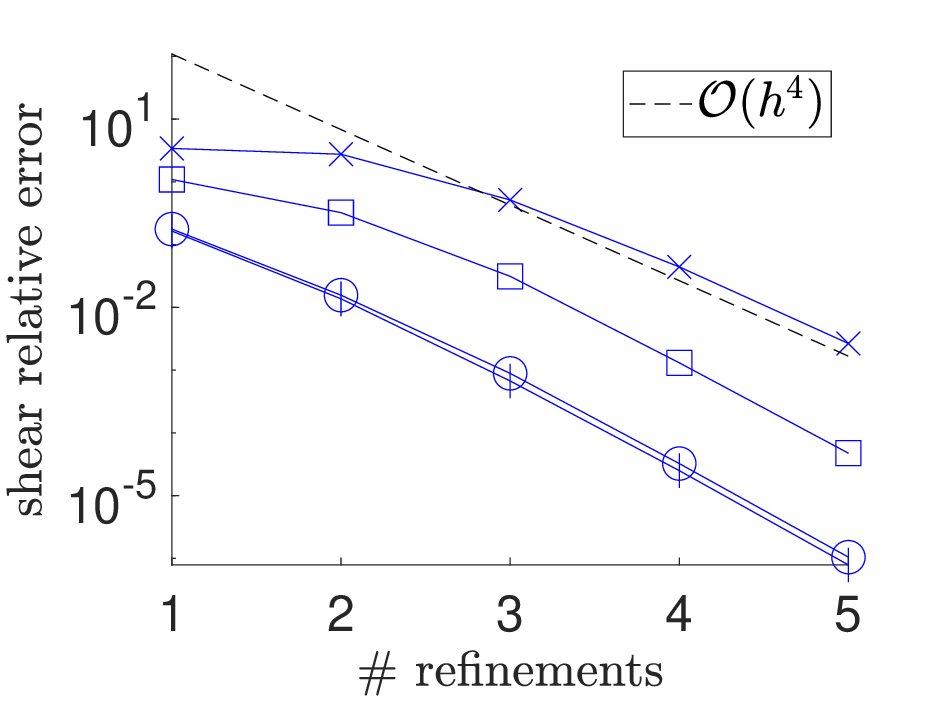}
	\caption{}
	\label{fig:cg4-shear-error}
\end{subfigure}
\caption{Relative errors for the lowest order ($p=4$) standard continuous 
	finite element spaces \cref{eq:standard-cg} for $t \in \{ 1, 10^{-1}, 
	10^{-2}, 10^{-3} \}$. See \cref{tab:plot-legend} 
	for 
	the legend.}
	\end{figure}

While conditions \ref{mitc:red-commute}--\ref{mitc:stokes-inf-sup-discrete} 
in the foregoing theory are more or less well-known, our informal approach to 
arriving at the conditions seems to be more intuitive than alternative 
derivations and, in addition, naturally leads to a new condition 
\ref{mitc:harmonic-inf-sup-discrete} that is needed in order for bounds of 
the type presented in \cref{lem:apriori-take-0} to hold for mixed boundary 
conditions and nontrivial topologies. While the new condition 
\ref{mitc:harmonic-inf-sup-discrete} may seem natural, the form in which it 
currently appears does not lend itself to standard analysis. Therefore, in 
the next section, we explore the condition more deeply with the aim of developing 
a set of sufficient conditions for \ref{mitc:harmonic-inf-sup-discrete} that 
are amenable to more familiar techniques.

\section{Demystifying condition \ref{mitc:harmonic-inf-sup-discrete}}
\label{sec:demist-harm}

Condition \ref{mitc:harmonic-inf-sup-discrete} involves the discrete harmonic 
forms. We start by giving a more explicit characterization of the 
continuous harmonic forms $\harmonic{H}_{\Gamma}(\Omega)$ 
\cref{eq:harmonic-forms-def} before discussing their discrete counterparts. 

\subsection{Structure of continuous harmonic forms}

As mentioned earlier, harmonic forms consist of non-trivial vector fields with 
vanishing rotation that are orthogonal to $\grad H^1_{\Gamma}(\Omega)$.
Harmonic forms fall into two categories. The first type can be written as the 
gradient of a function belonging to the space
\begin{align}
	\label{eq:w-gamma-def}
	\mathcal{W}_{\Gamma}(\Omega) := \{ w \in H^1(\Omega) : w|_{\Gamma_{cs}^{(1)}} 
	= 0 \text{ and } w|_{\Gamma_{cs}^{(i)}} \in \mathbb{R}, \ 2 \leq i \leq 
	N_{cs} \},
\end{align}
where $\{ \Gamma_{cs}^{{(i)}} \}_{i=1}^{N_{cs}}$ denotes the $N_{cs}$ connected 
components of $\Gamma_{cs} := \Gamma \setminus \Gamma_f$. More precisely, the
subspace 
\begin{align}
	\label{eq:w-gamma-harmonic-def}
	\mathfrak{W}_{\Gamma}(\Omega) := \{ w \in \mathcal{W}_{\Gamma}(\Omega) : 
	(\grad w, \grad v) = 0 \  \forall v \in H^1_{\Gamma}(\Omega) \},
\end{align}
has dimension $N_{cs} - 1$ and $\grad \mathfrak{W}_{\Gamma}(\Omega) \subseteq 
\harmonic{H}_{\Gamma}(\Omega)$ corresponds to harmonic forms in the first 
category.

The remaining harmonic forms are associated with circulations around the holes in 
the domain. Specifically, for each connected component of the boundary $\partial 
\Omega_i$, $i \in \{0,\ldots, H\}$, the circulation of a vector field 
$\bdd{\gamma}$ around $\Omega_i$ is a bounded linear functional $\circop_i : 
\hrot \to \mathbb{R}$ given by
\begin{align}
	\label{eq:circulation-def}
	\circop_i \bdd{\gamma} := \int_{\partial \Omega_i} \unitvec{t} \cdot 
	\bdd{\gamma} \d{s}  \qquad \forall \bdd{\gamma} \in \hrot.
\end{align}
Note that if $\bdd{\gamma} \in \hrotgamma$, then $\circop_i \bdd{\gamma} = 0$ if 
$\partial \Omega_i \subset \Gamma_{cs}$. Let $\mathfrak{I}$ denote the set of 
indices $i \in \{0,\ldots, H\}$ of the domains $\Omega_i$ for which $|\partial 
\Omega_i \cap \Gamma_f| > 0$. It is convenient to let $\mathfrak{I}^*$ denote the 
subset of $\mathfrak{I}$ obtained when an (arbitrary) single index is omitted 
from $\mathfrak{I}$. For definiteness, we choose the smallest such index so that, 
in summary,
\begin{align}
	\label{eq:index-sets-def}
	\mathfrak{I} := \{ i \in \{1,2,\ldots, H\} : |\partial \Omega_i \cap 
	\Gamma_f| > 0 \} \quad \text{and} \quad  	\mathfrak{I}^* := \begin{cases}
		\mathfrak{I} \setminus \{ \min \mathfrak{I} \} & \text{if } 
		|\mathfrak{I}| > 0, \\
		\emptyset & \text{otherwise}.
	\end{cases}
\end{align}
Associated with $\mathfrak{I}^*$ is an operator $\circop_{\mathfrak{I}^*} : \hrot 
\to \mathbb{R}^{|\mathfrak{I}^*|}$ that maps $\bdd{\gamma} \in \hrot$ to the 
vector whose entries are $\circop_i \bdd{\gamma}$ for $i \in \mathfrak{I}^*$. In 
the case where $\mathfrak{I}^* = \emptyset$, we use the convention 
$\circop_{\mathfrak{I}^*} \bdd{\gamma} := 0$ and $\mathbb{R}^0 := 0$.

The next result shows that the foregoing discussion accounts for all of the 
harmonic forms:
\begin{lemma}
	\label{lem:harmonic-complex-cont}
	The complex 
	\begin{equation}
		\label{eq:harmonic-complex-cont}
		\begin{tikzcd}[column sep=large]
			0 \arrow{r} & \mathfrak{W}_{\Gamma}(\Omega) \arrow{r}{\grad} & 
			\harmonic{H}_{\Gamma}(\Omega) \arrow{r}{ \circop_{\mathfrak{I}^*} } & 
			\mathbb{R}^{|\mathfrak{I}^*|} \arrow{r} & 0
		\end{tikzcd}
	\end{equation}
	is exact. Moreover,
	\begin{align}
		\label{eq:dim-harmonic-forms}
		\dim \harmonic{H}_{\Gamma}(\Omega) = |\mathfrak{I}^*| + N_{cs} - 1.
	\end{align}
\end{lemma}
The proof of \cref{lem:harmonic-complex-cont} appears in 
\cref{sec:proof-harmonic-complex-cont}.

\subsection{Sufficient conditions for \ref{mitc:harmonic-inf-sup-discrete}}

If the discrete complex in the bottom row of \cref{eq:R-L2-project-commute} 
completely captures the structure of the continuous complex 
\cref{eq:de-rham-bcs}, then we would expect that the discrete harmonic forms 
$\harmonic{H}_{\Gamma}^h$ \cref{eq:harmonic-forms-discrete-def} also form an 
exact complex analogous to the continuous case \cref{eq:harmonic-complex-cont}. 
That is, 
$\grad (\mathbb{W}^h \cap \mathcal{W}_{\Gamma}(\Omega)) \subset 
\discretev{U}^h_{\Gamma}$ and the following complex would be exact:
\begin{equation}
	\label{eq:harmonic-complex-discrete}
	\begin{tikzcd}[column sep=large]
		0 \arrow{r} & \mathfrak{W}^h_{\Gamma} \arrow{r}{\grad} & 
		\harmonic{H}^h_{\Gamma} \arrow{r}{ 
			\circop_{\mathfrak{I}^*} } & \mathbb{R}^{|\mathfrak{I}^*|} \arrow{r} 
			& 0,
	\end{tikzcd}
\end{equation}
where 
\begin{align}
	\mathfrak{W}^h_{\Gamma} := \{ v \in \discrete{W}^h \cap 
	\mathcal{W}_{\Gamma}(\Omega) : (\grad v, \grad u) = 0 \ \forall u \in 
	\discrete{W}^h_{\Gamma} \}.
\end{align}	

One requirement for exactness is that if $\harmonic{h} \in 
\harmonic{H}^h_{\Gamma}$ satisfies $\circop_{\mathfrak{I}^*} \harmonic{h} = 
\vec{0}$, then $\harmonic{h} = \grad w$ for some $w \in \mathfrak{W}^h_{\Gamma}$. 
This statement is seen to be equivalent to the following condition on recalling 
that $\harmonic{H}^h_{\Gamma} \oplus \grad \discrete{W}^h_{\Gamma}$ is an 
orthogonal decomposition of $\rot$-free functions in $\discretev{U}^h_{\Gamma}$:

\vspace{1em}

\begin{description}
	\item[(H1)\label{harm:gradient-cond}] $\grad (\mathbb{W}^h \cap 
	\mathcal{W}_{\Gamma}(\Omega)) \subset \discretev{U}^h_{\Gamma}$ and if 
	$\bdd{\eta} \in \discretev{U}^h_{\Gamma}$ satisfies $\rot \bdd{\eta} \equiv 
	0$ and $\circop_{\mathfrak{I}^*} \bdd{\eta} = \vec{0}$, then $\bdd{\eta} = 
	\grad w$ for some $w \in \discrete{W}^h \cap \mathcal{W}_{\Gamma}(\Omega)$.	
\end{description}

\vspace{1em}

The remaining property needed for exactness is that $\circop_{\mathfrak{I}^*} : 
\harmonic{H}^h_{\Gamma} \to \mathbb{R}^{|\mathfrak{I}^*|}$ is surjective. To 
reformulate this property in more familiar terms, we assume that a Fortin 
operator $\bdd{\Pi}_F : \bdd{\Theta}_{\Gamma}(\Omega) \to 
\discretev{V}^h_{\Gamma}$ exists for the pair 
$\discretev{V}^h_{\Gamma} \times \discrete{Q}^h_{\Gamma}$. In general, the 
existence of a Fortin operator follows from \ref{mitc:stokes-inf-sup-discrete}  
\cite[Remark 5.4.3]{BoffiBrezziFortin13}, but here we additionally require that 
the Fortin operator preserves the edge tangential moments on individual edges of 
the mesh:
\vspace{1em}

\begin{description}
	\item[(H2)\label{harm:v-fortin}] There exists a linear operator 
	$\bdd{\Pi}_{F} : \bdd{\Theta}_{\Gamma}(\Omega) \to \discretev{V}_{\Gamma}^h$ 
	and a constant $C_F \geq 1$ independent of $h$ satisfying the following for 
	all $\bdd{\psi} \in \bdd{\Theta}_{\Gamma}(\Omega)$:
	\begin{subequations}
		\begin{alignat}{2}
			\int_{e} \unitvec{t} \cdot \bdd{\Pi}_F \bdd{\psi} \d{s}
			&= \int_{e}\unitvec{t} \cdot \bdd{\psi} \d{s} \qquad & &\forall e 
			\in \mathcal{E}_h, \\
			P \rot \bdd{\Pi}_F \bdd{\psi} &= P \rot \bdd{\psi}, \qquad & & \\
			\|\bdd{\Pi}_F \bdd{\psi} \|_1 &\leq C_F \|\bdd{\psi}\|_{1}.
		\end{alignat}
	\end{subequations} 
\end{description}
\vspace{1em}

\noindent In a similar vein, we require:
\vspace{1em}
\begin{description}
	\item[(H3)\label{harm:r-interp-edge}] The reduction operator $\bdd{R}$ 
	preserves edge tangential moments:
	\begin{align}
		\label{eq:r-interp-circ}
		\int_{e} \unitvec{t} \cdot \bdd{R} \bdd{\theta} \d{s} = \int_{e} 
		\unitvec{t} \cdot \bdd{\theta}  \d{s} \qquad \forall e \in \mathcal{E}_h, 
		\ 
		\forall \bdd{\theta} \in  \discretev{V}^h_{\Gamma}.
	\end{align}
\end{description}
\vspace{1em}

The following result shows that we may replace conditions 
\ref{mitc:stokes-inf-sup-discrete}--\ref{mitc:harmonic-inf-sup-discrete} with 
conditions \ref{harm:gradient-cond}--\ref{harm:r-interp-edge} :
\begin{theorem}
	\label{thm:harm-conds-harm-inf-sup}
	Suppose that \ref{mitc:red-commute}--\ref{mitc:red-bounded} and 
	\ref{harm:gradient-cond}--\ref{harm:r-interp-edge} hold. Then, the discrete 
	complex \cref{eq:harmonic-complex-discrete} is exact and 
	\ref{mitc:stokes-inf-sup-discrete}--\ref{mitc:harmonic-inf-sup-discrete} hold 
	with
	\begin{align}
		\beta_{\rot} \geq C C_F^{-1} \quad \text{and} \quad 
		\beta_{\harmonic{H}} \geq C (C_F C_{\bdd{R}})^{-1}.
	\end{align}
\end{theorem}
The proof of \cref{thm:harm-conds-harm-inf-sup} appears in 
\cref{sec:proof-harm-conds-harm-inf-sup}.

In summary, we have replaced the inf-sup condition 
\ref{mitc:stokes-inf-sup-discrete} involving the pair $\discretev{V}^h_{\Gamma} 
\times \discrete{Q}^h_{\Gamma}$ and the inf-sup condition 
\ref{mitc:harmonic-inf-sup-discrete} concerning discrete harmonic forms with a 
property involving gradients of discrete functions \ref{harm:gradient-cond}, the 
existence of a particular Fortin operator for the pair $\discretev{V}^h_{\Gamma} 
\times \discrete{Q}^h_{\Gamma}$ \ref{harm:v-fortin}, and an algebraic condition 
on the reduction operator \ref{harm:r-interp-edge}. The resulting three 
conditions, summarized in \cref{fig:pmitc-conditions}, give the following 
corollary of \cref{lem:apriori-take-0}:
\begin{corollary}
	\label{cor:apriori-pmitc}
	Suppose that
	\ref{mitc:red-commute}--\ref{mitc:red-bounded} and 
	\ref{pmitc:gradient-cond}--\ref{pmitc:r-interp-edge} hold. Then, 
	$\beta_{\bdd{R}}$ defined in \cref{eq:invert-xi-inf-sup-u} satisfies
	\begin{align}
		\beta_{\bdd{R}} \geq \frac{C }{  C_F^2 C_{\bdd{R}}^3}. 
	\end{align}	
\end{corollary}
\begin{proof}
	The result follows from \cref{lem:apriori-take-0,thm:harm-conds-harm-inf-sup}.
\end{proof}

\begin{figure}[htb]
	\begin{tcolorbox}
		\begin{description}		
			
			\item[(H1)\label{pmitc:gradient-cond}] $\grad (\mathbb{W}^h \cap 
			\mathcal{W}_{\Gamma}(\Omega)) \subset \discretev{U}^h_{\Gamma}$ and 
			if $\bdd{\eta} \in \discretev{U}^h_{\Gamma}$ satisfies $\rot 
			\bdd{\eta} \equiv 0$ and $\circop_{\mathfrak{I}^*} \bdd{\eta} = 
			\vec{0}$, then $\bdd{\eta} = \grad w$ for some $w \in \discrete{W}^h 
			\cap \mathcal{W}_{\Gamma}(\Omega)$.	
			
			\vspace{1em}
			
			\item[(H2)\label{pmitc:v-fortin}] There exists a linear operator 
			$\bdd{\Pi}_{F} : \bdd{\Theta}_{\Gamma}(\Omega) \to 
			\discretev{V}_{\Gamma}^h$ and a constant $C_F \geq 1$ independent of 
			$h$ such that for all $\bdd{\psi} \in \bdd{\Theta}_{\Gamma}(\Omega)$
			\begin{subequations}
				\label{peq:v-fortin}
				\begin{alignat}{2}
					\int_{e} \unitvec{t} \cdot \bdd{\Pi}_F \bdd{\psi} \d{s}
					&= \int_{e}\unitvec{t} \cdot \bdd{\psi} \d{s} \qquad & 
					&\forall e \in \mathcal{E}_h, \\
					P \rot \bdd{\Pi}_F \bdd{\psi} &= P \rot \bdd{\psi}, \qquad & 
					& \\
					\|\bdd{\Pi}_F \bdd{\psi} \|_1 &\leq C_F \|\bdd{\psi}\|_{1}.
				\end{alignat}
			\end{subequations}   
			
			\item[(H3)\label{pmitc:r-interp-edge}] The reduction operator 
			$\bdd{R}$ preserves edge tangential moments:
			\begin{align}
				\label{peq:r-interp-circ}
				\int_{e} \unitvec{t} \cdot \bdd{R} \bdd{\theta} \d{s} = \int_{e} 
				\unitvec{t} \cdot \bdd{\theta} \d{s}  \qquad \forall e \in 
				\mathcal{E}_h, \ \forall \bdd{\gamma} \in  
				\discretev{V}^h_{\Gamma}.
			\end{align}	
			
		\end{description}
	\end{tcolorbox}
	\caption{Set of sufficient conditions for 
	\ref{mitc:stokes-inf-sup-discrete}--\ref{mitc:harmonic-inf-sup-discrete} 
	assuming \ref{mitc:red-commute}--\ref{mitc:red-bounded}.}
	\label{fig:pmitc-conditions}
\end{figure}

\section{Examples of finite element families satisfying the conditions}
\label{sec:mitc-examples}

Let $\{\mathcal{T}_h\}$ be a family of shape-regular triangulations of 
$\Omega$ satisfying the usual assumptions \cite[p. 38]{Cia02}, and recall that  
$\mathcal{V}_h$ and $\mathcal{E}_h$ denote the set of vertices and edges of 
$\mathcal{T}_h$. We give two examples of families of elements from 
\cite{Brezzi89,Brezzi1991MITC,Peisker92}
satisfying 
conditions \ref{mitc:red-commute}--\ref{mitc:red-bounded} and 
\ref{pmitc:gradient-cond}--\ref{pmitc:r-interp-edge}. In each case, we define 
discrete spaces $\discrete{W}^h$, $\discretev{V}^h$, $\discretev{U}^h$, and 
$\discrete{Q}^h$ with the corresponding spaces satisfying boundary conditions 
chosen in the usual way:
\begin{subequations}
	\label{eq:example-families-bcs}
	\begin{alignat}{2}
		\discrete{W}^h_{\Gamma} &:= \discrete{W}^h \cap H^1_{\Gamma}(\Omega), 
		\qquad &  \discretev{V}^h_{\Gamma} &:= \discretev{V}^h \cap 
		\bdd{\Theta}_{\Gamma}(\Omega), \\
		\discretev{U}^h_{\Gamma} &:= \discretev{U}^h \cap \hrotgamma, \qquad & 
		\discrete{Q}_{\Gamma}^h &:= \discrete{Q}^h \cap L^2_{\Gamma}(\Omega).
	\end{alignat}	
\end{subequations}

\subsection{Raviart-Thomas MITC family}
\label{sec:rt-family}

For $p \in \mathbb{N}$, choose
\begin{align}
	\label{eq:rt-family-w}
	\discrete{W}^h &:= \{ v \in C(\Omega) : v|_{K} \in \mathcal{P}_p(K) \ \forall 
	K \in \mathcal{T}_h \}, \\
	\label{eq:rt-family-u}
	\discretev{U}^h &:= \{ \bdd{\gamma} \in \hrot : \bdd{\gamma}|_{K} \in 
	\mathcal{P}_{p-1}(K)^2 + \bdd{x}^{\perp} \mathcal{P}_{p-1}(K) \ \forall K \in 
	\mathcal{T}_h \}, \\
	\label{eq:rt-family-q}
	\discrete{Q}^h &:= \{ q \in L^2(\Omega) : q|_{K} \in \mathcal{P}_{p-1}(K) \ 
	\forall K \in \mathcal{T}_h \},
\end{align}
where $\bdd{x}^{\perp} := (-x_2, x_1)^T$ for $\bdd{x} \in \mathbb{R}^2$.
Evidently, $\discrete{W}^h$ is the space of continuous piecewise polynomials of 
degree $p$, $\discretev{U}^h$ is the space of (rotated) Raviart-Thomas (RT) 
vector 
fields of degree $p$, and $\discrete{Q}^h$ is the space of discontinuous 
piecewise polynomials of degree $p-1$. We take 
$\bdd{R} : \bdd{H}^1(\Omega) + \discretev{U}^h \to \discretev{U}^h$ to be the 
standard RT finite element interpolant defined by the conditions:
\begin{subequations}
	\label{eq:rt-r-operator}
	\begin{alignat}{2}
		\label{eq:rt-r-operator-edge}
		\int_{e} \unitvec{t} \cdot (\bdd{R} \bdd{\gamma} - \bdd{\gamma}) r \d{s} 
		&= 0 \qquad & &\forall r \in \mathcal{P}_{p-1}(e), \ \forall e \in 
		\mathcal{E}_h, \\
		\label{eq:rt-r-operator-interior}
		\int_{K} (\bdd{R} \bdd{\gamma} - \bdd{\gamma}) \cdot \bdd{\eta} 
		\d{\bdd{x}} &= 0 \qquad & &\forall \bdd{\eta} \in \mathcal{P}_{p-2}(K)^2, 
		\ \forall K \in \mathcal{T}_h.
	\end{alignat}
\end{subequations}
Finally,  take the rotation space $\discretev{V}^h$ to be
\begin{align}
	\label{eq:rt-family-v}
	\discretev{V}^h &:= \left\{ \bdd{\psi} \in \bdd{C}(\Omega) : \bdd{\psi}|_{K} 
	\in \mathcal{P}_{p+1}(K)^2 \ \forall K \in \mathcal{T}_h \text{ and } 
	\bdd{\psi}|_{e} \in \mathcal{P}_{p}(K)^2 \ \forall e \in \mathcal{E}_h 
	\right\}, 
\end{align}
so that $\discretev{V}^h$ is the space of continuous piecewise polynomials of 
degree $p$ augmented with interior degree $p+1$ bubble functions. 

Conditions \ref{mitc:red-commute}--\ref{mitc:red-bounded} and 
\ref{pmitc:gradient-cond}--\ref{pmitc:r-interp-edge} are readily verified by 
simply collecting well-known results from the literature:
\begin{theorem}
	\label{thm:rt-family-satisfies-conds}
	Let $\discrete{W}^h_{\Gamma}$, $\discretev{V}^h_{\Gamma}$, 
	$\discretev{U}^h_{\Gamma}$, and $\discrete{Q}^h_{\Gamma}$ be chosen as in 
	\cref{eq:rt-family-w}, \cref{eq:rt-family-v}, \cref{eq:rt-family-u}, and 
	\cref{eq:rt-family-q}  for some $p \geq 2$, and let $\bdd{R}$ be defined as 
	in \cref{eq:rt-r-operator}. Then, conditions 
	\ref{mitc:red-commute}-\ref{mitc:red-bounded} and 
	\ref{harm:gradient-cond}--\ref{harm:r-interp-edge} are satisfied. 
\end{theorem}
\begin{proof}
	\ref{mitc:red-commute} is standard, see e.g. \cite[Theorem 
	5.2]{ArnFalkWin06}. \ref{mitc:red-bounded} is also standard, although not 
	typically stated in the above form, so we provide a short proof. Let 
	$\bdd{\psi} \in \discretev{V}^h_{\Gamma}$. Then, \cite[eq 
	(5.4)]{ArnFalkWin06} shows that
	\begin{align}
		\label{proof:eq:r-approx-v}
		\| \bdd{\psi} - \bdd{R} \bdd{\psi} \|_{K} \leq C h_K |\bdd{\psi}|_{1, K} 
		\qquad \forall K \in \mathcal{T}_h, 
	\end{align}
	where $h_K = \mathrm{diam}(K)$, 
	and so applying a standard inverse inequality then gives
	\begin{align*}
		\|\bdd{R} \bdd{\psi}\|_K \leq \|\bdd{\psi}\|_K + \|\bdd{\psi} - \bdd{R} 
		\bdd{\psi}\|_K \leq C \|\bdd{\psi}\|_K \qquad \forall K \in \mathcal{T}_h.
	\end{align*}
	\ref{mitc:red-bounded} now follows from squaring the above inequality and 
	summing over the elements.
	
	The inclusion $\grad (\mathbb{W}^h \cap \mathcal{W}_{\Gamma}(\Omega)) \subset 
	\discretev{U}^h_{\Gamma}$ follows by definition. Suppose that $\bdd{\eta} 
	\in \discretev{U}^h_{\Gamma}$ satisfies $\rot \bdd{\eta} \equiv 0$ and 
	$\circop_{\mathfrak{I}^*} \bdd{\eta} = \vec{0}$. Then, $\circop_i \bdd{\eta} 
	= 0$ for all $i \in \{0, 1,2,\ldots, H\}$, and so \cite[p. 37 Theorem 
	3.1]{GiraultRaviart86} shows that $\bdd{\eta} = \grad w$ for some $w \in 
	H^1(\Omega)$.
	
	Let $K \in \mathcal{T}_h$. Then, $\grad w|_{K} = \bdd{\eta}|_{K} \in 
	\mathcal{P}_{p-1}(K)^2 + \bdd{x}^{\perp} \mathcal{P}_{p-1}(K)$ is rot-free 
	(i.e. $\rot (\grad w|_{K}) \equiv 0$) and so \cite[Lemma 3.8]{ArnFalkWin06} 
	shows that $\grad w|_{K} = \grad v$ for some $v \in \mathcal{P}_p(K)$. 
	Consequently, $w|_{K} \in \mathcal{P}_p(K)$ which means that $w \in 
	\discrete{W}^h$. Since $\bdd{\gamma} \in \discretev{U}^h_{\Gamma}$, we have 
	$\unitvec{t} \cdot \grad w|_{\Gamma_{cs}} = 0$, and so 
	$w|_{\Gamma_{cs}^{(i)}} \in \mathbb{R}$. The function $u = w - 
	w|_{\Gamma_{cs}^{(1)}}$ satisfies $u \in \discrete{W}^h \cap 
	\mathcal{W}_{\Gamma}(\Omega)$ and $\grad u = \bdd{\gamma}$. 
	\ref{pmitc:gradient-cond} now follows.
	
	The operator $\bdd{\Pi}_F$ constructed in the proofs of Proposition 8.4.3, 
	Proposition 8.5.9, and Corollary 8.5.2 of \cite{BoffiBrezziFortin13} satisfies
	\ref{pmitc:v-fortin}. For completeness, we provide a full construction in 
	\cref{lem:v-fortin-bubble}. Condition \ref{pmitc:r-interp-edge} follows by 
	construction \cref{eq:rt-r-operator-edge}.
\end{proof}

We also have the following error estimate:
\begin{corollary}
	\label{thm:rt-family-apriori}
	Let $\discrete{W}^h_{\Gamma}$, $\discretev{V}^h_{\Gamma}$, 
	$\discretev{U}^h_{\Gamma}$, $\discrete{Q}^h_{\Gamma}$, and $\bdd{R}$ be 
	chosen as in \cref{thm:rt-family-satisfies-conds}. Let $(w, \bdd{\theta}) \in 
	H^1_{\Gamma}(\Omega) \times \bdd{\Theta}_{\Gamma}(\Omega)$ denote the 
	solution to \cref{eq:mr-variational-form} with $\bdd{\gamma} = \lambda t^{-2} 
	\bdd{\Xi}(w, \bdd{\theta})$ and let $(w_h, \bdd{\theta}_h) \in 
	\discrete{W}^h_{\Gamma} \times \discretev{V}^h_{\Gamma}$ denote the finite 
	element solution \cref{eq:mr-variational-form-fem} with $\bdd{\gamma}_h = 
	\lambda t^{-2} \bdd{\Xi}_{\bdd{R}}(w_h, \bdd{\theta}_h)$. Then, there holds
	\begin{align}
		\begin{aligned}
			\label{eq:apriori-rt}
			&\| w - w_h \|_1 + \| \bdd{\theta} - \bdd{\theta}_h \|_1 + t \| 
			\bdd{\gamma} - \bdd{\gamma}_h \| \\
			&\ \leq C \left( \inf_{v \in \discrete{W}^h_{\Gamma}} \| w - v \|_1 
			+ \inf_{ \bdd{\psi} \in \discretev{V}^h_{\Gamma}} \|\bdd{\theta} - 
			\bdd{\psi}\|_1 
			+ t \inf_{ \bdd{\eta} \in \discretev{U}^h_{\Gamma}}  \|\bdd{\gamma} - 
			\bdd{\eta}\| 
			+  \inf_{q \in \discrete{Q}^h_{\Gamma}} \|\rot \bdd{\theta} - q\| 
			\right. \\ 
			&\qquad \qquad \left. + \inf_{ \bdd{\eta} \in 
			\discretev{U}^h_{\Gamma}} \| \bdd{\theta} - \bdd{\eta}\|   
			+ h  \inf_{ \bdd{\eta} \in \discretev{U}^h_{\Gamma}}  \|\bdd{\gamma} 
			- \bdd{\eta}\| + h \inf_{ \bdd{\rho} \in 
			\bm{\mathcal{P}}_{p-2}(\mathcal{T}_h) } \|\bdd{\gamma} - \bdd{\rho} 
			\|  \right),
		\end{aligned}
	\end{align}
	where
	\begin{align}
		\bm{\mathcal{P}}_{p-2}(\mathcal{T}_h) := \{ \bdd{\rho} \in 
		\bdd{L}^2(\Omega) : \bdd{\rho}|_{K} \in \mathcal{P}_{p-2}(K)^2 \ \forall 
		K \in \mathcal{T}_h \}.
	\end{align}
	Additionally, if $w \in H^{p+1}(\Omega)$, $\bdd{\theta} \in 
	\bdd{H}^{p+1}(\Omega)$, and $\bdd{\gamma} \in \bdd{H}^{p}(\Omega)$, then 
	there holds
	\begin{multline}
		\label{eq:apriori-rt-smooth}
		\| w - w_h \|_1 + \| \bdd{\theta} - \bdd{\theta}_h \|_1 + t \| 
		\bdd{\gamma} - \bdd{\gamma}_h \| \\ \leq C h^{p} \left( \|w\|_{p+1} + 
		\|\bdd{\theta}\|_{p+1} + \|\bdd{\gamma}\|_{p-1} + t\|\bdd{\gamma}\|_{p} 
		\right),
	\end{multline}
	where the constants $C$ in \cref{eq:apriori-rt,eq:apriori-rt-smooth} are 
	independent of $h$.
\end{corollary}
\begin{proof}
	Thanks to \cref{cor:apriori-pmitc} and \cref{eq:reduction-rot-error}, we need 
	only consider the final term in \cref{eq:apriori-take-0}. To this end, let 
	$\bdd{P}_{p-2} : \bdd{L}^2(\Omega) \to \bm{\mathcal{P}}(\mathcal{T}_h)$ be 
	the following projection:	
	\begin{align*}
		(\bdd{P}_{p-2} \bdd{\eta}, \bdd{\rho} ) = (\bdd{\eta}, \bdd{\rho}) \qquad 
		\forall \bdd{\rho} \in \bm{\mathcal{P}}_{p-2}(\mathcal{T}_h), \ \forall 
		\bdd{\eta} \in \bdd{L}^2(\Omega).
	\end{align*}
	Then, thanks to \cref{eq:rt-r-operator-interior,proof:eq:r-approx-v}, there 
	holds	
	\begin{align*}
		\sup_{ \bdd{\psi} \in \discretev{V}^h_{\Gamma} } \frac{(\bdd{P} 
			\bdd{\gamma}, \bdd{\psi} - \bdd{R} \bdd{\psi} ) }{ \| \bdd{\psi} \|_1 
			} 
		&= \sup_{ \bdd{\psi} \in \discretev{V}^h_{\Gamma} } \frac{(\bdd{P} 
			\bdd{\gamma} - \bdd{P}_{p-2} \bdd{\gamma}, \bdd{\psi} - \bdd{R} 
			\bdd{\psi} ) }{ \| \bdd{\psi} \|_1 } \\
		&\leq C_{\bdd{R}} h \| \bdd{P} \bdd{\gamma} - \bdd{P}_{p-2} \bdd{\gamma} 
		\|  \\
		&\leq C_{\bdd{R}} h \left( \inf_{ \bdd{\eta} \in 
		\discretev{U}^h_{\Gamma}}  \|\bdd{\gamma} - \bdd{\eta}\| + \inf_{ 
		\bdd{\rho} \in \bm{\mathcal{P}}_{p-2}(\mathcal{T}_h) } \|\bdd{\gamma} - 
		\bdd{\rho} \| \right),
	\end{align*}
	which completes the proof of \cref{eq:apriori-rt}. Inequality 
	\cref{eq:apriori-rt-smooth}  then follows from standard approximation results.
\end{proof}

\subsection{Brezzi-Douglas-Marini MITC family}
\label{sec:bdm-family}

Another family of triangular elements arises from choosing $\discretev{U}^h$ to 
be the (rotated) Brezzi-Douglas-Marini (BDM) element defined for $p \geq 2$ by
\begin{align}
	\label{eq:bdm-family-w}
	\discrete{W}^h &:= \{ v \in C(\Omega) : v|_{K} \in \mathcal{P}_{p+1}(K) \ 
	\forall K \in \mathcal{T}_h \}, \\
	\label{eq:bdm-family-u}
	\discretev{U}^h &:= \{ \bdd{\gamma} \in \hrot : \bdd{\gamma}|_{K} \in 
	\mathcal{P}_{p}(K)^2 \ \forall K \in \mathcal{T}_h \},
\end{align}
and define $\bdd{R} : H^1(\Omega) + \discretev{U}^h \to \discretev{U}^h$ to be 
the canonical interpolant (see e.g. \cite{ArnFalkWin06}):
\begin{subequations}
	\label{eq:bdm-r-operator}
	\begin{alignat}{2}
		\int_{e} \unitvec{t} \cdot (\bdd{R} \bdd{\gamma} - \bdd{\gamma}) r \d{s} 
		&= 0 \qquad & &\forall r \in \mathcal{P}_{p}(e), \ \forall e \in 
		\mathcal{E}_h, \\
		\int_{K} (\bdd{R} \bdd{\gamma} - \bdd{\gamma}) \cdot \bdd{\eta} 
		\d{\bdd{x}} &= 0 \qquad & &\forall \bdd{\eta} \in \mathcal{P}_{p-2}(K)^2 
		+ \bdd{x}^{\perp} \mathcal{P}_{p-2}(K), \ \forall K \in \mathcal{T}_h.
	\end{alignat}
\end{subequations}
The rotation space $\discretev{V}^h$ and auxiliary space $\discrete{Q}^h$ are 
chosen as in \cref{eq:rt-family-v} and \cref{eq:rt-family-q}. All of the results 
for the Raviart-Thomas family carry over to this family again by collecting known 
results:
\begin{theorem}
	\label{thm:bdm-family-satisfies-conds}
	Let $\discrete{W}^h_{\Gamma}$, $\discretev{V}^h_{\Gamma}$, 
	$\discretev{U}^h_{\Gamma}$, and $\discrete{Q}^h_{\Gamma}$ be chosen as in
	\cref{eq:bdm-family-w}, \cref{eq:rt-family-v}, \cref{eq:bdm-family-u}, and 
	\cref{eq:rt-family-q} for $p \geq 2$ and let $\bdd{R}$ is defined as in 
	\cref{eq:bdm-r-operator}. Then, conditions 
	\ref{mitc:red-commute}--\ref{mitc:red-bounded} and 
	\ref{harm:gradient-cond}--\ref{harm:r-interp-edge} are satisfied.
	
	Moreover, let $(w, \bdd{\theta}) \in H^1_{\Gamma}(\Omega) \times 
	\bdd{\Theta}_{\Gamma}(\Omega)$ denote the solution to 
	\cref{eq:mr-variational-form} with $\bdd{\gamma} = \lambda t^{-2} 
	\bdd{\Xi}(w, \bdd{\theta})$ and $(w_h, \bdd{\theta}_h) \in 
	\discrete{W}^h_{\Gamma} \times \discretev{V}^h_{\Gamma}$  denote the finite 
	element solution \cref{eq:mr-variational-form-fem} with $\bdd{\gamma}_h = 
	\lambda t^{-2} \bdd{\Xi}_{\bdd{R}}(w_h, \bdd{\theta}_h)$. Then, the a priori 
	estimate \cref{eq:apriori-rt} holds and if $w \in H^{p+1}(\Omega)$, 
	$\bdd{\theta} \in \bdd{H}^{p+1}(\Omega)$, and $\bdd{\gamma} \in 
	\bdd{H}^{p}(\Omega)$, then \cref{eq:apriori-rt-smooth} holds, where the 
	constants are again independent of $h$ but may depend on $p$.
\end{theorem}
Despite the fact that everything needed to verify conditions 
\ref{mitc:red-commute}--\ref{mitc:red-bounded} and 
\ref{harm:gradient-cond}--\ref{harm:r-interp-edge} is readily available in the 
literature, \cref{thm:bdm-family-satisfies-conds} seems to be new. In particular, 
it covers the case of mixed boundary conditions and domains that may contain 
holes.

\begin{remark}
	The corresponding result for the families of quadrilateral elements in 
	\cite{StenbergSuri97}
	also holds. In particular, conditions 
	\ref{mitc:red-commute}--\ref{mitc:red-bounded} and 
	\ref{harm:gradient-cond}--\ref{harm:r-interp-edge} may be verified with the 
	same arguments in the proof of \cref{thm:rt-family-satisfies-conds} using the 
	corresponding results in \cite{StenbergSuri97}.	
\end{remark}

\section{High order schemes} 
\label{sec:high-order}

Thus far, we have considered methods based on polynomials of arbitrary, but 
fixed, order for which it is largely irrelevant whether or not the constants 
$C_{\bdd{R}}$ and $C_F$ depend on the polynomial degree $p$. Here, we focus on 
schemes where the order $p$ is 
increased, and it will be important to track the $p$-dependence of the constants.
In order to reflect this priority, we replace the ``$h$" with ``$p$" in our 
notation.
One benefit of using high order elements is that the reduction operator can be 
reduced to the identity operator on $\discretev{V}^p_{\Gamma}$. Of course, this 
does not necessarily mean that $\bdd{R}$ is the identity on the full space 
$\bdd{\Theta}_{\Gamma}(\Omega)$, but it does mean that
the scheme with the reduction operator \cref{eq:mr-variational-form-fem} and the 
original scheme \cref{eq:mr-variational-form-fem-nored} are equivalent: Find 
$(w_p, \bdd{\theta}_p) \in \discrete{W}^p_{\Gamma} \times 
\discretev{V}^p_{\Gamma}$ such that
\begin{align}
	\label{eq:mr-variational-high-order}
	a(\bdd{\theta}_p, \bdd{\psi}) + \lambda t^{-2} ( \bdd{\Xi}(w_p, 
	\bdd{\theta}_p), \bdd{\Xi}(v, \bdd{\psi}) ) = F(v) + G(\bdd{\psi}) \qquad 
	\forall (v, \bdd{\psi}) \in \discrete{W}^p_{\Gamma} \times 
	\discretev{V}^p_{\Gamma}.
\end{align}

The first observation is 
that if $\bdd{R} \discretev{V}^p_{\Gamma} = \discretev{V}^p_{\Gamma}$, then we 
may choose 
\begin{align}
	\label{eq:image-xi}
	\discretev{U}^p_{\Gamma} = \image \bdd{\Xi}^p := \grad 
	\discrete{W}^p_{\Gamma} +  \discretev{V}^p_{\Gamma} = \{ \bdd{\Xi}(v, 
	\bdd{\psi}) : (v, \bdd{\psi}) \in \discrete{W}^p_{\Gamma} \times 
	\discretev{V}^p_{\Gamma} \}.
\end{align}
Moreover, \cref{remark:consequences-sequence-r} shows that if such an operator 
$\bdd{R}$ exists, then we must have
\begin{align}
	\label{eq:q-nored-choice}
	\discrete{Q}^p_{\Gamma} = \rot \discretev{U}^p_{\Gamma} = \rot 
	\discretev{V}^p_{\Gamma}.
\end{align} 
With these spaces in hand, we now construct an operator $\bdd{R}$ that satisfies 
\ref{mitc:red-commute}. The first requirement is that the following diagram 
commutes
\begin{equation}
	\label{eq:R-L2-project-commute-nor}
	\begin{tikzcd}
		& \bdd{\Theta}_{\Gamma}(\Omega) \arrow{r}{\rot} \arrow{d}{\bdd{R}} & 
		L^2_{\Gamma}(\Omega) \arrow{d}{P} \\
		\discrete{W}^p_{\Gamma} \arrow{r}{\grad} & \image \bdd{\Xi}^p 
		\arrow{r}{\rot} & \rot \discretev{V}^p_{\Gamma},
	\end{tikzcd}
\end{equation}
where $P : L^2_{\Gamma}(\Omega) \to \rot 
\discretev{V}^p_{\Gamma}$ is the $L^2$-projection onto $\rot 
\discretev{V}^p_{\Gamma}$. To this end, let $\discretev{N}^p_{\Gamma}$ denote the
subspace of $\image \bdd{\Xi}^p$ that is $\bdd{L}^2$-orthogonal to $\rot$-free 
vector fields:
\begin{align}
	\discretev{N}^p_{\Gamma} := \{ \bdd{\gamma} \in \image \bdd{\Xi}^p : 
	(\bdd{\gamma}, \bdd{\zeta}) = 0 \ \forall \bdd{\zeta} \in \image \bdd{\Xi}^p 
	\text{ with } \rot \bdd{\zeta} \equiv 0 \}
\end{align}
and let $\bdd{R} : \hrotgamma \to \image \bdd{\Xi}^p$ satisfy the following 
conditions where $\bdd{\gamma} \in \hrotgamma$: 
\begin{subequations}
	\label{eq:no-r-r-def}
	\begin{alignat}{2}
		\label{eq:no-r-r-def-1}
		(\bdd{R} \bdd{\gamma}, \bdd{\eta} ) &= (\bdd{\gamma}, \bdd{\eta}) \qquad 
		& &\forall \bdd{\eta} \in \discretev{N}^p_{\Gamma}, \\
		\label{eq:no-r-r-def-2}
		(\rot \bdd{R} \bdd{\gamma}, q) &= (\rot \bdd{\gamma}, q) \qquad & 
		&\forall q \in \rot \discretev{V}^p_{\Gamma}.
	\end{alignat}
\end{subequations}
The following lemma shows that \cref{eq:no-r-r-def} defines a projection operator 
$\bdd{R}$ satisfying the commuting diagram property 
\cref{eq:R-L2-project-commute-nor}:
\begin{lemma}
	Conditions \cref{eq:no-r-r-def} define a projection operator $\bdd{R} : 
	\hrotgamma \to \image \bdd{\Xi}^p$ such that the diagram 
	\cref{eq:R-L2-project-commute-nor} commutes. 
\end{lemma}
\begin{proof}
	Since $\image \bdd{\Xi}^p = \discretev{N}^p_{\Gamma} \oplus 
	(\discretev{N}^p_{\Gamma})^{\perp}$, where 
	$(\discretev{N}^p_{\Gamma})^{\perp} := \{ \bdd{\eta} \in \image \bdd{\Xi}^p : 
	(\bdd{\eta}, \bdd{\gamma}) = 0 \ \forall \bdd{\gamma} \in 
	\discretev{N}^p_{\Gamma} \}$, it is a straightforward exercise to verify that 
	the values of the linear functionals
	\begin{align*}
		(\bdd{\gamma}, \bdd{\zeta}) \quad \forall \bdd{\zeta} \in 
		\discretev{N}^p_{\Gamma} \quad \text{and} \quad (\rot \bdd{\eta}, q) 
		\quad \forall q \in \rot \discretev{V}^p_{\Gamma}
	\end{align*}
	uniquely determine elements $\bdd{\gamma} \in \discretev{N}^p_{\Gamma}$ and 
	$\bdd{\eta} \in (\discretev{N}^p_{\Gamma})^{\perp}$, respectively. Moreover 
	the conditions in \cref{eq:no-r-r-def} are well-defined for $\bdd{\gamma} \in 
	\hrotgamma$, so $\bdd{R}$ is a well-defined projection. Additionally, 
	\cref{eq:no-r-r-def-2} shows that $\rot \bdd{R} \bdd{\gamma} = P \rot 
	\bdd{\gamma}$ for all $\bdd{\gamma} \in \hrotgamma \supset 
	\bdd{\Theta}_{\Gamma}(\Omega)$, and so the diagram 
	\cref{eq:R-L2-project-commute-nor} commutes.
\end{proof}

\begin{figure}[htb]
	\begin{tcolorbox}
		\begin{description}		
			
			\item[(N1)\label{pmitc-nored:gradient-cond}] If $\bdd{\theta} \in 
			\discretev{V}^p_{\Gamma}$ satisfies $\rot \bdd{\theta} \equiv 0$ and 
			$\circop_{\mathfrak{I}^*} \bdd{\theta} = \vec{0}$, then $\bdd{\theta} 
			= \grad w$ for some $w \in \discrete{W}^p \cap 
			\mathcal{W}_{\Gamma}(\Omega)$.	
			
			\vspace{1em}
			
			\item[(N2)\label{pmitc-nored:v-fortin}] There exists a linear 
			operator 
			$\bdd{\Pi}_{F} : \bdd{\Theta}_{\Gamma}(\Omega) \to 
			\discretev{V}_{\Gamma}^p$ and a constant $C_F \geq 1$ satisfying the 
			following for all $\bdd{\psi} \in 
			\bdd{\Theta}_{\Gamma}(\Omega)$
			\begin{subequations}
				\label{peq-nored:v-fortin}
				\begin{alignat}{2}
					\int_{e} \unitvec{t} \cdot \bdd{\Pi}_F \bdd{\psi} \d{s} &= 
					\int_{e} \unitvec{t} \cdot \bdd{\psi} \d{s} \qquad & &\forall 
					e \in \mathcal{E}_h, \\
					\rot \bdd{\Pi}_F \bdd{\psi} &= P \rot 
					\bdd{\psi}, \qquad & & \\
					\|\bdd{\Pi}_F \bdd{\psi} \|_1 &\leq C_F \|\bdd{\psi}\|_{1},
				\end{alignat}
				where $P : L^2(\Omega) \to \rot \discretev{V}^p_{\Gamma}$ is the 
				$L^2$-projection onto $\rot \discretev{V}^p_{\Gamma}$.
			\end{subequations}   
			
		\end{description}
	\end{tcolorbox}
	\caption{Sufficient conditions for 
		\ref{mitc:red-commute}--\ref{mitc:harmonic-inf-sup-discrete} for the 
		particular choices of $\discretev{U}^p_{\Gamma}$ \cref{eq:image-xi}, 
		$\discrete{Q}^p_{\Gamma}$ \cref{eq:q-nored-choice}, and $\bdd{R}$ 
		\cref{eq:no-r-r-def}, where $\bdd{R}|_{\discretev{V}^p_{\Gamma}} = I$.}
	\label{fig:pmitc-nored-conditions}
\end{figure}

For any choice of $\discrete{W}^p_{\Gamma}$ and $\discretev{V}^p_{\Gamma}$, 
selecting $\discretev{U}^p_{\Gamma}$ as in \cref{eq:image-xi} and
$\discrete{Q}^p_{\Gamma}$ as in \cref{eq:q-nored-choice} means that the operator 
$\bdd{R}$ defined in \cref{eq:no-r-r-def} satisfies conditions 
\ref{mitc:red-commute}, \ref{mitc:red-bounded} with $C_{\bdd{R}} = 1$, and 
\ref{harm:r-interp-edge}. The remaining conditions 
\ref{harm:gradient-cond}--\ref{harm:v-fortin} can be simplified for these 
choices. Since $\circop_{\mathfrak{I}^*} \image \bdd{\Xi}^p = 
\circop_{\mathfrak{I}^*} \discretev{V}^p_{\Gamma}$, we may replace 
$\discretev{U}^p_{\Gamma}$ with $\discretev{V}^p_{\Gamma}$ in 
\ref{harm:gradient-cond}. Similarly, we may replace 
$\discrete{Q}^p_{\Gamma}$ with our particular choice $\rot 
\discretev{V}^p_{\Gamma}$ in \ref{harm:v-fortin}. With these simplified 
conditions in \cref{fig:pmitc-nored-conditions}, 
we arrive at the following error estimate:
\begin{theorem}
	\label{lem:apriori-take-0-nored}
	Let $(w, \bdd{\theta}) \in H^1_{\Gamma}(\Omega) \times 
	\bdd{\Theta}_{\Gamma}(\Omega)$ denote the solution to 
	\cref{eq:mr-variational-form} with $\bdd{\gamma} = \lambda t^{-2} 
	\bdd{\Xi}(w, \bdd{\theta})$ and let $(w_p, \bdd{\theta}_p) \in 
	\discrete{W}^p_{\Gamma} \times \discretev{V}^p_{\Gamma}$ denote the finite 
	element solution \cref{eq:mr-variational-high-order} with $\bdd{\gamma}_p 
	= \lambda t^{-2} \bdd{\Xi}(w_p, \bdd{\theta}_p)$. Then, there holds
	\begin{align}
		\begin{aligned}
			\label{eq:apriori-take-0-nored}
			&\| w - w_p \|_1 + \| \bdd{\theta} - \bdd{\theta}_p \|_1 + t \| 
			\bdd{\gamma} - \bdd{\gamma}_p \| \\
			&\qquad \leq \frac{ C  }{ \beta_{\bdd{\Xi}} + t } \left(  \inf_{v \in 
				\discrete{W}^p_{\Gamma}} \| w - v \|_1 + \inf_{ \bdd{\psi} \in 
				\discretev{V}^p_{\Gamma}} \|\bdd{\theta} - \bdd{\psi}\|_1 + t 
				\inf_{ 
				\bdd{\eta} \in \image \bdd{\Xi}^p }  \|\bdd{\gamma} - 
				\bdd{\eta}\| 
			\right),
		\end{aligned}
	\end{align}
	where
	\begin{align}
		\label{eq:inf-sup-xi}
		\beta_{\bdd{\Xi}} := \inf_{ \substack{\bdd{\eta} \in \image \bdd{\Xi}^p 
				\\ \bdd{\eta} \neq \bdd{0} } } \sup_{ \substack{ (v, \bdd{\psi}) 
				\in 
				\discrete{W}^p_{\Gamma} \times \discretev{V}^p_{\Gamma} \\ (v, 
				\bdd{\psi}) \neq (0, \bdd{0}) } } \frac{ ( \bdd{\eta}, 
				\bdd{\Xi}(v, 
			\bdd{\psi}) )_{\rot} }{ (\|v\|_1 + \|\bdd{\psi}\|_1) 
			\|\bdd{\eta}\|_{\rot} }.
	\end{align}
	Moreover, if the spaces $\discrete{W}^p_{\Gamma}$ and  
	$\discretev{V}^p_{\Gamma}$ satisfy 
	\ref{pmitc-nored:gradient-cond}--\ref{pmitc-nored:v-fortin}, then 
	$\beta_{\bdd{\Xi}} \geq C C_F^{-2}$.
\end{theorem}
The proof of \cref{lem:apriori-take-0-nored} appears in 
\cref{sec:proof-apriori-0}.

\subsection{Macro-elements}
\label{sec:macro}

Given a mesh $\mathcal{T}_h$, let $\mathcal{T}_h^{*}$ denote the Alfeld 
refinement of $\mathcal{T}_h$ obtained by subdividing every triangle of 
$\mathcal{T}_h$ into 3 subtriangles by connecting the barycenter (or other 
interior point) of each element to the vertices of the element. For $p \geq 2$, 
let
\begin{subequations}
	\label{eq:macro-spaces}
	\begin{align}
		\label{eq:macro-cg}
		\discrete{W}^p &:= \{ v \in C(\Omega) : v|_{K} \in \mathcal{P}_{p+1}(K) \ 
		\forall K \in \mathcal{T}_h^{*} \}, \\
		\label{eq:macro-cg-vector}
		\discretev{V}^p &:= \{ \bdd{\psi} \in \bdd{C}(\Omega) : \bdd{\psi}|_{K} 
		\in \mathcal{P}_{p}(K)^2 \ \forall K \in \mathcal{T}_h^{*} \}. 
	\end{align}
\end{subequations}

\noindent The following result, proved in 
\cref{sec:proof-macro-high-order-conds}, shows that this choice satisfies 
conditions \ref{pmitc-nored:gradient-cond}--\ref{pmitc-nored:v-fortin}:
\begin{theorem}
	\label{thm:macro-high-order-conds}
	Let $\discrete{W}^p$ and $\discretev{V}^p$ be chosen as in 
	\cref{eq:macro-cg,eq:macro-cg-vector} with $p \geq 2$. Then, conditions 
	\ref{pmitc-nored:gradient-cond}--\ref{pmitc-nored:v-fortin} are satisfied 
	with $C_F$ independent of $h$ and $p$, and	
	\begin{align}
		\label{eq:macro-high-order-im-xi}
		\image \bdd{\Xi}^p = \{ \bdd{\gamma} \in \hrotgamma : \bdd{\gamma}|_{K} 
		\in 
		\mathcal{P}_{p}(K)^2 \ \forall K \in \mathcal{T}_h^{*} \}.
	\end{align}
\end{theorem}

\noindent Consequently, this family satisfies the quasi-optimal error estimate 
\cref{eq:apriori-take-0-nored}.

\subsection{Standard elements}
\label{sec:standard-family}

For $p \geq 4$, consider the standard continuous finite element spaces:
\begin{subequations}
	\label{eq:standard-spaces}
	\begin{align}
		\label{eq:standard-cg}
		\discrete{W}^p &:= \{ v \in C(\Omega) : v|_{K} \in \mathcal{P}_{p+1}(K) \ 
		\forall K \in \mathcal{T}_h \}, \\
		\label{eq:standard-cg-vector}
		\discretev{V}^p &:= \{ \bdd{\psi} \in \bdd{C}(\Omega) : \bdd{\psi}|_{K} 
		\in \mathcal{P}_{p}(K)^2 \ \forall K \in \mathcal{T}_h \}. 
	\end{align}
\end{subequations}
In previous work \cite[Theorem 8.1]{AinCP24KirchII}, we showed that the 
inf-sup constant $\beta_{\bdd{\Xi}}$ satisfies
\begin{align}
	\label{eq:standard-cg-inf-sup}
	\beta_{\bdd{\Xi}} \geq \beta_0 \xi_{\mathcal{T}},
\end{align}
where $\beta_0 > 0$ depends only on shape regularity, $\Omega$, and the boundary 
conditions, while $\xi_{\mathcal{T}}$ depends only on the topology of the mesh.
In particular, under mild conditions on the mesh topology, $\xi_{\mathcal{T}}$ is 
bounded away from zero. An explicit 
formula for $\xi_{\mathcal{T}}$ and additional details may be found in 
\cite[Appendix A]{AinCP23KirchI}. Moreover, \cite[Theorem 8.3]{AinCP24KirchII} 
shows that
\begin{align*}
	\image \bdd{\Xi}^p = \{ \bdd{\gamma} \in \hrot : \bdd{\gamma}|_{K} \in 
	\mathcal{P}_{p}(K)^2 \ \forall K \in \mathcal{T}_h \text{ and } \rot 
	\bdd{\gamma} \in \rot \discretev{V}^p_{\Gamma} \},
\end{align*}
which means that this family also satisfies \cref{eq:apriori-take-0-nored}.

\begin{remark}
	\label{rem:high-order-convergence}
	For the sequence of meshes obtained by refining the mesh in 
	\cref{fig:box-mesh}, Theorem 4.1 and Lemma 4.6 of \cite{AinCP21LE} show that 
	$\rot 
	\discretev{V}^p_{\Gamma}$ coincides with discontinuous piecewise polynomials 
	of 
	degree $p-1$ and $\xi_{\mathcal{T}}$ in \cref{eq:standard-cg-inf-sup} is  
	bounded away from zero independent of the number of levels of mesh 
	refinements. Consequently, $\image \bdd{\Xi}^p$ coincides with the BDM space 
	\cref{eq:bdm-family-u} of degree $p$ and $\beta_{\bdd{\Xi}}$ is uniformly 
	bounded away from zero. Applying \cref{lem:apriori-take-0-nored} shows that 
	the total error \cref{eq:combined-error} is bounded uniformly in $t$ and 
	decays as $\mathcal{O}(h^p)$.
\end{remark}

\begin{remark}
	In the literature, it is sometimes stated that schemes without a reduction 
	operator will suffer from locking: ``In practice,
	it is well known that if we use `any reasonable conforming approximation of 
	[$H^1_{\Gamma}(\Omega) \times \bdd{\Theta}_{\Gamma}(\Omega) $]', we will get 
	pretty bad answers for small $t$" \cite[p. 599]{BoffiBrezziFortin13}; and 
	``For 
	conventional finite element 
	discretizations one obverses that [$\nabla \discrete{W}_{\Gamma}^h  
	\not\subset
	\discretev{V}^h_{\Gamma}$], thus the Kirchhoff constraint [$\nabla w - 
	\bdd{\theta} = 0$] cannot be satisfied by the discrete solution, the 
	formulation 
	locks" \cite[p. 720]{Pechstein17}, which is repeated in \cite[p. 6]{Sky23}. 
	However, the above discussion shows that the families 
	\cref{eq:macro-spaces,eq:standard-spaces} are locking-free.
\end{remark}

\section{Error analysis}
\label{sec:error-analysis}

The first step in the analysis of the finite element approximation 
\cref{eq:mr-variational-form-fem} is to introduce the discrete shear stress 
$\bdd{\gamma}_h := t^{-2}\bdd{\Xi}_{\bdd{R}}(w_h, \bdd{\theta}_h) \in \image 
\bdd{\Xi}_{\bdd{R}}^h \subseteq \discretev{U}_{\Gamma}$ as an auxiliary variable. 
Then, $w_h$, $\bdd{\theta}_{h}$, and $\bdd{\gamma}_{h}$ satisfy the following 
mixed formulation: Find $(w_h, \bdd{\theta}_h) \in  \discrete{W}^h_{\Gamma} 
\times \discretev{V}_{\Gamma}^h$ and $\bdd{\gamma}_h \in \image 
\bdd{\Xi}_{\bdd{R}}^h$ such that 
\begin{subequations}
	\label{eq:mr-mixed-form-fem}
	\begin{alignat}{2}
		\label{eq:mr-mixed-form-fem-1}
		a(\bdd{\theta}_{h}, \bdd{\psi}) + \langle \bdd{\gamma}_h, 
		\bdd{\Xi}_{\bdd{R}}(v, \bdd{\psi}) \rangle &= F(v) + G(\bdd{\psi}) \qquad 
		& &\forall (v, \bdd{\psi}) \in  \discrete{W}_{\Gamma}^h \times 
		\discretev{V}_{\Gamma}^h, \\
		\label{eq:mr-mixed-form-fem-2}
		\langle \bdd{\eta}, \bdd{\Xi}_{\bdd{R}}(w_h, \bdd{\theta}_{h}) \rangle - 
		t^2 (\bdd{\gamma}_h, \bdd{\eta}) &= 0 \qquad & &\forall \bdd{\eta} \in 
		\image \bdd{\Xi}_{\bdd{R}}^h,
	\end{alignat}
\end{subequations}
where $\langle \cdot, \cdot \rangle$ denotes the extension of the $L^2$ inner 
product to $\dual{\hrotgamma} \times \hrotgamma$. The converse also holds:
\begin{lemma}
	\label{lem:mr-primal-mixed-equiv-fem}
	If  $(w_h, \bdd{\theta}_h) \in  \discrete{W}^h_{\Gamma} \times 
	\discretev{V}^h_{\Gamma}$ satisfies \cref{eq:mr-variational-form-fem}, then 
	$w_h$, $\bdd{\theta}_h$, and $\bdd{\gamma}_h = t^{-2}\bdd{\Xi}_{\bdd{R}}(w_h, 
	\bdd{\theta}_h) \in \image \bdd{\Xi}_{\bdd{R}}^h$ satisfy 
	\cref{eq:mr-mixed-form-fem}. Additionally, if $(w_h, \bdd{\theta}_{h}, 
	\bdd{\gamma}_h) \in  \discrete{W}^h_{\Gamma} \times \discretev{V}^h_{\Gamma} 
	\times \image \bdd{\Xi}_{\bdd{R}}^h$ satisfies \cref{eq:mr-mixed-form-fem}, 
	then $\bdd{\gamma}_h = t^{-2} \bdd{\Xi}_{\bdd{R}}(w_h, \bdd{\theta}_h)$ and 
	$w_h$ and $\bdd{\theta}_h$ satisfy \cref{eq:mr-variational-form-fem}.
\end{lemma}
\begin{proof}
	The first statement follows from the above discussion. Now suppose that 
	$(w_h, \bdd{\theta}_{h}, \bdd{\gamma}_h) \in  \discrete{W}^h_{\Gamma} \times 
	\discretev{V}^h_{\Gamma} \times \image \bdd{\Xi}_{\bdd{R}}^h$ satisfies 
	\cref{eq:mr-mixed-form-fem}. \Cref{eq:mr-mixed-form-fem-2} then gives 
	$\bdd{\gamma}_h = t^{-2} \bdd{\Xi}_{\bdd{R}}(w_h, \bdd{\theta}_h)$, and 
	\cref{eq:mr-variational-form-fem} now follows from 
	\cref{eq:mr-mixed-form-fem-1}.
\end{proof}

\noindent We will now analyze the mixed formulation \cref{eq:mr-mixed-form-fem} 
in order to obtain improved error estimates for the finite element scheme 
\cref{eq:mr-variational-form-fem}.

\subsection{Characterization of \texorpdfstring{$\image 
\bdd{\Xi}_{\bdd{R}}^h$}{$\image \Xi_R^h$}}
\label{sec:image-xi-x}

We first show that a consequence of conditions 
\ref{mitc:red-commute}--\ref{mitc:harmonic-inf-sup-discrete} is that $\image 
\bdd{\Xi}_{\bdd{R}}^h = \discretev{U}^h_{\Gamma}$, and we then construct a right 
inverse of $\bdd{\Xi}_{\bdd{R}} : \discrete{W}^h_{\Gamma} \times 
\discretev{V}^h_{\Gamma} \to \discretev{U}^h_{\Gamma}$:
\begin{lemma}
	\label{lem:invert-xi-discrete}
	Suppose that \ref{mitc:red-commute}--\ref{mitc:harmonic-inf-sup-discrete} 
	hold. For every $\bdd{\gamma} \in \discretev{U}^h_{\Gamma}$, there exists $w 
	\in \discrete{W}^h_{\Gamma}$ and $\bdd{\theta} \in \discretev{V}^h_{\Gamma}$ 
	such that
	\begin{align}
		\label{eq:invert-xi-discrete}
		\bdd{\Xi}_{\bdd{R}}(w, \bdd{\theta}) = \bdd{\gamma} \quad \text{and} 
		\quad \| w \|_1 + \| \bdd{\theta} \|_1 \leq C (\beta_{\rot} 
		\beta_{\harmonic{H}})^{-1} C_{\bdd{R}}^2    \| \bdd{\gamma}\|_{\rot},
	\end{align}
	Consequently, $\image \bdd{\Xi}_{\bdd{R}}^h = \discretev{U}^h_{\Gamma}$ and 
	$\beta_{\bdd{R}}$ in \cref{eq:invert-xi-inf-sup-u} satisfies 
	\cref{eq:inf-sup-equivalences-reduced}.
\end{lemma}
\begin{proof}
	Let $\bdd{\gamma} \in \discretev{U}^h_{\Gamma}$ be given. Thanks to 
	\cref{eq:stokes-inf-sup-discrete}, there exists $\bdd{\theta}_{\rot} \in 
	\discretev{V}^h_{\Gamma}$ such that
	\begin{align*}
		(\rot \bdd{\theta}_{\rot}, q) = -(\rot \bdd{\gamma}, q) \qquad \forall q 
		\in \discrete{Q}^h_{\Gamma} \quad \text{with} \quad \| 
		\bdd{\theta}_{\rot} \|_1 \leq  \beta_{\rot}^{-1} \| \rot \bdd{\gamma} \|.
	\end{align*}
	Then, \cref{eq:R-L2-project-commute} gives
	\begin{align*}
		(\rot \bdd{R} \bdd{\theta}_{\rot}, q) = (P \rot \bdd{\theta}_{\rot}, q) = 
		(\rot \bdd{\theta}, q) = -(\rot \bdd{\gamma}, q) \qquad \forall q \in 
		\discrete{Q}^h_{\Gamma},
	\end{align*}
	and so
	\begin{align*}
		\rot \bdd{R} \bdd{\theta}_{\rot} = -\rot \bdd{\gamma} \quad \text{and} 
		\quad \| \bdd{\theta}_{\rot} \|_1 \leq \beta_{\rot}^{-1} \| \rot 
		\bdd{\gamma} \|.
	\end{align*}
	Additionally, thanks to \ref{mitc:harmonic-inf-sup-discrete}, there exists 
	$\bdd{\theta}_{\harmonic{H}} \in \discretev{V}^h_{\Gamma}$ such that $\rot 
	\bdd{R} \bdd{\theta}_{\harmonic{H}} \equiv 0$,
	\begin{align*}
		(\bdd{R} \bdd{\theta}_{\harmonic{H}}, \harmonic{h}) = -(\bdd{\gamma} + 
		\bdd{R} \bdd{\theta}_{\rot}, \harmonic{h}) \qquad \forall \harmonic{h} 
		\in \harmonic{H}^h_{\Gamma}, \quad \text{and} \quad \| 
		\bdd{\theta}_{\harmonic{H}} \|_{1} \leq \beta_{\harmonic{H}}^{-1} \| 
		\bdd{\gamma} + \bdd{R} \bdd{\theta}_{\rot}\|,
	\end{align*}
	and we define $\bdd{\theta} := \bdd{\theta}_{\rot} + 
	\bdd{\theta}_{\harmonic{H}}$ so that $\|\bdd{\theta}\|_{1} \leq C C_{\bdd{R}} 
	(\beta_{\rot} \beta_{\harmonic{H}} )^{-1} \|\bdd{\gamma}\|_{\rot}$ thanks to 
	\ref{mitc:red-bounded}. 
	
	The vector field $\bdd{\eta} := \bdd{\gamma} + \bdd{R}\bdd{\theta}$ satisfies 
	$\rot \bdd{\eta} \equiv 0$ and $(\bdd{\eta}, \harmonic{h}) = 0$ for all 
	$\harmonic{h} \in \harmonic{H}^h_{\Gamma}$. Thus, there exists $w \in 
	\discrete{W}^h_{\Gamma}$ such that $\grad w = \bdd{\eta}$, and so 
	$\bdd{\Xi}_{\bdd{R}}(w, \bdd{\theta}) = \bdd{\gamma}$. Thanks to 
	Poincar\'{e}'s inequality and \cref{eq:reduction-l2-approx}, we have
	\begin{align*}
		\| w \|_1 \leq C_P \| \grad w \| \leq C_P ( \|\bdd{\gamma}\| + \|\bdd{R} 
		\bdd{\theta}\|) \leq C C_{\bdd{R}}^2 (\beta_{\rot} \beta_{\harmonic{H}} 
		)^{-1} \|\bdd{\gamma}\|_{\rot},
	\end{align*}
	which completes the proof.	 
\end{proof}

\subsection{Stability and well-posedness}
\label{sec:three-field-well-posedness-discrete}

With a concrete characterization of $\image \bdd{\Xi}_{\bdd{R}}^h$ in hand, we 
now consider the stability of \cref{eq:mr-mixed-form-fem}. To this end, define 
the following discrete dual norm:
\begin{align}
	\label{eq:dual-rot-x-norm}
	\| \bdd{\eta} \|_{\dual{(\discretev{U}^h_{\Gamma})}} := \sup_{ \bdd{0} \neq 
		\bdd{\gamma} \in \discretev{U}^h_{\Gamma} } \frac{ \langle \bdd{\eta}, 
		\bdd{\gamma} \rangle }{ \|\bdd{\gamma}\|_{\rot} }  = \sup_{ \bdd{0} \neq 
		\bdd{\gamma} \in \discretev{U}^h_{\Gamma} } \frac{ \langle \bdd{P} 
		\bdd{\eta}, \bdd{\gamma} \rangle }{ \|\bdd{\gamma}\|_{\rot} } 
	\qquad \forall \bdd{\eta} \in \dual{\hrotgamma}.
\end{align}
Then, we see that the following inf-sup constant for the mixed system 
\cref{eq:mr-mixed-form-fem} is bounded below by $\beta_{\bdd{R}}$:
\begin{align}
	\label{eq:inf-sup-discrete-dual}
	\begin{aligned}
		\beta_{\mathrm{saddle}} &:= \inf_{ \substack{\bdd{\eta} \in 
		\discretev{U}^h_{\Gamma} \\ \bdd{\eta} \neq \bdd{0} } } \sup_{ \substack{ 
		(v, \bdd{\psi}) \in \discrete{W}^h_{\Gamma} \times 
		\discretev{V}^h_{\Gamma} \\ (v, \bdd{\psi}) \neq (0, \bdd{0}) } } \frac{ 
		\langle \bdd{\eta}, \bdd{\Xi}_{\bdd{R}}(v, \bdd{\psi}) \rangle }{ 
		(\|v\|_1 + \|\bdd{\psi}\|_1) 
		\|\bdd{\eta}\|_{\dual{(\discretev{U}^h_{\Gamma})}} } \\
		&\geq  \beta_{\bdd{R}} \inf_{ \substack{\bdd{\eta} \in 
				\discretev{U}^h_{\Gamma}  \\ \bdd{\eta} \neq \bdd{0} } } \sup_{ 
			\substack{ \bdd{\gamma} \in \discretev{U}^h_{\Gamma} \\ \bdd{\gamma} 
			\neq 
				\bdd{0} } } \frac{ \langle \bdd{\eta}, \bdd{\gamma} \rangle }{ 
			\|\bdd{\gamma}\|_{\rot} 
			\|\bdd{\eta}\|_{\dual{(\discretev{U}^h_{\Gamma})}} } =  
		\beta_{\bdd{R}}.
	\end{aligned}
\end{align}
We then have the following well-posedness result for the mixed system 
\cref{eq:mr-mixed-form-fem} with general data: 
\begin{theorem}
	\label{thm:mr-mixed-form-gen-fem-bad}
	Suppose that \ref{mitc:red-bounded} holds and let $F \in 
	\dual{(\discrete{W}^h_{\Gamma})}$, $G \in \dual{(\discretev{V}^h_{\Gamma})}$, 
	$\bdd{g}_1 \in \hrotgamma$, and $\bdd{g}_2 \in \bdd{L}^2(\Omega)$ be given. 
	Then, there exists a unique solution to the following mixed problem: Find 
	$w_h \in \discrete{W}^h_{\Gamma}$, $\bdd{\theta}_h \in 
	\discretev{V}^h_{\Gamma}$, and $\bdd{\gamma}_h \in \discretev{U}^h_{\Gamma}$ 
	such that
	\begin{subequations}
		\label{eq:mr-mixed-form-gen-fem-bad}
		\begin{alignat}{2}
			\label{eq:mr-mixed-form-gen-fem-bad-1}
			a(\bdd{\theta}_{h}, \bdd{\psi}) + \langle \bdd{\gamma}_h, 
			\bdd{\Xi}_{\bdd{R}}(v,\bdd{\psi}) \rangle &= F(v) + G(\bdd{\psi}) 
			\qquad & &\forall (v, \bdd{\psi}) \in \discrete{W}^h_{\Gamma} \times 
			\discretev{V}^h_{\Gamma}, \\
			\label{eq:mr-mixed-form-gen-fem-bad-2}
			\langle \bdd{\eta}, \bdd{\Xi}_{\bdd{R}}(w_h, \bdd{\theta}_h) \rangle 
			- t^2(\bdd{\gamma}_h, \bdd{\eta}) &= \langle \bdd{\eta}, \bdd{g}_1 
			\rangle + (\bdd{\eta}, \bdd{g}_2)  \qquad & &\forall \bdd{\eta} \in 
			\discretev{U}^h_{\Gamma}.
		\end{alignat}
	\end{subequations}
	Additionally, the solution satisfies
	\begin{multline}
		\label{eq:mr-mixed-form-gen-fem-bad-stability}
		\|w_h \|_1 + \|\bdd{\theta}_{h}\|_1 + (\beta_{\bdd{R}} + t) \| 
		\bdd{\gamma}_h \|_{\dual{(\discretev{U}^h_{\Gamma})}} + t 
		\|\bdd{\gamma}_h\| \\ \leq C C_{\bdd{R}} \left( 
		\|F\|_{\dual{(\discrete{W}^h_{\Gamma})}} + 
		\|G\|_{\dual{(\discretev{V}^h_{\Gamma})}} + (\beta_{\bdd{R}} + t)^{-1} \| 
		\bdd{g}_1 \|_{\rot} + t^{-1} \|\bdd{g}_2\| \right),
	\end{multline}
	where $\beta_{\bdd{R}}$ is defined in \cref{eq:invert-xi-inf-sup-u},
	\begin{align}
		\label{eq:w-v-dual-norm-def}
		\|F\|_{\dual{(\discrete{W}^h_{\Gamma})}} := \sup_{ 0 \neq v \in 
		\discrete{W}^h_{\Gamma} } \frac{F(v)}{\|v\|_1}, \quad \text{and} \quad 
		\|G\|_{\dual{(\discretev{V}^h_{\Gamma})}} := \sup_{ \bdd{0} \neq 
		\bdd{\psi} \in \discretev{V}^h_{\Gamma} } \frac{G(\bdd{\psi})}{ 
		\|\bdd{\psi}\|_1 }. 
	\end{align}
\end{theorem}
\begin{proof}
	We verify the conditions of \cref{thm:sing-perturb-saddle-stable}. We set $V 
	= \discrete{W}^h_{\Gamma} \times \discretev{V}^h_{\Gamma}$ equipped with the 
	norm $\|(w, \bdd{\theta})\|_{V} := \|v\|_1 + \|\bdd{\theta}\|_1$, $Q = 
	\discretev{U}^h_{\Gamma}$ equipped with the norm $\|\bdd{\eta}\|_{Q} := \| 
	\bdd{\eta} \|_{\dual{(\discretev{U}^h_{\Gamma})}}$, and $W = 
	\discretev{U}^h_{\Gamma}$ equipped with the norm $\|\bdd{\eta}\|_{W} := 
	\|\bdd{\eta}\|$. With an abuse of notation, we also set for $w, v \in 
	\discretev{W}^h_{\Gamma}$, $\bdd{\theta}, \bdd{\psi} \in 
	\discretev{V}^h_{\Gamma}$, and $\bdd{\eta} \in \discretev{U}^h_{\Gamma}$
	\begin{align*}
		a(w, \bdd{\theta}; v, \bdd{\psi}) := a(\bdd{\theta}, 
		\tilde{\bdd{\theta}}) \quad \text{and} \quad b(v, \bdd{\psi}; \bdd{\eta}) 
		:= \langle \bdd{\eta}, \bdd{\Xi}_{\bdd{R}}(v, \bdd{\psi}) \rangle.
	\end{align*}
	
	First note that for all $\bdd{\eta} \in \discretev{U}^h_{\Gamma}$, there holds
	\begin{align*}
		\| \bdd{\eta} \|_W = \sup_{ \bdd{0} \neq \bdd{\gamma} \in 
		\discretev{U}^h_{\Gamma} } \frac{ \langle \bdd{\eta}, \bdd{\gamma} 
		\rangle }{ \|\bdd{\gamma}\|_{\rot} } \leq \sup_{ \bdd{0} \neq 
		\bdd{\gamma} \in \discretev{U}^h_{\Gamma} } \frac{ \| \bdd{\eta} \| \| 
		\bdd{\gamma} \| }{ \|\bdd{\gamma}\|_{\rot} } \leq \|\bdd{\eta}\|_{Q},
	\end{align*}
	and so $C_w = 1$ in \cref{thm:sing-perturb-saddle-stable}. 
	Moreover, for all $\bdd{\theta} \in \discretev{V}^h_{\Gamma}$, there holds 
	$\| \bdd{\theta} \|_1^2 \leq C a(\bdd{\theta}, \bdd{\theta})$ and
	\begin{align*}
		\|w\|_1 \leq C_P \|\grad w\| \leq C_P (\|\grad w - \bdd{R} \bdd{\theta}\| 
		+ \|\bdd{R} \bdd{\theta}\|) \leq C C_{\bdd{R}} \left( \| 
		\bdd{\Xi}_{\bdd{R}}(w, \bdd{\theta})\| + \|\bdd{\theta}\|_1 \right),
	\end{align*}
	and so \cref{eq:a-coercive} then gives $\alpha = C C_{\bdd{R}}^{-2}$ in 
	\cref{thm:sing-perturb-saddle-stable}. The result now follows from 
	\cref{eq:generic-sing-perturb-saddle-stability,eq:inf-sup-discrete-dual}.
\end{proof}

\subsection{A priori error estimate}
\label{sec:apriori-error-red}

At last, we arrive at the following a priori error estimate:
\begin{theorem}
	\label{thm:apriori-error-gen}
	Suppose that \ref{mitc:red-bounded} holds. Let $(w, \bdd{\theta}) \in 
	H^1_{\Gamma}(\Omega) \times \bdd{\Theta}_{\Gamma}(\Omega)$ denote the 
	solution to \cref{eq:mr-variational-form} with $\bdd{\gamma} = t^{-2} 
	\bdd{\Xi}(w, \bdd{\theta})$ and let $(w_h, \bdd{\theta}_h) \in 
	\discrete{W}^h_{\Gamma} \times \discretev{V}^h_{\Gamma}$ denote the finite 
	element solution \cref{eq:mr-variational-form-fem} with $\bdd{\gamma}_h = 
	t^{-2} \bdd{\Xi}_{\bdd{R}}(w_h, \bdd{\theta}_h)$. Then, there holds
	\begin{align}
		\begin{aligned}
			\label{eq:apriori-take-1}
			&\| w - w_h \|_1 + \| \bdd{\theta} - \bdd{\theta}_{h} \|_1 + 
			(\beta_{\bdd{R}} + t) \| \bdd{\gamma} - \bdd{\gamma}_h 
			\|_{\dual{(\discretev{U}^h_{\Gamma})}} + t \| \bdd{\gamma} - 
			\bdd{\gamma}_h\| \\
			&\quad \leq  \frac{ C C_{\bdd{R}} }{ \beta_{\bdd{R}} + t } \bigg( 
			\inf_{v \in \discrete{W}^h_{\Gamma}} \| w - v \|_1 + t \inf_{ 
			\bdd{\eta} \in \discretev{U}^h_{\Gamma}}  \|\bdd{\gamma} - 
			\bdd{\eta}\|  \\ 
			&\qquad \qquad \qquad \qquad  \inf_{ \bdd{\psi} \in 
			\discretev{V}^h_{\Gamma}} \left( \|\bdd{\theta} - \bdd{\psi}\|_1 + 
			\|\bdd{\psi} - \bdd{R} \bdd{\psi}\|_{\rot} \right)
			+\sup_{ \bdd{\psi} \in \discretev{V}^h_{\Gamma} } \frac{( \bdd{P} 
				\bdd{\gamma}, \bdd{\psi} - \bdd{R} \bdd{\psi} ) }{ \| \bdd{\psi} 
				\|_1 
			} \bigg).
		\end{aligned}
	\end{align}
	Moreover, if \ref{mitc:red-commute} and 
	\ref{mitc:stokes-inf-sup-discrete}--\ref{mitc:harmonic-inf-sup-discrete} 
	holds, then $\beta_{\bdd{R}}$ defined in \cref{eq:invert-xi-inf-sup-u} 
	satisfies \cref{eq:inf-sup-equivalences-reduced}.
\end{theorem}
\begin{proof}
	Let $w_I \in \discrete{W}_{\Gamma}^h$, $\bdd{\theta}_I \in 
	\discretev{V}_{\Gamma}^h$, and $\bdd{\gamma}_I \in \discretev{U}_{\Gamma}^h$ 
	be arbitrary. Then, there holds
	\begin{multline*}
		a(\bdd{\theta}_{h} - \bdd{\theta}_I, \bdd{\psi}) + \langle \bdd{\gamma}_h 
		- \bdd{\gamma}_I, \bdd{\Xi}_{\bdd{R}}(v, \bdd{\psi}) \rangle \\
		= a(\bdd{\theta} - \bdd{\theta}_I, \bdd{\psi}) + \langle \bdd{\gamma} - 
		\bdd{\gamma}_I, \bdd{\Xi}(v, \bdd{\psi}) \rangle + \langle 
		\bdd{\gamma}_I, \bdd{\psi} - \bdd{R} \bdd{\psi} \rangle  \qquad \forall 
		(v, \bdd{\psi}) \in \discrete{W}^h_{\Gamma} \times 
		\discretev{V}^h_{\Gamma} 
	\end{multline*}	
	and
	\begin{multline*}
		\langle \bdd{\eta}, \bdd{\Xi}_{\bdd{R}}(w_h - w_I, \bdd{\theta}_{h} - 
		\bdd{\theta}_I) \rangle - t^2( \bdd{\gamma}_h -  \bdd{\gamma}_I, 
		\bdd{\eta} ) \\
		= \langle \bdd{\eta}, \bdd{\Xi}(w - w_I, \bdd{\theta} - \bdd{\theta}_I) 
		\rangle - t^2( \bdd{\gamma} -  \bdd{\gamma}_I, \bdd{\eta} ) + \langle 
		\bdd{\eta},  \bdd{\theta}_I - \bdd{R} \bdd{\theta}_I \rangle \qquad 
		\forall \bdd{\eta} \in \discretev{U}^h_{\Gamma}.
	\end{multline*}
	Applying \cref{eq:mr-mixed-form-gen-fem-bad-stability} then gives
	\begin{multline*}
		\| w_h - w_I \|_1 + \| \bdd{\theta}_{h} - \bdd{\theta}_I \|_1 + 
		(\beta_{\bdd{R}} + t) \| \bdd{\gamma}_h - \bdd{\gamma}_I 
		\|_{\dual{(\discretev{U}^h_{\Gamma})}} + t \| \bdd{\gamma}_h - 
		\bdd{\gamma}_I\| \\
		\leq \frac{ C C_{\bdd{R}} }{ \beta_{\bdd{R}} + t } \bigg( \| w - w_I\|_1 
		+ \|\bdd{\theta} - \bdd{\theta}_I\|_1 + \| \bdd{\gamma} - 
		\bdd{\gamma}_I\|_{\dual{(\discretev{U}^h_{\Gamma})}} + t \|\bdd{\gamma} - 
		\bdd{\gamma}_I\| \\ 
		+ \|\bdd{\theta}_I - \bdd{R} \bdd{\theta}_I\|_{\rot} + \sup_{ \bdd{\psi} 
			\in \discretev{V}^h_{\Gamma} } \frac{\langle \bdd{\gamma}_I, 
			\bdd{\psi} - 
			\bdd{R} \bdd{\psi} \rangle }{ \| \bdd{\psi} \|_1 } \bigg).
	\end{multline*}
	The triangle inequality then gives
	\begin{multline*}
		\| w - w_h \|_1 + \| \bdd{\theta} - \bdd{\theta}_{h} \|_1 + 
		(\beta_{\bdd{R}} + t) \| \bdd{\gamma} - \bdd{\gamma}_h 
		\|_{\dual{(\discretev{U}^h_{\Gamma})}} + t \| \bdd{\gamma} - 
		\bdd{\gamma}_h\| \\
		\leq \frac{ C C_{\bdd{R}} }{ \beta_{\bdd{R}} + t } \bigg( \| w - w_I\|_1 
		+ \|\bdd{\theta} - \bdd{\theta}_I\|_1 + \| \bdd{\gamma} - 
		\bdd{\gamma}_I\|_{\dual{(\discretev{U}^h_{\Gamma})}} + t \|\bdd{\gamma} - 
		\bdd{\gamma}_I\| \\ 
		+  \|\bdd{\theta}_I - \bdd{R} \bdd{\theta}_I\|_{\rot} + \sup_{ \bdd{\psi} 
		\in \discretev{V}^h_{\Gamma} } \frac{\langle \bdd{\gamma}_I, \bdd{\psi} - 
		\bdd{R} \bdd{\psi} \rangle }{ \| \bdd{\psi} \|_1 }  \bigg).
	\end{multline*}
	We now choose $\bdd{\gamma}_I = \bdd{P} \bdd{\gamma}$ so that $\| 
	\bdd{\gamma} - \bdd{P} \bdd{\gamma}\|_{\dual{(\discretev{U}^h_{\Gamma})}} = 
	0$ by 
	\cref{eq:dual-rot-x-norm}, and inequality \cref{eq:apriori-take-1} then 
	follows from taking the infimum over all $(w_I, \bdd{\theta}_I) \in 
	\discrete{W}^h_{\Gamma} \times \discretev{V}^h_{\Gamma}$.
	
	If \ref{mitc:red-commute}--\ref{mitc:harmonic-inf-sup-discrete} holds, then 
	the lower bound for $\beta_{\bdd{R}}$ follows from 
	\cref{lem:invert-xi-discrete}.
\end{proof}

\subsection{Proofs of \cref{lem:apriori-take-0} and 
	\cref{lem:apriori-take-0-nored}}
\label{sec:proof-apriori-0} 

\cref{lem:apriori-take-0} follows from \cref{thm:apriori-error-gen}. Similarly,
\cref{lem:apriori-take-0-nored} follows from  \cref{thm:apriori-error-gen} on 
choosing $\discretev{U}^h_{\Gamma}$ as in \cref{eq:image-xi}, 
$\discrete{Q}^h_{\Gamma}$ as in \cref{eq:q-nored-choice} and $\bdd{R}$ as in 
\cref{eq:no-r-r-def}. \hfill \qedsymbol

\section{Summary}
\label{sec:conclusion}

We have shown that the de Rham complex provides a road map for developing schemes 
for the RM plate in the case of non-simply connected domains with mixed boundary 
conditions. In particular, conditions 
\ref{mitc:red-commute}--\ref{mitc:harmonic-inf-sup-discrete} materialized 
naturally by considering the quasi-optimal approximation properties 
\cref{eq:ker-opt-approx-true-kerxi}. 

Interestingly, adopting the de Rham complex perspective highlights the need for a 
new condition \ref{mitc:harmonic-inf-sup-discrete} in the case of problems on 
non-simply connected domains or with mixed boundary conditions. The absence of 
condition \ref{mitc:harmonic-inf-sup-discrete} in the existing literature may be 
attributed to the assumption that the plate is clamped and simply supported in 
nearly all of the existing literature 
\cite{Ain02,Arn93,Arn07,Arn05,Arnold89,Arn97RM,%
	BatheBrezzi85,BatheBrezzi87,Behrens11,Bosing10,%
	Bramble98,Brezzi89,Brezzi86,Brezzi1991MITC,%
	Calo14,Carstensen11,Chen25,Chinosi06,%
	Veiga13,DiPietro22,Duran92,Falk2000,Falk2008,%
	Hansbo11,Huang24,Huoyuan18,Lovadina05,%
	Pechstein17,Peisker92,Sky23,StenbergSuri97,%
	Ye20,Zhang23}, 
that the plate is simply supported with mixed boundary conditions 
\cite{Amara02RM,Veiga19,Veiga04,Veiga12,Fuhrer23RM,%
	Pitkaranta96,Pitkaranta00,SuriBabSch95},
or that the solution has high regularity \cite{Veiga15,Hughes88}.
Verifying \ref{mitc:harmonic-inf-sup-discrete} theoretically brings new 
challenges and is not amenable to standard techniques. We derive an alternative 
set of sufficient conditions \ref{pmitc:gradient-cond}--\ref{pmitc:r-interp-edge} 
that are more easily verified and, in fact, satisfied by common families of 
locking-free elements for clamped and simply connected plates.

\appendix
\renewcommand{\thesection}{\Alph{section}} 
\makeatletter
\def\@seccntformat#1{\@ifundefined{#1@cntformat}%
	{\csname the#1\endcsname.\hspace{0.5em}}
	{\csname #1@cntformat\endcsname}}
\newcommand\section@cntformat{\appendixname\ \thesection.\hspace{0.5em}}
\makeatother

\section{Proof of \cref{lem:inf-sup-equivalences}}
\label{sec:proof-inf-sup-equivalences}

Note that 
\cref{mot:mitc:red-commute,mot:mitc:stokes-inf-sup-discrete,mot:mitc:harmonic-inf-sup-discrete}
 imply that \ref{mitc:red-commute}, 
\ref{mitc:stokes-inf-sup-discrete}--\ref{mitc:harmonic-inf-sup-discrete} hold. 
Then, one can repeat the proof of \cref{lem:invert-xi-discrete} replacing 
condition \ref{mitc:red-bounded} with \cref{mot:mitc:red-bounded} to show that 
$\beta_{\bdd{R}} \geq C \beta_{\rot} \beta_{\harmonic{H}} M_{\bdd{R}}^{-2}$.  \\

\noindent Now suppose that the inf-sup condition \cref{eq:invert-xi-inf-sup-u} 
holds. 

\noindent \textbf{Step 1: Inverting $\rot : \discretev{U}^h_{\Gamma} \to 
\discrete{Q}^h_{\Gamma}$. } Let $q \in \discrete{Q}^h_{\Gamma}$ be given. We now 
show that there exists $\bdd{\eta} \in \discretev{U}^h_{\Gamma}$ such that
\begin{align}
	\label{eq:proof:invert-rot-u}
	\rot \bdd{\eta} = q \quad \text{and} \quad \|\bdd{\eta}\|_{\rot} \leq C 
	M_{\bdd{R}} \|q\|.
\end{align}
Thanks to \cite[Lemma A.3]{AinCP23KirchI}, there exists $\bdd{\theta} \in 
\bdd{\Theta}_{\Gamma}(\Omega)$ such that
\begin{align*}
	(\rot \bdd{\theta}, r) = (q, r) \qquad \forall q \in \discrete{Q}^h_{\Gamma} 
	\quad \text{and} \quad \| \bdd{\theta} \|_1 \leq  C \| q \|.
\end{align*}
Then, $\bdd{\eta} := \bdd{R} \bdd{\theta}$ satisfies  $\rot \bdd{\eta} = q$ since
\begin{align*}
	(\rot \bdd{R} \bdd{\theta}, r) = (P \rot \bdd{\theta}, r) = (\rot 
	\bdd{\theta}, r) = (q, r) \qquad \forall r \in \discrete{Q}^h_{\Gamma},
\end{align*}
where we used \cref{eq:R-L2-project-commute}. Thus,
\begin{align*}
	\| \bdd{\eta}\|_{\rot} \leq M_{\bdd{R}} \|\bdd{\theta}\|_{1} \leq C 
	M_{\bdd{R}}  \|q\|.
\end{align*}

\noindent \textbf{Step 2: Inverting $\bdd{\Xi}_{\bdd{R}}$. } Since $\image 
\bdd{\Xi}_{\bdd{R}}^h \subset \discretev{U}^h_{\Gamma}$ and the inf-sup condition 
\cref{eq:invert-xi-inf-sup-u} holds, we have that $\image \bdd{\Xi}_{\bdd{R}}^h 
\subset \discretev{U}^h_{\Gamma}$ and for every $\bdd{\gamma} \in 
\discretev{U}^h_{\Gamma}$, there exists $w \in \discrete{W}^h_{\Gamma}$ and 
$\bdd{\theta} \in \discretev{V}^h_{\Gamma}$ such that 
\begin{align}
	\label{eq:proof:invert-xi-discrete-inf-sup-equiv}
	\bdd{\Xi}_{\bdd{R}}(w, \bdd{\theta}) = \bdd{\gamma} \quad \text{and} \quad \| 
	w \|_1 + \| \bdd{\theta} \|_1 \leq \beta_{\bdd{R}}^{-1}   \| 
	\bdd{\gamma}\|_{\rot}.
\end{align}	

\noindent \textbf{Step 3: \cref{mot:mitc:stokes-inf-sup-discrete}. } Let $q \in 
\discrete{Q}^h_{\Gamma}$. Thanks to the Riesz representation theorem, there 
exists $\bdd{\gamma} \in \discretev{U}^h_{\Gamma}$ such that
\begin{align*}
	(\bdd{\gamma}, \bdd{\eta})_{\rot} = (q, \rot \bdd{\eta}) \qquad \forall 
	\bdd{\eta} \in \discretev{U}^h_{\Gamma} \quad \text{and} \quad 
	\|\bdd{\gamma}\|_{\rot} \leq \|q\|.
\end{align*}
Let $\bdd{\eta} \in \discretev{U}^h_{\Gamma}$ satisfying 
\cref{eq:proof:invert-rot-u} be as in Step 1. Then, there holds
\begin{align*}
	\|q\| = \frac{(q, \rot \bdd{\eta})}{\|q\|} = \frac{(\bdd{\gamma}, 
	\bdd{\eta})_{\rot} }{\|q\|} \leq \frac{\|\bdd{\gamma}\|_{\rot} \| 
	\bdd{\eta}\|_{\rot} }{\|q\|} \leq C M_{\bdd{R}} \|\bdd{\gamma}\|_{\rot}. 
\end{align*}
Similarly, let $w \in \discrete{W}^h_{\Gamma}$ and $\bdd{\theta} \in 
\discretev{V}^h_{\Gamma}$ satisfying 
\cref{eq:proof:invert-xi-discrete-inf-sup-equiv} be as in Step 2. Then,
\begin{align*}
	(\rot \bdd{R} \bdd{\theta}, q) = (\rot \bdd{\gamma}, q) = \| 
	\bdd{\gamma}\|_{\rot}^2 \geq C M_{\bdd{R}}^2 \|q\|^2 \quad \text{and} \quad 
	\|\bdd{\theta}\|_1 \leq \beta_{\bdd{R}}^{-1} \|\bdd{\gamma}\|_{\rot} \leq 
	\beta_{\bdd{R}}^{-1}  \|q\|,
\end{align*}
and so \cref{mot:mitc:stokes-inf-sup-discrete} holds with $\beta_{\rot} \geq C 
\beta_{\bdd{R}} M_{\bdd{R}}^{-2}$ as
\begin{align*}
	\sup_{\bdd{0} \neq \bdd{\psi} \in \discretev{V}^h_{\Gamma}} \frac{ (\rot 
	\bdd{\psi}, q) }{ \|\bdd{\psi}\|_1} \geq  \frac{ (\rot \bdd{\theta}, q) }{ 
	\|\bdd{\theta}\|_1} \geq \frac{C \beta_{\bdd{R}}}{M_{\bdd{R}}^2} \|q\|.
\end{align*}

\noindent \textbf{Step 4: \cref{mot:mitc:harmonic-inf-sup-discrete}. } Let 
$\bdd{\gamma} \in \harmonic{H}^h_{\Gamma}$, and let  $w \in 
\discrete{W}^h_{\Gamma}$ and $\bdd{\theta} \in \discretev{V}^h_{\Gamma}$ 
satisfying \cref{eq:proof:invert-xi-discrete-inf-sup-equiv} be as in Step 2. 
Then, $\rot \bdd{R} \bdd{\theta} = \rot \bdd{\Xi}_{\bdd{R}}(w, \bdd{\theta}) = 
\rot \bdd{\gamma} = 0$,
\begin{align*}
	(\bdd{R} \bdd{\theta}, \bdd{\gamma}) = (\bdd{\Xi}_{\bdd{R}}(w, \bdd{\theta}), 
	\bdd{\gamma}) = \|\bdd{\gamma}\|^2, \quad \text{and} \quad \|\bdd{\theta}\|_1 
	\leq \beta_{\bdd{R}}^{-1} \|\bdd{\gamma}\|_{\rot} = \beta_{\bdd{R}}^{-1} 
	\|\bdd{\gamma}\|.
\end{align*}
\cref{mot:mitc:harmonic-inf-sup-discrete} now follows with $\beta_{\harmonic{H}} 
\geq \beta_{\bdd{R}}$.	\hfill \qedsymbol

\section{Proofs of results in \cref{sec:demist-harm}}

\subsection{Proof of \cref{lem:harmonic-complex-cont}}
\label{sec:proof-harmonic-complex-cont}

\textbf{Step 1: $\circop_{\mathfrak{I}^*} : \harmonic{H}_{\Gamma}(\Omega) \to 
\mathbb{R}^{|\mathfrak{I}^*|}$ is surjective. } Let $\vec{\kappa} \in 
\mathbb{R}^{|\mathfrak{I}^*|}$ be given and let $\{ \Gamma_f^{(j)} \}$, $j \in 
\{1, \ldots, N_f\}$ denote the $N_f$ connected components of $\Gamma_f$. Let 
$\vec{\omega} \in \mathbb{R}^{N_f}$ be chosen so that
\begin{align}
	\label{eq:proof:omega-choice}
	\sum_{ \substack{ 1 \leq j \leq N_f \\ \Gamma_f^{(j)} \subset \partial 
	\Omega_i } } \omega_j = \kappa_i  \quad \forall i \in \mathfrak{I}^*, \quad 
	\sum_{j=1}^{N_f} \omega_i = 0, \quad \text{and} \quad |\vec{\omega}| \leq C 
	|\vec{\kappa}|.
\end{align}
Thanks to \cite[Theorem 8.2]{AinCP24KirchII}, there exists $\bdd{\theta} \in 
\bdd{\Theta}_{\Gamma}(\Omega)$ satisfying
\begin{align}
	\label{eq:proof:invert-circulation-cont}
	\rot \bdd{\theta} \equiv 0, \quad  \int_{\Gamma_f^{(j)}} 
	\unitvec{t}\cdot\bdd{\theta} \d{s} = \omega_j \quad \forall j \in \{1,\ldots, 
	N_f\}, \quad \text{and} \quad \|\bdd{\theta}\|_{1} \leq C |\vec{\omega}|.
\end{align}
Note that $\circop_{\mathfrak{I}^*} \bdd{\theta} = \vec{\kappa}$. Let 
$\harmonic{h} \in \harmonic{H}_{\Gamma}(\Omega)$ be given by
\begin{align*}
	(\harmonic{h}, \harmonic{g}) = (\bdd{\theta}, \harmonic{g}) \qquad \forall 
	\harmonic{g} \in \harmonic{H}_{\Gamma}(\Omega).
\end{align*}
Then, $\circop_{\mathfrak{I}^*} \harmonic{h} = \vec{\kappa}$ since $\harmonic{h} 
- \bdd{\theta} \in \grad H^1_{\Gamma}(\Omega)$, and so $\circop_{\mathfrak{I}^*} 
: \harmonic{H}_{\Gamma}(\Omega) \to \mathbb{R}^{|\mathfrak{I}^*|}$ is surjective.

\textbf{Step 2: $\grad \mathfrak{W}_{\Gamma}(\Omega) = \ker 
\circop_{\mathfrak{I}^*}$. } Let $\harmonic{h} \in \harmonic{H}_{\Gamma}(\Omega)$ 
satisfy $\circop_{\mathfrak{I}^*} \harmonic{h} = \vec{0}$. Then, for $i \in 
\mathfrak{I} \setminus \mathfrak{I}^*$,
\begin{align*}
	\circop_i \harmonic{h} = \int_{\Gamma} \unitvec{t} \cdot \harmonic{h} \d{s} - 
	\sum_{j \in \mathfrak{I}^*} \circop_j \harmonic{h} = \int_{\Omega} \rot 
	\harmonic{h} \d{\bdd{x}} = 0,
\end{align*}
and so $\circop_i \harmonic{h} = 0$ for all $i \in \{0,\ldots, H\}$. Thanks to 
\cite[p. 37 Theorem 3.1]{GiraultRaviart86}, there exists $v \in H^1(\Omega)$ such 
that $\grad v = \harmonic{h}$. Since $\harmonic{h} \in \hrotgamma$, 
$v|_{\Gamma_{cs}^{(i)}} \in \mathbb{R}$ and so $w : = v - v|_{\Gamma_{cs}^{(0)}}$ 
satisfies $\grad w = \harmonic{h}$, $w \in \mathcal{W}_{\Gamma}(\Omega)$, and
\begin{align*}
	(\grad w, \grad u) = (\harmonic{h}, \grad u) = 0 \qquad \forall u \in 
	H^1_{\Gamma}(\Omega).
\end{align*}
Thus, $w \in \mathfrak{W}_{\Gamma}(\Omega)$, and so $\grad 
\mathfrak{W}_{\Gamma}(\Omega) = \ker \circop_{\mathfrak{I}^*}$.

\textbf{Step 3: $\dim \harmonic{H}_{\Gamma}(\Omega) = |\mathfrak{I}^*| + N_{cs} - 
	1$. } Suppose that $w \in \mathfrak{W}_{\Gamma}(\Omega)$ and $\grad w = 
	\bdd{0}$. 
Thanks to Poincar\'{e}'s inequality, $\|w\|_1 \leq C_P \|\grad w\| = 0$ and so $w 
\equiv 0$. Consequently, $\ker (\grad : \mathfrak{W}_{\Gamma}(\Omega) \to 
\harmonic{H}_{\Gamma}(\Omega) ) = \{0\}$, and 
the complex \cref{eq:harmonic-complex-cont} is exact. The splitting lemma 
\cite[p. 147]{Hatcher02} then shows that $\harmonic{H}_{\Gamma}(\Omega)$ is 
isomorphic to the direct sum $\mathfrak{W}_{\Gamma}(\Omega) \oplus 
\mathbb{R}^{|\mathfrak{I}^*|}$. A simple consequence of the trace theorem is that 
$\dim \mathfrak{W}_{\Gamma}(\Omega) = N_{cs} - 1$, and so 
\cref{eq:dim-harmonic-forms} follows. \hfill \qedsymbol

\subsection{Auxiliary results for \cref{thm:harm-conds-harm-inf-sup}}

\begin{lemma}
	\label{lem:invert-circ}
	Suppose that \ref{mitc:red-commute} and 
	\ref{harm:v-fortin}--\ref{harm:r-interp-edge} hold. Then, for every 
	$\vec{\kappa} \in \mathbb{R}^{|\mathfrak{I}^*|}$, 
	there exists $\bdd{\theta} \in \discretev{V}^h_{\Gamma}$ satisfying
	\begin{align}
		\label{eq:invert-circ-alt}
		\rot \bdd{R} \bdd{\theta} \equiv 0, \quad \circop_{\mathfrak{I}^*} 
		\bdd{R} \bdd{\theta} = \vec{\kappa}, \quad \text{and} \quad \| 
		\bdd{\theta} \|_1 \leq C C_F  |\vec{\kappa}|.
	\end{align}
\end{lemma}
\begin{proof}
	Let $\vec{\kappa} \in \mathbb{R}^{|\mathfrak{I}^*|}$ be given. Let 
	$\vec{\omega} \in \mathbb{R}^{N_f}$ be chosen as in 
	\cref{eq:proof:omega-choice}. Thanks to \cite[Theorem 8.2]{AinCP24KirchII}, 
	there exists $\bdd{\psi} \in \bdd{\Theta}_{\Gamma}(\Omega)$ satisfying
	\begin{align*}
		\rot \bdd{\psi} \equiv 0, \quad  \int_{\Gamma_f^{(j)}} 
		\unitvec{t}\cdot\bdd{\psi} \d{s} = \omega_j \quad \forall j \in 
		\{1,\ldots, N_f\}, \quad \text{and} \quad \|\bdd{\psi}\|_{1} \leq C 
		|\vec{\kappa}|.
	\end{align*}
	Let $\bdd{\theta} := \bdd{\Pi}_F \bdd{\psi} \in \discretev{V}^h_{\Gamma}$, 
	where $\bdd{\Pi}_F$ is the operator in \ref{harm:v-fortin}, so that
	\begin{align}
		\label{eq:proof:invert-rot-circ-omega-alt}
		P \rot \bdd{\theta} \equiv 0, \quad \int_{\Gamma_f^{(i)}} \unitvec{t} 
		\cdot \bdd{\theta} \d{s} = \omega_j \quad \forall j \in \{1,\ldots, 
		N_f\}, 
		\quad \text{and} \quad \|\bdd{\theta}\|_1 \leq C C_F |\vec{\kappa}|,
	\end{align}
	and in particular, $\circop_{\mathfrak{I}^*} \bdd{\theta} = \vec{\kappa}$. 
	Thanks to \ref{harm:r-interp-edge}, there holds $\circop_{\mathfrak{I}^*} 
	\bdd{R} \bdd{\theta} = \circop_{\mathfrak{I}^*} \bdd{\theta} = \vec{\kappa}$. 
	The commuting diagram property \ref{mitc:red-commute} then gives $\rot 
	\bdd{R} \bdd{\theta} = P \rot \bdd{\theta} \equiv 0$.
\end{proof}

\begin{lemma}
	\label{lem:invert-trace}
	Suppose that \ref{mitc:red-commute}--\ref{mitc:red-bounded} and 
	\ref{harm:gradient-cond}--\ref{harm:r-interp-edge} hold. For every $\{ c_i 
	\}_{i=2}^{N_{cs}} \subset \mathbb{R}$, there exists $w \in \discrete{W}^h 
	\cap \mathcal{W}_{\Gamma}(\Omega)$ and $\bdd{\theta} \in 
	\discretev{V}_{\Gamma}^h$ satisfying
	\begin{subequations}
		\label{eq:invert-trace-alt}
		\begin{alignat}{2}
			\bdd{R} \bdd{\theta} &= \grad w, \qquad & & \\
			w|_{\Gamma_{cs}^{(i)}} &= c_i, \qquad & &2 \leq i \leq N_{cs}, \\
			\|\bdd{\theta}\|_1 + C_{\bdd{R}}^{-1} \|w\|_1 &\leq C C_F 
			\sum_{i=2}^{N_{cs}} |c_i|. \qquad & &
		\end{alignat}
	\end{subequations}	
\end{lemma}
\begin{proof}
	Let $\{ c_i \}_{i=2}^{N_{cs}}$ be given and set $c_1 := 0$. Thanks to 
	\cite[Lemma 8.1]{AinCP24KirchII}, there exists $\tilde{w} \in H^2(\Omega)$ 
	satisfying
	\begin{align}
		\label{eq:proof:h2-interp-trace}
		\tilde{w}|_{\Gamma_{cs}^{(i)}} = c_i, \quad i \in \{1,\ldots, N_{cs}\}, 
		\quad \partial_n \tilde{w}|_{\Gamma_{cs}} = 0, \quad \text{and} \quad 
		\|\tilde{w}\|_2 \leq C \sum_{i=2}^{N_{cs}} |c_i|.
	\end{align}
	Let $\bdd{\theta} := \bdd{\Pi}_F \grad \tilde{w} \in 
	\discretev{V}^h_{\Gamma}$, where $\bdd{\Pi}_F$ is the operator in 
	\ref{harm:v-fortin}, so that
	\begin{align*}
		P \rot \bdd{\theta} \equiv 0, \quad \int_{e} \unitvec{t} \cdot 
		(\bdd{\theta} - 
		\grad \tilde{w}) \d{s} = 0 \quad \forall e \in \mathcal{E}_h, \quad 
		\text{and} \quad \|\bdd{\theta}\|_1 \leq C C_{F}  \sum_{i=2}^{N_{cs}} 
		|c_i|.
	\end{align*}
	Thanks to \ref{harm:r-interp-edge} and \ref{mitc:red-commute}, there holds
	\begin{align*}
		\rot \bdd{R} \bdd{\theta} \equiv 0 \quad \text{and} \quad \int_{e} 
		\unitvec{t} \cdot \bdd{R} \bdd{\theta} \d{s} = \int_{e} \unitvec{t} \cdot 
		\grad \tilde{w} \d{s} \qquad \forall e \in \mathcal{E}_h.
	\end{align*}
	Hence, $\circop_{\mathfrak{I}^*} \bdd{R} \bdd{\theta} = \vec{0}$. Thanks to 
	\ref{harm:gradient-cond}, $\bdd{R} \bdd{\theta} = \grad w$ for some $w \in 
	\discrete{W}^h \cap \mathcal{W}_{\Gamma}(\Omega)$.
	
	Now let $a \in \mathcal{V}_h$ lie on $\Gamma_{cs}^{(1)}$ and let $b \in 
	\mathcal{V}_h \setminus \{a\}$. Since the mesh forms a connected graph, there 
	exists a set of edges $\mathcal{E}_h^{ab} \subset \mathcal{E}_h$ that form a 
	path from $a$ to $b$. Moreover, 
	\begin{align*}
		w(b) = w(a) + \sum_{e \in \mathcal{E}_h^{ab}} \int_{e}  \unitvec{t} \cdot 
		\grad w \d{s} = \sum_{e \in \mathcal{E}_h^{ab}} \int_{e} \unitvec{t} 
		\cdot \bdd{R} \bdd{\theta} \d{s} = \sum_{e \in \mathcal{E}_h^{ab}}  
		\int_{e} \unitvec{t} \cdot \grad \tilde{w} \d{s},
	\end{align*}
	and so
	\begin{align*}
		w(b) = \sum_{e \in \mathcal{E}_h^{ab}} \int_{e} \unitvec{t} \cdot \grad 
		\tilde{w} \d{s} = \tilde{w}(a) + \sum_{e \in \mathcal{E}_h^{ab}} \int_{e} 
		\unitvec{t} \cdot \grad \tilde{w} \d{s} = \tilde{w}(b).
	\end{align*}
	For $i \in \{2,\ldots, N_{cs}\}$, we choose $b$ to lie on $\Gamma_{cs}^{(i)}$ 
	so that $w|_{\Gamma_{cs}^{(i)}} = w(b) = \tilde{w}(b) = c_{i}$. 
	\cref{eq:invert-trace-alt} now follows from Poincar\'{e}'s inequality 
	\cref{eq:poincare-inequality} and \ref{mitc:red-bounded}: $\|w\|_1 \leq C_P 
	\|\grad w\| = C_P \|\bdd{R} \bdd{\theta}\| \leq C C_P C_{\bdd{R}} 
	\|\bdd{\theta}\|_1$.
\end{proof}

\subsection{Proof of \cref{thm:harm-conds-harm-inf-sup}}
\label{sec:proof-harm-conds-harm-inf-sup}

\textbf{Step 1: \ref{mitc:stokes-inf-sup-discrete}. }
Let $q \in \discrete{Q}^h_{\Gamma} \subset L^2_{\Gamma}(\Omega)$ be given. Thanks 
to \cite[Lemma A.3]{AinCP23KirchI}, there exists $\bdd{\psi} \in 
\bdd{\Theta}_{\Gamma}(\Omega)$ such that $\rot \bdd{\psi} = q$ and 
$\|\bdd{\psi}\|_1 \leq 
C \|q\|$. Choosing $\bdd{\theta} = \bdd{\Pi}_F \bdd{\psi}$ then gives 
$\beta_{\rot} \geq C/C_F$.

\textbf{Step 2: Exactness of  \cref{eq:harmonic-forms-discrete-def}. } Let 
$\harmonic{h} \in \harmonic{H}^h_{\Gamma}$ satisfy $\circop_{\mathfrak{I}^*} 
\harmonic{h} \equiv 0$. Thanks to \ref{harm:gradient-cond}, $\harmonic{h} = \grad 
w$ for some $w \in \discrete{W}^h \cap \mathcal{W}_{\Gamma}(\Omega)$. Since 
$\harmonic{h} \in \harmonic{H}^h_{\Gamma}$, there holds
\begin{align*}
	(\grad w, \grad v) = (\harmonic{h}, \grad v) = 0 \ \forall v \in 
	\discrete{W}^h_{\Gamma},
\end{align*}
and so $w \in \mathfrak{W}^h_{\Gamma}$. 

Now let $\vec{\kappa} \in \mathbb{R}^{|\mathfrak{I}^*|}$ be given. As shown in 
the proof of \cref{lem:invert-circ}, there exists $\bdd{\theta} \in 
\discretev{V}^h_{\Gamma}$ such that $\circop_{\mathfrak{I}^*} = \vec{\kappa}$. 
$\bdd{\gamma} := \bdd{R} \bdd{\theta} \in \discretev{U}^h_{\Gamma}$ then 
satisfies $\circop_{\mathfrak{I}^*} \bdd{\gamma} =  \circop_{\mathfrak{I}^*} = 
\vec{\kappa}$ thanks to \ref{harm:r-interp-edge}. The exactness of 
\cref{eq:harmonic-forms-discrete-def} then follows. 

\textbf{Step 3: \ref{mitc:harmonic-inf-sup-discrete}. } Let $\harmonic{h} \in 
\harmonic{H}^h_{\Gamma}$ be given. Thanks to \cref{lem:invert-circ}, there exists 
$\bdd{\theta}_{\circop} \in \discretev{V}^h_{\Gamma}$ satisfying
\begin{align*}
	\rot \bdd{R} \bdd{\theta}_{\circop} \equiv 0, \quad \circop_{\mathfrak{I}^*} 
	\bdd{R} \bdd{\theta}_{\circop} = \circop_{\mathfrak{I}^*} \harmonic{h}, \quad 
	\text{and} \quad \|\bdd{\theta}_{\circop}\|_1 \leq C C_{F} \|\harmonic{h}\|,
\end{align*}
where we used the trace theorem. Consequently, \ref{harm:gradient-cond} shows 
that $\harmonic{h} - \bdd{R} \bdd{\theta}_{\circop} = \grad v$ for some $v \in 
\discrete{W}^h \cap \mathcal{W}_{\Gamma}(\Omega)$.  Thanks to 
\cref{lem:invert-trace}, there exists $\bdd{\theta}_{\mathrm{trace}} \in 
\discretev{V}^h_{\Gamma}$ and $w \in \discrete{W}^h \cap 
\mathcal{W}_{\Gamma}(\Omega)$ satisfying
\begin{align*}
	\bdd{R} \bdd{\theta}_{\mathrm{trace}} = \grad w, \quad w|_{\Gamma_{cs}} = 
	v|_{\Gamma_{cs}}, \quad \text{and} \quad \|\bdd{\theta}_{\mathrm{trace}} \|_1 
	\leq C C_{F} \|v\|_1,
\end{align*}
where we again used the trace theorem. 

Let $\bdd{\theta} := \bdd{\theta}_{\circop} + \bdd{\theta}_{\mathrm{trace}}$. By 
construction, $\harmonic{h} - \bdd{R} \bdd{\theta} = \grad (v - w) \in 
\discrete{W}^h_{\Gamma}$, and so
\begin{align*}
	(\bdd{R} \bdd{\theta}, \harmonic{h}) = \|\harmonic{h}\|^2 + (\grad (v-z-w), 
	\harmonic{h}) = \|\harmonic{h}\|^2
\end{align*}
and Poincar\'{e}'s inequality and \cref{eq:reduction-l2-approx} gives
\begin{align*}
	\|\bdd{\theta}\|_1 \leq C C_{F} \left( \|\harmonic{h}\| +  \|v\|_1 \right) 
	\leq C C_{F} C_{\bdd{R}} \|\harmonic{h}\|,
\end{align*}
and so $\beta_{\harmonic{H}} \geq C (C_{F} C_{\bdd{R}})^{-1}$.	 \hfill \qedsymbol

\section{Stability of singularly perturbed saddle point problems}

Let $V$, $Q$, and $W$ be Hilbert spaces with norms 
$\|\cdot\|_V$, $\|\cdot\|_{Q}$, and $\|\cdot\|_{W}$, and suppose that $W 
\subseteq Q$ is dense in $Q$ and continuously embedded in $Q$; i.e., there exists 
$C_W \geq 1$ such that $\|q\|_{Q} \leq C_W \|q\|_W$ for all $q \in Q$. Let $a : V 
\times V \to \mathbb{R}$ and $b : V \times Q \to \mathbb{R}$ be bounded bilinear 
forms with $a(\cdot,\cdot)$ positive semidefinite and consider the following 
singularly perturbed saddle point problem:  Find $(u, p) \in V \times W$ such that
\begin{subequations}
	\label{eq:generic-sing-perturb-saddle}
	\begin{alignat}{2}
		\label{eq:generic-sing-perturb-saddle-1}
		a(u, v) + b(v, p) &= F(v) \qquad & &\forall v \in V, \\
		\label{eq:generic-sing-perturb-saddle-2}
		b(u, q) - t^2 (p, q)_W &= G_1(q) + G_2(q) \qquad & &\forall q \in W,
	\end{alignat}
\end{subequations}
where $F \in \dual{V}$, $G_1 \in \dual{Q}$, $G_2 \in \dual{W}$, and $t \in (0, 
1]$.

Let $B: V \to \dual{Q}$ be the linear operator associated with the 
bilinear form $b$: $Bv = b(v, \cdot)$ for all $v \in V$. One consequence of 
\cite[Theorem 4.3.4]{BoffiBrezziFortin13} is that if there 
exist $\alpha > 0$ and $\beta > 0$ satisfying
\begin{align} \label{eq:generic-sing-perturb-coercive-infsup}
	\alpha \| v \|_{V}^2 \leq a(v, v) + \| Bv \|_{\dual{W}}^2 \qquad \forall v 
	\in V \quad \text{and} \quad \inf_{ 0 \neq q \in Q} \sup_{0 \neq v \in V} 
	\frac{b(v, q)}{ \|v\|_V \|q\|_Q } > \beta,
\end{align}
then there exists a unique solution to \cref{eq:generic-sing-perturb-saddle}, and 
the solution satisfies
\begin{align*}
	\|u\|_V + \|p\|_Q + t\|p\|_W \leq C \left( \|F\|_{\dual{V}} + 
	\|G_1\|_{\dual{Q}} + t^{-1} \|G_2\|_{\dual{W}} \right),
\end{align*}
where $C$ is independent of $t$. The following result gives the precise 
dependence of $C$ on on the constants $\alpha$, $\beta$, and $C_W$.

\begin{theorem}
	\label{thm:sing-perturb-saddle-stable}
	Let $V$, $Q$, and $W$ be Hilbert spaces as above and suppose that the bounded 
	bilinear forms 	$a(\cdot,\cdot)$ and $b(\cdot,\cdot)$ satisfy 
	\cref{eq:generic-sing-perturb-coercive-infsup}. Then, for every $F \in 
	\dual{V}$, $G_1 \in \dual{Q}$, $G_2 \in \dual{W}$, and 
	$t \in (0, 1]$, the unique solution to 
	\cref{eq:generic-sing-perturb-saddle} satisfies
	\begin{align}
		\label{eq:generic-sing-perturb-saddle-stability}
		\|u\|_V + \frac{\beta + t}{\|a\| C_W} \|p\|_{Q} + t\|p\|_{W}  \leq C 
		\frac{\|a\| }{\sqrt{\alpha}} \left( \|F\|_{\dual{V}} + \frac{C_W}{\beta + 
			t} \|G_1\|_{\dual{Q}} + \frac{1}{t} \|G_2\|_{\dual{W}} \right),
	\end{align}
	where $C > 0$ is a universal constant and
	\begin{align*}
		\|a\| := \sup_{0 \neq u \in V} \sup_{0 \neq v \in V} \frac{|a(u, 
			v)|}{\|u\|_V \|v\|_V}.
	\end{align*}
\end{theorem}
\begin{proof}
	To show 
	\cref{eq:generic-sing-perturb-saddle-stability}, we follow the arguments in 
	the proof of \cite[Theorem 4.3.4]{BoffiBrezziFortin13} and explicitly track 
	the dependence of terms on the various constants.
	
	Let $(u, p) \in V \times W$ satisfy \cref{eq:generic-sing-perturb-saddle}. 
	Taking $v=u$ in the first equation \cref{eq:generic-sing-perturb-saddle-1} 
	and subtracting the second equation \cref{eq:generic-sing-perturb-saddle-2} 
	with $q = p$, we obtain
	\begin{align*}
		a(u, u) + t^2 \|p\|_{W}^2 = F(v) - G_1(p) - G_2(p).
	\end{align*}
	If $\beta > 0$, then the first equation 
	\cref{eq:generic-sing-perturb-saddle-1} gives
	\begin{align}
		\label{eq:proof:sing-perturb-p-q-bound}
		\beta \|p\|_{Q} \leq \sup_{0 \neq v \in V} \frac{b(v, p)}{\|v\|_V} = 
		\sup_{0 \neq v \in V} \frac{F(v) - a(u, v) }{\|v\|_V} \leq 
		\|F\|_{\dual{V}} + \|a\| \|u\|_{V},
	\end{align}
	and so
	\begin{align*}
		| G_1(p)| \leq 
		\beta^{-1} \|G_1\|_{\dual{Q}} \left( \|F\|_{\dual{V}} + \|a\| \|u\|_{V} 
		\right) \leq \epsilon \|a\|^2 \|u\|_{V}^2 + \frac{1}{2} 
		\|F\|_{\dual{V}}^2 + \frac{2}{\epsilon \beta^2} \|G_1\|_{\dual{Q}}^2,
	\end{align*}
	where $0 < \epsilon < 1$ is to be chosen. Alternatively, we have
	\begin{align*}
		|G_1(p)| \leq \|G_1\|_{\dual{W}} \|p\|_{W} \leq \frac{4}{t^2} 
		\|G_1\|_{\dual{W}} + \frac{t^2}{4} \|p\|_{W} \leq  \frac{4C_W^2}{t^2} 
		\|G_1\|_{\dual{Q}}^2 + \frac{t^2}{4} \|p\|_{W} ,
	\end{align*}
	which then gives the following for $\beta \in [0, 1]$:
	\begin{align*}
		| G_1(p) |  \leq \epsilon \|a\|^2 \|u\|_{V}^2  + \frac{t^2}{4} \|p\|_{W} 
		+ \frac{1}{2} \|F\|_{\dual{V}}^2 + \frac{8 C_W^2}{\epsilon (\beta + 
		t)^2}  \|G_1\|_{\dual{Q}}^2,
	\end{align*}
	where we used
	\begin{align*}
		\min\left\{ \frac{1}{\beta}, \frac{1}{t} \right\} = \frac{1}{\max\{\beta, 
		t\}} \leq \frac{2}{\beta + t}
	\end{align*}
	Thanks to the bound
	\begin{align*}
		| G_2(p)| \leq \|G_2\|_{\dual{W}} \| p \|_W \leq \frac{4}{t^2} 
		\|G_2\|_{\dual{W}}^2 + \frac{t^2}{4} \|p\|_{W}^2,
	\end{align*}
	we obtain
	\begin{align*}
		a(u, u) + \frac{t^2}{2} \|p\|_{W}^2 \leq 2 \epsilon \|a\|^2 \|u\|_{V}^2   
		+ \frac{1}{\epsilon} \|F\|_{\dual{V}}^2 + \frac{8 C_W^2}{\epsilon (\beta 
		+ t)^2}  \|G_1\|_{\dual{Q}}^2 + \frac{4}{t^2} \|G_2\|_{\dual{W}}^2.
	\end{align*}
	The second equation \cref{eq:generic-sing-perturb-saddle-2} means that $p = 
	t^{-2} \mathfrak{R}_{W} ( Bu - G_1 - G_2 )$, where $\mathfrak{R}_W : \dual{W} 
	\to W$ is the Riesz operator, and so
	\begin{align*}
		\| Bu \|_{\dual{W}}^2 &\leq 2 \|Bu - G_1 - G_2 \|_{\dual{W}}^2 + 2\|G_1 + 
		G_2 \|_{\dual{W}}^2 \\
		&\leq 2t^{-2} \|Bu - G_1 - G_2 \|_{\dual{W}}^2  + 4\|G_1\|_{\dual{W}}^2 + 
		4\|G_2\|_{\dual{W}}^2 \\
		&\leq  2 t^2 \|p\|_W^2 + 4C_W^2\|G_1\|_{\dual{Q}}^2 + 
		4\|G_2\|_{\dual{W}}^2.
	\end{align*}
	Collecting results then gives
	\begin{multline*}
		a(u, u) + \| Bu \|_{\dual{W}}^2 + t^2 \|p\|_{W}^2 \\
		\leq 12 \epsilon \|a\|^2 \|u\|_{V}^2 + C \left( \frac{1}{\epsilon} 
		\|F\|_{\dual{V}}^2 + \frac{C_W^2}{\epsilon(\beta + t)^2} 
		\|G_1\|_{\dual{Q}}^2 + \frac{1}{t^2} \|G_2\|_{\dual{W}}^2 \right).  
	\end{multline*}
	Applying \cref{eq:generic-sing-perturb-coercive-infsup} and choosing 
	$\epsilon = 1/(24 \|a\|^2)$, we obtain
	\begin{align*}
		\|u\|_{V} + t \|p\|_{W} \leq C \frac{\|a\|}{\sqrt{\alpha}} \left( 
		\|F\|_{\dual{V}} + \frac{C_W}{\beta + t} \|G_1\|_{\dual{Q}} + \frac{1}{t} 
		\|G_2\|_{\dual{W}}  \right).
	\end{align*}
	Inequality \cref{eq:generic-sing-perturb-saddle-stability} now follows from 
	\cref{eq:proof:sing-perturb-p-q-bound} and that
	\begin{align*}
		\frac{\beta + t}{C_W} \| p \|_{ Q } \leq \beta \| p \|_{Q} +  t \|p\|_{W}.
	\end{align*}
\end{proof}

\section{Fortin operators}
\label{sec:proof-v-fortin}

We begin with an auxiliary interpolant commonly used as the Fortin operator for 
the $\bm{\mathcal{P}}^2 \times \mathcal{P}^0$ Stokes element \cite[Proposition 
8.4.3]{BoffiBrezziFortin13} adapted for mixed boundary conditions.
\begin{lemma}
	\label{lem:v-fortin-p2}
	Let $\discretev{P}_2 := \{ \bdd{\psi} \in \bdd{C}(\Omega) : \bdd{\psi}|_{K} 
	\in \mathcal{P}_{2}(K)^2 \ \forall K \in \mathcal{T}_h \}$.  Then, there 
	exists a linear operator $\bdd{\Pi}_{2} : \bdd{H}^1(\Omega) \to 
	\discretev{P}_2$ and constant $C_2 \geq 1$ independent of $h$ such that for 
	all $\bdd{\psi} \in \bdd{H}^1(\Omega)$:
	\begin{subequations}
		\label{eq:v-fortin-p2}
		\begin{alignat}{2}
			\label{eq:v-fortin-p2-1}
			\int_{e} \bdd{\Pi}_2 \bdd{\psi} \d{s} &= \int_{e} \bdd{\psi} \d{s} 
			\qquad & & \forall e \in 
			\mathcal{E}_h, \\
			\label{eq:v-fortin-p2-2}
			\int_{K} \rot \bdd{\Pi}_2 \bdd{\psi} \d{\bdd{x}} &= \int_{K} \rot 
			\bdd{\psi} \d{\bdd{x}} \qquad & 
			&\forall K \in \mathcal{T}_h, \\
			\label{eq:v-fortin-p2-3}
			\|\bdd{\Pi}_2 \bdd{\psi} \|_1 &\leq C_2 \|\bdd{\psi}\|_{1}, \qquad & 
			& \\
			\label{eq:v-fortin-p2-4}
			\bdd{\psi} \in \bdd{\Theta}_{\Gamma}(\Omega) &\implies \bdd{\Pi}_2 
			\bdd{\psi} \in \discretev{P}_2  \cap \bdd{\Theta}_{\Gamma}(\Omega). 
			\qquad & &
		\end{alignat}
	\end{subequations}   
\end{lemma}
\begin{proof}
	We modify the proof of \cite[Proposition 8.4.3]{BoffiBrezziFortin13} which is 
	restricted to the case that $\Gamma = \Gamma_c$.
	
	Let $\bdd{\psi} \in \bdd{H}^1(\Omega)$ be given. Let $\bdd{\Pi}_1 : 
	\bdd{H}^1(\Omega) \to \discretev{P}_1 := \{ \bdd{\psi} \in \bdd{C}(\Omega) : 
	\bdd{\psi}|_{K} \in \mathcal{P}_{1}(K)^2 \ \forall K \in \mathcal{T}_h \}$ 
	denote the piece-wise linear Scott-Zhang interpolant \cite{Scott90} satisfying
	\begin{alignat*}{2}
		\| \bdd{\Pi}_1 \bdd{\psi} \|_1 &\leq C \|\bdd{\psi}\|_1, \qquad & & \\
		h_K^{-1} \| \bdd{\Pi}_1 \bdd{\psi} - \bdd{\psi} \|_{K} + | \bdd{\Pi}_1 
		\bdd{\psi} - \bdd{\psi} |_{1, K} &\leq C \| \bdd{\psi}\|_{1, K} \qquad & 
		& \forall K \in \mathcal{T}_h, \\
		\bdd{\psi} \in \bdd{\Theta}_{\Gamma}(\Omega) &\implies \bdd{\Pi}_1 
		\bdd{\psi} \in \discretev{P}_1  \cap \bdd{\Theta}_{\Gamma}(\Omega). 
		\qquad & &
	\end{alignat*}
	We then define $\bdd{\Pi}_2 \bdd{\psi}$ by the rule
	\begin{alignat*}{2}
		\bdd{\Pi}_2 \bdd{\psi}(\bdd{a}) &= \bdd{\Pi}_1 \bdd{\psi}(\bdd{a}) \qquad 
		& & \forall \bdd{a} \in \mathcal{V}_h, \\
		\int_{e} \bdd{\Pi}_2 \bdd{\psi} \d{s} &= \int_{e} \bdd{\psi} \d{s} \qquad 
		& &\forall e \in \mathcal{E}_h.
	\end{alignat*}
	Properties \cref{eq:v-fortin-p2-1,eq:v-fortin-p2-4} follow immediately by 
	construction. Moreover, \cref{eq:v-fortin-p2-2} follows from 
	\cref{eq:v-fortin-p2-1} since $(\rot \bdd{\Pi}_2 \bdd{\psi}, 1)_K = 
	(\bdd{\Pi}_2 \bdd{\psi}, \unitvec{t})_{\partial K} = (\bdd{\psi}, 
	\unitvec{t})_{\partial K} = (\rot \bdd{\psi}, 1)_K$. Finally, 
	\cref{eq:v-fortin-p2-3} follows from a standard scaling argument analogous to 
	the proof of \cite[Proposition 8.4.3]{BoffiBrezziFortin13}.
\end{proof}

We now construct a Fortin operator satisfying \ref{harm:v-fortin} for the RT and 
BDM families.
\begin{lemma}
	\label{lem:v-fortin-bubble}
	Let $\discretev{V}^h$ and $\discrete{Q}^h$ be defined as in 
	\cref{eq:rt-family-v} and \cref{eq:rt-family-q} with $p \geq 2$. Then, there 
	exists a linear operator $\bdd{\Pi}_{F} : \bdd{H}^1(\Omega) \to 
	\discretev{V}^h$ and constant $C_F \geq 1$ independent of $h$ such that for 
	all $\bdd{\psi} \in \bdd{H}^1(\Omega)$:
	\begin{subequations}
		\label{peq:v-fortin-bubble}
		\begin{alignat}{2}
			\int_{e} \unitvec{t} \cdot \bdd{\Pi}_F \bdd{\psi} \d{s} &= \int_{e} 
			\unitvec{t} 
			\cdot \bdd{\psi} \d{s} \qquad & &\forall e \in \mathcal{E}_h, \\
			(\rot \bdd{\Pi}_F \bdd{\psi}, q) &= (\rot \bdd{\psi}, q) \qquad & 
			&\forall q \in \discrete{Q}^h, \\
			\bdd{\psi} \in \bdd{\Theta}_{\Gamma}(\Omega) &\implies \bdd{\Pi}_F 
			\bdd{\psi} \in \discretev{V}^h_{\Gamma}, \qquad & &\\
			\|\bdd{\Pi}_F \bdd{\psi} \|_1 &\leq C_F \|\bdd{\psi}\|_{1}.
		\end{alignat}
	\end{subequations}   
\end{lemma}
\begin{proof}
	The proof is identical to the proof of \cite[Proposition 
	8.5.9]{BoffiBrezziFortin13} on taking $\tilde{\Pi}_1$ in 
	\cite{BoffiBrezziFortin13} equal to $\bdd{\Pi}_2$ 
	given by \cref{lem:v-fortin-p2}.
\end{proof}

Finally, we construct a Fortin operator satisfying \ref{pmitc:v-fortin} for the 
macro-element family in \cref{sec:macro}.
\begin{lemma}
	\label{lem:v-fortin-macro}
	Let $\discretev{V}^h$ be defined as in \cref{eq:macro-cg-vector} 
	with $p \geq 2$. Then, there exists a linear operator 
	$\bdd{\Pi}_{F} : \bdd{H}^1(\Omega) \to \discretev{V}^h$ and constant 
	$C_F \geq 1$ independent of $h$ and $p$ such that for 
	all $\bdd{\psi} \in \bdd{H}^1(\Omega)$:
	\begin{subequations}
		\label{eq:v-fortin-macro}
		\begin{alignat}{2}
			\label{eq:v-fortin-macro-1}	
			\int_{e} \unitvec{t} \cdot \bdd{\Pi}_F \bdd{\psi} \d{s} &= \int_{e} 
			\unitvec{t} \cdot \bdd{\psi} \d{s} \qquad & &\forall e \in 
			\mathcal{E}_h, \\
			\label{eq:v-fortin-macro-2}
			(\rot \bdd{\Pi}_F \bdd{\psi}, q) &= (\rot \bdd{\psi}, q) \qquad & 
			&\forall q \in \rot \discretev{V}^h, \\
			\label{eq:v-fortin-macro-3}
			\|\bdd{\Pi}_F \bdd{\psi} \|_1 &\leq C_F \|\bdd{\psi}\|_{1}, \qquad & 
			& \\
			\label{eq:v-fortin-macro-4}
			\bdd{\psi} \in \bdd{\Theta}_{\Gamma}(\Omega) &\implies \bdd{\Pi}_F 
			\bdd{\psi} \in \discretev{V}^h_{\Gamma}. \qquad & &
		\end{alignat}
	\end{subequations}   
	Moreover, there holds
	\begin{align}
		\label{eq:v-macro-rot}
		\rot \discretev{V}^h_{\Gamma} = \{q \in L_{\Gamma}^2(\Omega) : q|_{K} 
		\in	\mathcal{P}_{p-1}(K) \ \forall K \in \mathcal{T}_h\}.	
	\end{align}
\end{lemma}
\begin{proof}
	For $K \in \mathcal{T}_h$, we define the local spaces
	\begin{align*}
		\discretev{V}_0(K) &:= \{ \bdd{\psi} \in \bdd{C}(K) 
		\cap \bdd{H}^1_0(K) : \bdd{\eta}|_{L} \in \mathcal{P}_{p}(L) \ \forall L 
		\in \mathcal{T}_h^* 
		\text{ with }  L \subset K \}, \\
		\discrete{Q}_0(K) &:= \{ q \in 
		L_0^2(K) : q|_{L} \in \mathcal{P}_{p-1}(L) \ \forall L \in 
		\mathcal{T}_h^* 
		\text{ with }  L \subset K \}.
	\end{align*}
	
	Let $\bdd{\psi} \in \bdd{H}^1(\Omega)$ be given. Then, the function 
	$\bdd{\phi} := \bdd{\psi} - \bdd{\Pi}_2 \bdd{\psi}$, where $\bdd{\Pi}_2$ is 
	given by \cref{lem:v-fortin-p2}, satisfies $\rot \bdd{\phi}|_{K} \in 
	L^2_0(K)$ for all $K \in \mathcal{T}_h$.
	
	For each $K \in \mathcal{T}_h$, let $q_K \in \discrete{Q}_0(K)$ denote the 
	$L^2(K)$-projection of $\rot \bdd{\phi}$ onto $\discrete{Q}_0(K)$. We then 
	apply \cite{Arn92} if $p=2$, \cite[Theorem 6.4.1]{Qin94} if $p=3$, or 
	\cite[Lemma 
	A.1]{AinCP23KirchI} if $p \geq 4$ to find $\bdd{\psi}_K \in 
	\discretev{V}_0(K)$ satisfying
	\begin{align*}
		\rot \bdd{\psi}_K = q_K \quad \text{and} \quad |\bdd{\psi}_K|_{1, K} 
		\leq C \|q_K\|_K \leq C \|\bdd{\phi}\|_{K}.
	\end{align*}
	Wed define $\bdd{\Pi}_F \bdd{\psi}$ by the rule 
	$\bdd{\Pi}_F \bdd{\psi}|_{K} = \bdd{\Pi}_2 \bdd{\psi}|_{K} + \bdd{\psi}_K$ 
	on $K \in \mathcal{T}_h$. Since $\bdd{\psi}_K|_{\partial K} = \bdd{0}$, 
	\cref{eq:v-fortin-macro-1,eq:v-fortin-macro-4} follow from 
	\cref{lem:v-fortin-p2}. Property \cref{eq:v-fortin-macro-2} follows by 
	construction on noting that $\rot \discretev{V}^h = \{ 
	q \in L^2(\Omega) : q|_{L} \in \mathcal{P}_{p-1}(L) \ \forall L \in 
	\mathcal{T}_h^*\}$:
	\begin{align*}
		\int_{K} \rot \bdd{\Pi}_F \bdd{\psi} \d{\bdd{x}} = \int_{K} \rot 
		\bdd{\Pi}_2 \bdd{\psi} \d{\bdd{x}} = \int_{K} \rot \bdd{\psi} \d{\bdd{x}} 
		\qquad \forall K \in \mathcal{T}_h,
	\end{align*}
	and for all $r \in \discrete{Q}_0(K)$ and $K \in \mathcal{T}_h$, there holds
	\begin{align*}
		\int_{K} (\rot \bdd{\Pi}_F \bdd{\psi}) r \d{\bdd{x}} = \int_{K} (\rot 
		\bdd{\Pi}_2 \bdd{\psi} + q_K) 
		r \d{\bdd{x}}   = \int_{K} (\rot \bdd{\psi}) r \d{\bdd{x}}.
	\end{align*}
	Finally, we have
	\begin{align*}
		| \bdd{\Pi}_F \bdd{\psi}|_{K}^2 \leq 2 (| \bdd{\Pi}_2 \bdd{\psi}|_{1, 
			K}^2 + | \bdd{\psi}_K |_{1, K}^2 ) \leq C(| \bdd{\Pi}_2 
			\bdd{\psi}|_{1, 
			K}^2 + \|\bdd{\psi}\|_{1, K}^2  ),
	\end{align*}
	and so \cref{eq:v-fortin-macro-3} follows from summing over the elements and 
	applying Poincar\'{e}'s inequality.
	
	Now let $q \in \{q \in L_{\Gamma}^2(\Omega) : q|_{K} 
	\in \mathcal{P}_{p-1}(K) \ \forall K \in \mathcal{T}_h\}$. Thanks to 
	\cite[Lemma A.3]{AinCP24KirchII}, there 
	exists $\bdd{\psi} \in \bdd{\Theta}_{\Gamma}(\Omega)$ such that $\rot 
	\bdd{\psi} = 
	q$ and the arguments above show that $\rot \bdd{\Pi}_F \bdd{\psi} = q$. 	 
	\Cref{eq:v-macro-rot} then follows.
\end{proof}

\subsection{Proof of \cref{thm:macro-high-order-conds}}
\label{sec:proof-macro-high-order-conds}

We first show that \ref{pmitc-nored:gradient-cond} holds. Let $\bdd{\theta} \in 
\discretev{V}^h_{\Gamma}$ satisfy $\rot \bdd{\theta} \equiv 0$ and 
$\circop_{\mathfrak{I}^*} \bdd{\theta} = \vec{0}$. Arguing as in the proof of 
\cref{thm:rt-family-satisfies-conds}, we have that $\bdd{\theta} = \grad w$ for 
some $w \in \mathcal{W}_{\Gamma}(\Omega)$. Since $\bdd{\theta}|_{K} \in 
\mathcal{P}_{p-1}(K)^2$, $w|_{K} \in \mathcal{P}_p(K)$ for all $K \in 
\mathcal{T}_h^{*}$, and so $w \in \discrete{W}^h \cap 
\mathcal{W}_{\Gamma}(\Omega)$. 

Thanks to \cref{lem:v-fortin-macro}, condition \ref{pmitc-nored:v-fortin} holds 
with $C_F$ independent of $h$ and $p$.

It remains to show \cref{eq:macro-high-order-im-xi}. Thanks to 
\cref{thm:harm-conds-harm-inf-sup} and condition \ref{mitc:red-commute}, there 
holds
\begin{align*}
	\dim \image \bdd{\Xi}^h &= \dim \grad \discrete{W}^h_{\Gamma} + \dim 
	\harmonic{H}^h_{\Gamma} + \dim \rot \discretev{V}^h_{\Gamma} \\
	&= \dim \grad 
	\discrete{W}^h_{\Gamma} + \dim \harmonic{W}^h_{\Gamma} + 
	|\mathfrak{I}^*| + \dim \rot \discretev{V}^h_{\Gamma}.
\end{align*}
Now let  
\begin{align*}
	\tilde{\discretev{U}}^h_{\Gamma} &:= \{ \bdd{\gamma} \in \hrotgamma : 
	\bdd{\gamma}|_{K} 
	\in 
	\mathcal{P}_{p}(K)^2 \ \forall K \in \mathcal{T}_h^{*} \} \\
	\tilde{\harmonic{H}}^h_{\Gamma} &:= \{ \harmonic{h} \in 
	\tilde{\discretev{U}}^h_{\Gamma} : (\harmonic{h}, \grad w) = 0 \ \forall w 
	\in \discrete{W}^h_{\Gamma} \}.
\end{align*}
Then, \cref{thm:bdm-family-satisfies-conds,thm:harm-conds-harm-inf-sup} show that
\begin{align*}
	\dim \tilde{\discretev{U}}^h_{\Gamma} &= \dim \grad \discrete{W}^h_{\Gamma} + 
	\dim 
	\tilde{\harmonic{H}}^h_{\Gamma} + \dim \rot \discretev{U}^h_{\Gamma} \\
	&= \dim \grad 
	\discrete{W}^h_{\Gamma} + \dim \harmonic{W}^h_{\Gamma} + |\mathfrak{I}^*| + 
	\dim \rot \tilde{\discretev{U}}^h_{\Gamma}.
\end{align*}
Thanks to \cref{remark:consequences-sequence-r,lem:v-fortin-macro}, $\rot 
\discretev{V}^h_{\Gamma} = \rot \tilde{\discretev{U}}^h_{\Gamma}$, and so $\dim 
\image \bdd{\Xi}^h =  \dim \tilde{\discretev{U}}^h_{\Gamma}$. 
\Cref{eq:macro-high-order-im-xi} now follows since $\image \bdd{\Xi}^h 
\subseteq  \tilde{\discretev{U}}^h_{\Gamma}$. \hfill \qedsymbol


\end{document}